 \documentclass[11pt]{article}
 
\usepackage{amssymb,amsmath,amsfonts}
\usepackage{graphicx,color,enumitem}
\usepackage{mathrsfs}
\usepackage{amsthm} 
\usepackage[dvipsnames]{xcolor}
\usepackage{bm}
\usepackage[numbers]{natbib}
\usepackage[]{appendix}
\usepackage{shuffle}
\usepackage{mathdots}

\usepackage[a4paper, total={6in, 9in}]{geometry}

\definecolor{mycolor}{rgb}
{0.796, 0.255, 0.329}

\usepackage[colorlinks,citecolor=blue,urlcolor=blue, linkcolor=red]{hyperref}

\usepackage{tikz}
\tikzstyle{vertex}=[circle, draw, inner sep=2pt, fill=white]

\newcommand{\e}{{\varepsilon}}

\newcommand{\E}{{\mathbb E}}

\renewcommand{\P}{{\mathbb P}}

\newcommand{\C}{{\mathbb C}}
\newcommand{\R}{{\mathbb R}}

\newcommand{\N}{{\mathbb N}}
\newcommand{\X}{{\mathbb X}}

\newcommand{\Acal}{{\mathcal A}}
\newcommand{\Bcal}{{\mathcal B}}

\newcommand{\Dcal}{{\mathcal D}}
\newcommand{\Ecal}{{\mathcal E}}
\newcommand{\Fcal}{{\mathcal F}}
\newcommand{\Gcal}{{\mathcal G}}

\newcommand{\Lcal}{{\mathcal L}}
\newcommand{\Ucal}{{\mathcal U}}

\newcommand{\Ncal}{{\mathcal N}}

\newcommand{\Rcal}{{\mathcal R}}
\newcommand{\Scal}{{\mathcal S}}
\newcommand{\Tcal}{{\mathcal T}}
\newcommand{\Vcal}{{\mathcal V}}
\newcommand{\Wcal}{{\mathcal W}}

\newcommand{\Xx}{\X}
\renewcommand{\a}{\mathbf a}
\renewcommand{\b}{\mathbf b}
\newcommand{\x}{\mathbf x}

\renewcommand{\u}{\mathbf u}
\renewcommand{\v}{\mathbf v}
\newcommand{\Ss}{\Tcal(S)}
\renewcommand{\gg}{h}
\newcommand{\rub}{r}
\newcommand{\pr}[2]{\pi_{#2}({#1})}
\newcommand{\co}[2]{{#1}_{#2}}
\newcommand{\el}[2]{{#1}^{#2}}
\newcommand{\tru}[2]{{#1}^{\leq #2}}

\newcommand{\Span}{\textup{span}}

\newcommand{\fdot}{{\,\cdot\,}}

\DeclareMathOperator{\tr}{Tr}

\newtheorem{theorem}{Theorem}

\newtheorem{scheme}[theorem]{Scheme}

\newtheorem{corollary}[theorem]{Corollary}

\theoremstyle{definition}
\newtheorem{definition}[theorem]{Definition}
\newtheorem{remark}[theorem]{Remark}
\newtheorem{example}[theorem]{Example}

\newtheorem{lemma}[theorem]{Lemma}

\newtheorem{proposition}[theorem]{Proposition}

\numberwithin{equation}{section}
\numberwithin{theorem}{section}

\begin{document}

\title{Signature SDEs from an affine and polynomial perspective}
\date{}
\author{%
  Christa Cuchiero\thanks{Vienna University, Department of Statistics and Operations Research, Data Science @ Uni Vienna, Oskar-Morgenstern-Platz 1, A-1090 Wien, Austria.
    {christa.cuchiero@univie.ac.at}}
  \and 
  Sara Svaluto-Ferro\thanks{University of Verona, Department of Economics,
		Via Cantarane 24, 37129 Verona, Italy. {sara.svalutoferro@univr.it}}
        \and
        Josef Teichmann \thanks{ETH Zurich, Department of Mathematics, R\"amistrasse 101, 8092 Zurich, Switzerland. {josef.teichmann@math.ethz.ch}
        \newline
The first author gratefully acknowledges financial support by the FWF START grant Y 1235.\newline
\textbf{Published version:}
This article has been published in the \emph{Electronic Journal of Probability}. The final published version is available at
\url{https://doi.org/10.1214/26-EJP1553}.
        }}
        
\maketitle
\begin{abstract}
Signature stochastic differential equations (SDEs) constitute a large class of stochastic processes, here driven by Brownian motions, whose characteristics are  linear maps of their own signature, i.e.~of iterated integrals of the process with itself, and therefore allow for a generic path dependence. We show that their prolongation with the corresponding signature is an affine and polynomial process taking values in the set of group-like elements of the extended tensor algebra. By relying on duality theory for affine or polynomial processes, 
we obtain explicit formulas in terms of converging power series for the Fourier-Laplace transform and the expected value of entire functions of the signature process' marginals.

The coefficients of these power series are solutions of Riccati and linear ordinary differential equations (ODEs) with values in the extended tensor algebra, respectively, whose vector fields can be expressed in terms of the characteristics of the corresponding SDEs.  We thus construct a class of stochastic processes that is universal (in a sense specified in the introduction) within It\^o-diffusions with path-dependent characteristics and allows for an explicit characterization of the Fourier-Laplace transform and hence the full law on path space. 
The practical applicability of this affine and polynomial approach is illustrated by several numerical examples.
\end{abstract}

\noindent\textbf{Keywords:} Signature SDEs; affine and polynomial processes;  entire functions on path spaces and groups; power series expansions for Fourier-Laplace transforms;  expected signature\\
\noindent \textbf{MSC (2020)}: 60G20; 60L10; 58K20; 60L70

\tableofcontents

\section{Introduction}

A plethora of stochastic models used in diverse areas, like mathematical finance, biology or physics, stems from the class of affine and polynomial processes. This is not always visible at first sight, because many stochastic models only appear as projections of infinite dimensional processes, and the corresponding affine or polynomial lift is often not the first object of attention.

Recognizing whether a Markovian semimartingale model falls in the class of \emph{finite dimensional affine or polynomial processes,} as systematically introduced in \cite{DFS:03} and \cite{CKT:12, FL:16}, respectively, is usually simple and consists in verifying the properties of its semimartingale characteristics. 
Indeed, in a diffusion setup, the polynomial property holds true if the drift is affine and if the instantaneous covariance matrix is at most quadratic in the state variables. If the latter is also affine, then 
the affine property is additionally satisfied.
Even though this appears to be a rather narrow definition, solely the finite dimensional setting already contains many well known processes, e.g.~L\'evy processes, Ornstein-Uhlenbeck processes, Feller-type diffusions, Wishart processes (see \cite{B:91, CFMT:11}), the Black-Scholes model,  the Fisher-Snedecor process, or the Wright-Fisher diffusion, also called Kimura diffusion (see \cite{K:64}) or Jacobi process, as well as all possible combinations thereof.

Many of those processes also have \emph{infinite dimensional} analogs, which can often be realized as \emph{measure valued processes}, in particular when their state space corresponds to some positive cone or bounded domain, such as probability measures.  We refer to \cite{L:11, CLS:19, CDGS:22}, where it is shown that many well-known examples from the literature, including the Dawson-Watanabe superprocess and the Fleming-Viot process, are of affine or polynomial type.
Another possibility to move to an infinite dimensional setup is to consider \emph{Hilbert space valued processes} which have recently been treated in \cite{STY:20, CKK:22} for the affine case.

Beyond these measure valued and Hilbert space valued processes,  a wide variety of finite dimensional models, in particular non-Markovian ones,  appear as projections of infinite dimensional affine or polynomial processes. This applies in particular to \emph{Markovian lifts} of stochastic Volterra processes (see e.g.~\cite{P:85}) and Hawkes processes (see \cite{H:71}). As a consequence many  examples in the field of rough volatility and rough covariance modeling, such as the rough Heston model (see~\cite{ER:19}) or the rough Wishart processes (see~\cite{CT:19, AJ:22})
can be viewed as infinite dimensional affine or polynomial processes, as shown e.g.~in \cite{AE:19, CT:18}. 

Such Markovian lifts are of course not only restricted to Volterra-type processes, but can generically be applied by \emph{linearizing the models' characteristics}, making a multitude of stochastic processes a projection of sequence-valued affine and polynomial processes. This suggests a certain \emph{universality of the affine and polynomial class} in the space of \emph{all} stochastic processes driven by Brownian motion (or also other semimartingales, see \cite{CPS:24}). This universality comes at the cost of more involved infinite dimensional state spaces, such as $\{(x^n)_{n \in \mathbb{N}} \colon x \in \mathbb{R} \} \subset \mathbb{R}^\mathbb{N}$, than the canonical one $\mathbb{R}_+^m \times \mathbb{R}^n$ considered initially in \cite{DFS:03}.

One purpose of this article is to make certain aspects of this universality mathematically precise. Indeed, we shall ``linearize'' generic classes of stochastic differential equations, called \emph{signature SDEs} (introduced precisely in Section~\ref{sec:SIGSDE}), with the goal to develop an affine and polynomial theory for their prolongations. A signature SDE is an It\^o-diffusion  of the form
\begin{align}\label{eq:SigSDEintro}
dX_{t}={b}(\mathbb{X}_t) dt + \sqrt{{a}(\mathbb{X}_t)}dB_t,  \quad X_0=x, \tag{Sig-SDE}
\end{align}
where ${b}$ and ${a}$ are  \emph{entire maps of group-like elements} in the extended tensor algebra, the state space of the signature process $\mathbb{X}$, rigorously introduced in Section~\ref{sec22}. For the precise definition of entire maps we refer to Section~\ref{sigstatespace} and Definition~\ref{def:entire}. Note that in contrast to the usual setup where only finite or formal sums of signature components are considered\footnote{Notable exceptions are the works \cite{arous1989flots, KP:92, LO:14, DT:22} on 
stochastic Taylor series expansions and the recent paper by \cite{jaber2024path}.}, we deal here with a novel and proper notion of entire maps, being appropriately converging infinite linear combinations of group-like elements.

The aforementioned universality of signature SDEs holds on the level of the characteristics as all continuous path functionals can be approximated by linear and thus entire functions of the time-extended signature on compacts (see e.g. \cite{CGS:22}). By adding the signature of the corresponding  time-extended process  to the state variables, this universality then translates to affine and polynomial processes, as we can show that the
prolonged process is affine and polynomial on the state space of group-like elements.  

To prove this, we first consider in Section~\ref{subsec:affine} affine and polynomial processes on general subsets of the extended tensor algebra. Based on the duality point of view summarized in Section~\ref{sec:simpleaffpoly} for the simple one-dimensional case, we derive under certain integrability conditions the \emph{affine transform formula} as well as the \emph{moment formula}. The affine transform formula, which computes the Fourier-Laplace transform of the state vector, can be expressed either in terms of an \emph{extended tensor algebra valued Riccati ordinary differential equation (ODE)} (see Theorem~\ref{thm:affinetrans2}) or in terms of the corresponding \emph{transport partial differential equation (PDE)} (see Theorem~\ref{thm:affinetrans}). The latter has the advantage that zeros of the Fourier-Laplace transform can be treated without leading to exploding solutions of the Riccati ODEs. In the case of the moment formula, derived via polynomial technology, an \emph{extended tensor algebra valued linear ODE} has to be solved to obtain expressions for the moments and converging infinite sums thereof (see Theorem~\ref{thm:momentpoly}).

In Section~\ref{sec:SIGSDE} we apply this theory to the signature process of signature SDEs, a setting where affine and polynomial processes coincide (see Proposition~\ref{prop2}). 
Our main results (see Theorem~\ref{th:uniaffine} and  Theorem~\ref{eq:expectedsig}) can be summarized as follows:

\begin{theorem}\label{th:intro}
Let $X$ be given by \eqref{eq:SigSDEintro}.
Then, its signature process $\mathbb{X}$ is an affine and polynomial process on the state space of group-like elements, which satisfies under  several technical conditions,
\[
\mathbb{E}[\exp(\langle \mathbf{u}, \mathbb{X}_T  \rangle)]= \exp(\langle {\bm \psi}(T), \mathbb{X}_0 \rangle), 
\qquad \mathbb{E}[\langle \mathbf{u}, \mathbb{X}_T \rangle ]= \langle \mathbf{c}(T) ,  \mathbb{X}_0 \rangle,
\]
where ${\bm \psi}$  and ${\bf c}$ 
are solutions of  the tensor-algebra valued Riccati ODE  and linear ODE 
\begin{equation*}
 {\bm \psi}(t)=\u+\int_0^tR({\bm \psi}(s))ds, \qquad \mathbf{c}(t) =\u+\int_0^t L(\mathbf{c}(s))ds,
\end{equation*}
with $R$ and $L$  specified in \eqref{eq:R} and  \eqref{eq:L}, respectively.
\end{theorem}
The notation $\langle \mathbf{u}, \mathbf{x} \rangle$  is used for entire maps of group-like elements $\mathbf{x}$ and appropriate dual ones $\mathbf{u}$, see \eqref{eq:uxana} and Definition \ref{def:entire}.
Theorem \ref{th:intro} thus gives novel explicit formulas in terms of converging power (or rather linear) series for both, the Fourier-Laplace transform and the expected value of entire functions of the signature. 
This implies in particular that we can compute the \emph{expected signature} for all (appropriately integrable) signature SDEs by solving an infinite dimensional linear ODE.
{By explicit formulas we here mean the  (theoretical infinite dimensional) solution of these ODEs for which we also provide some
numerical schemes in Section \ref{sec:numerics}.}
We   express the vector fields $R$ and $L$ of the Riccati ODEs and the linear ODEs in terms of the characteristics of the corresponding SDEs (see \eqref{eq:R} and \eqref{eq:L}), yielding similar structures as in finite dimensions. Indeed we get quadratic expressions for the Riccati equations and linear ones  for the linear ODEs. In certain specific cases of signature SDEs, including classical polynomial processes, the results lead to a representation of the expected \emph{truncated} signature involving only a finite dimensional linear ODE. The corresponding solution can thus be computed via a finite dimensional matrix exponential (see Example~\ref{ex4}). 

As a special case we consider in Section~\ref{sec5}  one-dimensional signature SDEs which translate -- using a reparametrization of the signature  --  to \emph{classical SDEs with real-analytic characteristics}. The prolongation of these processes are then simply sequences of monomials that are again affine and polynomial processes whose corresponding deterministic dual processes are sequence-valued Riccati and linear ODEs, respectively. In this setup we can give concrete specifications under which the required integrability conditions for the respective transform formulas hold true.

Both, in the one-dimensional and the general setup, our approach efficiently calculates time-dependent expansion coefficients of solutions of the Kolmogorov equations, i.e., certain parabolic partial differential equations (or their Cole-Hopf transforms), where expansion is meant with respect to state variables, here the signature components. {In the one-dimensional case this is nothing else than a power-series expansion in the space variable for the solution of the Kolmogorov equation, whence implying its real analyticity in space. This is  reminiscent of Cauchy-Kovalevskaya theory, where, however, analyticity both in time and space is obtained, and which in the current framework is only applicable when the diffusion coefficient is identically zero. Even for the heat equation there are simple counterexamples showing that analyticity in time  (at 0) of the solution of the Kolmogorov equation can fail. Since we only deal with analyticity in space the class of SDEs for which this holds is much larger. We work this out in detail in the one-dimensional setting for SDEs with real analytic characteristics, see Section \ref{sec:concrete}, where we provide in particular existence and uniqueness results for the infinite dimensional linear ODEs and the  Riccati ODEs. Even though we state the results for simplicity only in the one-dimensional case, the results easily translate to multivariate settings, in particular to signature SDEs where the coefficients only depend on finitely many signature components. 
Apart from several sufficient conditions, e.g., uniform ellipticity (see Section \ref{sec:elliptic}),  that guarantee  the full validity of the moment and affine transform formulas (locally), we also provide an SDE with entire characteristics as counterexample, where real-analyticity of the Kolmogorov equation in space does not hold at degeneracy points, see Section \ref{ex:counterexample}.
}

On a different note, observe that we can characterize laws on path spaces by means of the signature \emph{without} exponential moment conditions (see~\cite{CO:22} and the references therein for related considerations with normalized signature), since we can work with Fourier-Laplace transforms.
We use the fact that a time-extended path $t \mapsto \widehat{X}_t(\omega)$ for $t \in [0,T]$
is uniquely characterized by its signature
at time $T$, denoted by $\widehat{\mathbb{X}}_T(\omega)$. The set
  \[
 \left\lbrace \mathbb{E}[\exp(\text{i} u \langle e_I, \widehat{\mathbb{X}}_T \rangle) ]\colon I \in \lbrace 0,\ldots,d \rbrace^n, \, n \in \mathbb{N}_0, u \in \mathbb{R} \right\rbrace
 \]
therefore characterizes the full law of the stochastic process $X$ on path space and is determined via the Riccati ODEs in the case of signature SDEs. For the notation with basis elements $e_I \in (\mathbb{R}^d)^{\otimes n}$ we refer to Section \ref{sec22}.
Note that this characteristic function differs from the notion used in \cite{Chevyrev16} or \cite{lou2023pcf}.

To illustrate the practical applicability of our affine and polynomial approach, e.g. in view of computing such path characteristic functions, we solve in Section~\ref{sec:Riccatinum} the infinite dimensional Riccati ODEs and transport PDEs to compute the Laplace-transform of a geometric Brownian motion as well as the Laplace transform of the fourth power of a Brownian motion. These cases are particularly challenging from a numerical point of view since exponential moments do not exist.
We also implement the linear ODEs to obtain the moment generating function of a Jacobi diffusion (see Section~\ref{ex1}).

Our theoretical and numerical results have important implications on several modern models that are especially relevant in machine learning contexts. This concerns in particular neural SDEs (with entire activation functions), as considered e.g.~in~\cite{GSSSZ:20, CKT:20, KFLL:21, SLG:22} and signature-based models in the spirit of \cite{PSS:20, LNP:20, NSSBWL:21, SLLHDL:21, CGS:22, CPS:22, CGMS:23, abi2025signature}. Both types of models can be cast in the current framework of signature SDEs and thus the Fourier-Laplace transform and the moments (as well as the expected signature) can be computed by affine and polynomial technology.
{
 We refer in particular to \cite{abi2025signature}, where numerical schemes for solving the truncated Riccati ODEs stemming from specific signature models in finance have  been  implemented.}


In this respect note that our expressions for the Fourier-Laplace transform and the expected signature in terms of power (or rather linear) series in the signature process are related to several recent works on signature cumulants and expected signature, in particular \cite{FHT:22, FGR:22, FG:22, FM:21, BDMN:21, RFC:22, BO:20, AGR:20}. For earlier work on expected signatures  in particular for the Brownian case and classical diffusion processes  we refer to \cite{FW:03, N:12,LN:15} and \cite{FS:17} for L\'evy processes.

The remainder of the article is organized as follows: in Section~\ref{sec:preliminary} we introduce affine and polynomial processes in a one-dimensional diffusion setup as well as the signature process of continuous semimartingales. Section~\ref{subsec:affine} is dedicated to affine and polynomial processes on general subsets of the extended tensor algebra, while in Section~\ref{sec4} we apply this theory to signature SDEs. 
We conclude the paper with one-dimensional signature SDEs, corresponding to classical SDEs with  real-analytic characteristics (see Section~\ref{sec5}), as well as  several examples and numerical illustrations thereof (see Section~\ref{sec:numerics}).

\section{Preliminaries}\label{sec:preliminary}

This section introduces affine and polynomial processes in the simplest possible setup, i.e.~the \emph{one}-dimensional setting. We also recall the fundamental notion of  signature in the setting of continuous semimartingales.

\subsection{Affine and polynomial processes from a duality point of view} \label{sec:simpleaffpoly}
Consider on a filtered probability space $(\Omega, \mathcal{F}, (\mathcal{F}_t)_{t\in[0,T]}, \mathbb{P})$, an It\^o diffusion process $(X_t)_{t\in[0,T]}$ of the form
\begin{align}\label{eq:poly}
dX_t=b(X_t)dt + \sqrt{a(X_t)} dB_t, \qquad X_0=x\in \R
\end{align}
with $a: \mathbb{R} \to \mathbb{R}_+$ and $b:\mathbb{R} \to \mathbb{R}$ continuous functions and $B$ a standard Brownian motion. 

\begin{definition}\label{def:polaff}
A weak solution $X$ of \eqref{eq:poly} is called \emph{polynomial process} if $b$ and $a$ admit the representations $b(x)=b_0+b_1 x$  and $a(x)=a_0+a_1 x+ a_2x^2$, respectively, for some constants $b_0,b_1,a_0,a_1$, and $a_2$.
If additionally $a_2=0$, then the process is called affine.
\end{definition}
Note that in this diffusion setting all affine processes are polynomial. In general this only holds true under auxiliary conditions on the moments of the  jump measure.
\begin{remark}
Alternatively, such processes can be equivalently defined via the martingale problem approach (as we will do later on in Section~\ref{subsec:affine}). {In particular, setting 
 \begin{equation}\label{eq:polygen}
\mathcal{A}f(x)=f'(x)b(x)+ \frac{1}{2} f''(x) a(x),
\end{equation}
it holds that 
the process
\[
N_t^f:=f(X_t)-f(X_0)- \int_0^t \mathcal{A}f (X_s) ds
\]
is a local martingale for every $ f\in C^2_b(\R)$ if and only if $X$ is a weak solution of \eqref{eq:poly}.}
\end{remark}

Processes satisfying Definition~\ref{def:polaff} have remarkable properties, namely that all marginal moments of a polynomial process, i.e.~$\mathbb{E}[X_t^n]$ can be computed by solving a system of linear ODEs. In the case of affine processes, we additionally get that exponential moments $\mathbb{E}[e^{uX_t}]$ for $u \in \mathbb{C}$ can be expressed in terms of solutions of Riccati ODEs whenever $\mathbb{E}[|e^{uX_t}|]< \infty$.

Here, we present these implications from the point of view of dual processes (see e.g. Chapter 4  in \cite{EK:93}),
which has not been considered in this form in the original papers  \cite{DFS:03} and \cite{CKT:12, FL:16}, where affine and polynomial processes were first systematically introduced. 
Denote again by $\mathcal{A}$  the (extended) generator  given by \eqref{eq:polygen}.
We  can then distinguish two different ways  to compute $\mathbb{E}_x[f(X_t)]$. They are given by 
\begin{enumerate}
\item the Kolmogorov backward equation: $\mathbb{E}_x[f(X_t)]=g(t,x)$, where
\[
\partial_t g(t,x)=\mathcal{A}g(t,x), \quad g(0,x)=f(x).
\]
\item\label{itemii} the duality method: 
let $(U_t)_{t\in[0,T]}$ be an independent Markov process with state space $U$ and infinitesimal generator $\mathcal{B}$. Assume that there is some $f: S \times U  \to \mathbb{R}$ such that
\[
\mathcal{A} (f(\cdot, u))(x) =\mathcal{B} (f(x, \cdot))(u), \quad \text{for all } x \in S,\, u \in U,\]
then (modulo technicalities)
$
\mathbb{E}_x[f(X_t,u)]=\mathbb{E}_u[f(x,U_t)].
$
This duality relationship can then be used as a numerical method if the Kolmogorov backward equation for the dual process, i.e.
\begin{align}\label{eq:KBdual}
\partial_t v(t,u)=\mathcal{B}v(t,u), \quad v(0,u)=f(x,u)
\end{align}
with
$v(t,u)=\mathbb{E}_x[f(X_t,u)]$ is easy to solve or if the dual process
$U$ is easy to simulate or even deterministic.
\end{enumerate}

The latter is the case for affine and polynomial processes, because (one choice of) the dual process is deterministic and corresponds to the solution of ODEs, explaining the often stated tractability properties.
In the case of affine processes, the function family consists of \emph{exponentials},
 in the case of polynomial processes they are \emph{polynomials}.
Since these families are distribution determining (in the polynomial case under exponential moment conditions), existence of the dual process implies uniqueness of the solution of the martingale problem for the primal process $X$.

\paragraph{Affine case}
As we shall only treat \emph{linear processes} below, meaning that the characteristics are linear instead of affine, we  here also set $b_0=0$ and $a_0=0$ in the definition of the affine process. The corresponding state space is given by $\R_+$.

The \emph{dual affine operator} $\mathcal{B}$ acting on $u \mapsto \exp(ux)$ is  defined via
\[
\mathcal{B}(\exp( \fdot x))(u):=\mathcal{A} (\exp(u \fdot))(x),
\]
for each $x\in S$, which explicitly reads $\Bcal g(u)=R(u)g'(u)$ for $g=\exp(\cdot x)$ and
\begin{equation*}
R(u):=\frac{1}{2}a_1 u^2 +b_1 u.
\end{equation*}
From this we can see that
$\mathcal{B}$ is the restriction of the  transport operator and 
the corresponding (deterministic) dual process is a solution of the following ODE
\begin{align}\label{eq:Riccati1d}
\partial_t \psi(t)= R(\psi(t)), \quad \psi(0)=u,\qquad t\in[0,T],
\end{align}
which is of Riccati type due to the quadratic form of $R$.
The duality method \ref{itemii} with the deterministic dual process $\psi$  yields the original formulation of the affine transform formula while \eqref{eq:KBdual} for computing expectations via  the Kolmogorov equation for the dual process, leads to the corresponding transport equation formulation.
\begin{theorem}\label{th:affineRic1d}
Let $T >0$ be fixed and let $(X_t)_{t\in[0,T]}$ be a {linear} process. Let $u \in \mathbb{C}$ such that $\mathbb{E}[|\exp(u X_T)|] < \infty$.
Then,
$
\mathbb{E}_x[\exp(u X_T)]=\exp(\psi(T)x),
$
where $\psi$ solves the Riccati ODE given by \eqref{eq:Riccati1d}. Moreover,
$
\mathbb{E}_x[\exp(u X_T)]=v(T,u),
$
where $v(t,u)$ is a solution to the following linear PDE of transport type
\begin{align*}
\partial_t v(t, u) &=\mathcal{B}  v(t, u)=  R(u) \partial_{u} v(t, u), \quad 
v(0,u)= \exp(ux), \quad t \in [0,T].
\end{align*}

\end{theorem}

Note that the advantage of the transport formulation is that $v$ can assume the value $0$, which necessarily corresponds to an explosion of $\psi$. In the current 
one-dimensional setup this cannot occur, but on the infinite dimensional  state spaces that we consider below this can become an issue.
In Section~\ref{subsec:affine} we shall generalize both results to affine processes on subsets of the extended tensor algebra $T((\mathbb{R}^d))$ introduced below. They will then be applied to  signature SDEs in Section~\ref{sec4}.

\paragraph{Polynomial case}
Fix $k\in \N$ and 
denote by $\mathcal{P}_k:=\{x \mapsto \sum_{i=0}^k u_i x^i \colon u_i \in \mathbb{R} \}$ of polynomials up to degree $k$.
Moreover, for a vector of coefficients
$u=(u_0,\ldots,u_k)^{\top} \in \mathbb{R}^{k+1}$ we define 
$
p(x,u):= \sum_{i=0}^k u_i x^i
$.

Note that for a polynomial process, the generator $\mathcal{A}$ maps
 $\mathcal{P}_k$ to  $\mathcal{P}_k$, which is an alternative defining property of polynomial processes. Define a \emph{dual polynomial operator $\mathcal{B}$}
 acting on $u \mapsto p(x,u)$ as
\[
\mathcal{B}(p(x,\cdot))(u):=\mathcal{A}( p(\cdot,u))(x),
\]
for every $x\in S$, which explicitly reads
$\mathcal{B}g(u) =(L_ku)^\top \nabla g(u) $
for $L_k\in  \mathbb{R}^{(k+1)\times (k+1)}$  such that
\begin{align*}
\mathcal{A}(p(\cdot,u))(x) =p(x,L_ku).
\end{align*}
The corresponding $\mathbb{R}^{k+1}$-valued (deterministic) dual process solves the linear ODE
\begin{align}\label{eq:linearODE1d}
\partial_t c(t)=L_k c(t), \quad c(0)=u,\qquad t\in [0,T], 
\end{align}
which describes the evolution of the coefficients vector.
Applying the duality method \ref{itemii}, thus yields the following theorem, which was initially proved in \cite{CKT:12, FL:16}.

\begin{theorem} \label{th:poly}
Fix $T >0$ and let $(X_t)_{t\in[0,T]}$ be a polynomial process. Denote by $c(t)=(c_{0}(t), \ldots, c_{k}(t))^{\top}$  the solution of  the linear ODE 
\eqref{eq:linearODE1d}.
Then,
$
\mathbb{E}_x[\sum_{i=0}^k u_iX_T^i]= \sum_{i=0}^kc_{i}(T)  x^i.
$

\end{theorem}

\begin{remark}\label{rem:bidual}
In this case \eqref{eq:KBdual} yields
$\mathbb{E}_x[\sum_{i=0}^k u_iX_T^i]=v(t,u),$ where $v(t, u)$
satisfies
$$
\partial_t v(t, u)=\Bcal v(t, u)= (L_ku)^\top \nabla_{u} v(t,u).$$
Inserting $\mathbb{E}_x[\sum_{i=0}^k u_iX_t^i]$ for $v(t,u)$
thus gives
$
\partial_t \mathbb{E}[\sum_{i=0}^k u_iX_t^i]=\langle L_k(u),\mathbb{E}[(1,X_t,\ldots,X_t^k)^\top]  \rangle
$
and in turn
\[
\partial_t \mathbb{E}[(1,X_t,\ldots,X_t^k)^\top]= L_k^{\top}\mathbb{E}[(1,X_t,\ldots,X_t^k)^\top].
\]
In the terminology of \cite{CS:21}, this corresponds to the so-called bidual moment formula.

\end{remark}

\subsection{Signature of continuous semimartingales}\label{sec22}

Our approach of passing from the narrow definition of one-dimensional affine and polynomial processes to a universal class is based on the signature process,  which  was first studied by \cite{C:57, C:77} and plays a prominent role in  \emph{rough path theory} introduced in \cite{L:98} (see also  the monographs by \cite{FV:10, FH:14}).

Following the lines of Section~2.1 in \cite{CGS:22} we start here by introducing some basic notions related to the signature  of an $\mathbb{R}^{d}$-valued continuous semimartingale. Similar introductions to the concept of signature can be found, e.g., in Section 2.2 of \cite{BHRS:21}.

For each $n \in \mathbb{N}_0=\{0,1,2, \ldots\}$ consider the $n$-fold tensor product of $\mathbb{R}^{d}$ given by
\begin{equation*}
(\mathbb{R}^{d})^{\otimes 0}:=\mathbb{R}, \qquad (\mathbb{R}^{d})^{\otimes n}:=\underbrace{\mathbb{R}^{d}\otimes\cdots\otimes\mathbb{R}^{d}}_{n}.
\end{equation*}
	For $d\in \N$, we define the extended tensor algebra on $\mathbb{R}^{d}$ as 
	\begin{equation*}
		T((\mathbb{R}^{d})):=\{\textbf{a}:=(a_{0},\dots,a_{n},\dots) : a_{n}\in(\mathbb{R}^{d})^{\otimes n}\}.
	\end{equation*}
	Similarly we introduce the truncated tensor algebra of order $N \in \mathbb{N}$ 
		\begin{equation*}
        T^{(N)}(\mathbb{R}^{d}):=\{\textbf{a}:=(a_{0},\dots,a_{N},0,\dots) : a_{n}\in(\mathbb{R}^{d})^{\otimes n}\}.
	\end{equation*}
	and the tensor algebra
$
		T(\mathbb{R}^{d}):=\bigcup_{N\in \N}T^{(N)}(\mathbb{R}^{d}).
$	For  $\textbf{a},\textbf{b}\in T((\mathbb{R}^{d}))$ and  $\lambda\in\R$ we set 
	$
	\pi_n(\textbf a):=a_n,$ $
		\tru {\mathbf a} N:=(a_0,\ldots,a_N),$
		$\textbf{a}+\textbf{b}:=(a_{0}+b_{0},\dots,a_{n}+b_{n},\dots),$ $
		\lambda \textbf{a}:= (\lambda a_{0},\dots, \lambda a_{n},\dots),$ $
		\textbf{a}\otimes \textbf{b}:=(c_{0},\dots, c_{n},\dots),
	$
where $c_{n}:=\sum_{k=0}^{n}a_{k}\otimes b_{n-k}$. Whenever needed, we identify $(a_0,\ldots,a_N)\in T(\R^d)$ with $(a_0,\ldots,a_N,0,\ldots,0)\in T((\R^d))$.
{Observe that 
$T((\mathbb{R}^{d}))$, equipped with addition, scalar multiplication, and the tensor product, forms an algebra over the real numbers.}
The neutral element is identified via $(1,0,\dots,0,\dots)$. The tensor algebra $T(\mathbb{R}^{d}) \subset T((\mathbb{R}^{d})) $ is a subalgebra and isomorphic to \emph{the} real free algebra with $d$ generators. Note that $T^{(N)}(\mathbb{R}^{d})$ is an algebra as well, namely \emph{the} real free $N$-nilpotent algebra. It appears as a factor algebra of $T((\mathbb{R}^{d}))$ by the ideal of elements with entries only after $N$. We shall need these algebraic structures to introduce signature properly.

For a multi-index  $I:=(i_1,\ldots,i_n) \in {\{1,\ldots,d\}}^n $  we denote its length by $|I|:=n$. We also consider the empty index $I:=\emptyset$ and set $|I|:=0$. For each multi-index $I$ with $|I|>0$ we define
\begin{equation}\label{eqn18}
	 I':=
	     \begin{cases}
	          (i_1,\ldots,i_{|I|-1}) &\text{if }|I|\geq2,\\
	          \emptyset&\text{if }|I|=1,
	     \end{cases}
\end{equation}
which just removes the last letter. Similarly we set $I'':=(I')'$ if $|I|>1$. Finally, we denote by $I^{ord}$ the multi-index obtained by sorting the entries of $I$ and define $I! :=\prod_{k=1}^d (I(k)! )
$
where $I(k)=\sum_{j=1}^{|I|}  1_{\{i_j=k\}}$ counts the number of times the letter $k$ appears in multiindex $I$.     
Next, for each $|I|\geq 1$ we set
	$$e_I:=e_{i_1}\otimes\cdots\otimes e_{i_n},$$
	where $e_1, \ldots, e_d$ denote the canonical basis vectors of $\mathbb{R}^d$.
	Denoting by $e_\emptyset$ the basis element corresponding to $(\R^d)^{\otimes 0}$, each element of $\textbf{a}\in T((\R^d))$ can thus be written as formal series
$$\textbf{a}=\sum_{|I|\geq 0}\a_Ie_I,$$
for some $\a_I\in \R$. Notice that convergence is not just formal here but componentwise (which is actually the standard locally convex structure on the sequence space) and whence properly defined -- still we maintain the wording of formal series. To extract the $I$-th component from $\textbf{a}$, we shall either write $\textbf{a}_I$ or $\langle \textbf{a}, e_I \rangle$. For each $\textbf{a}\in T(\R^d)$ and each $\textbf{b}\in T((\R^d))$ we set
$$\langle \textbf{a},\textbf{b}\rangle:=\sum_{|I|\geq 0} \a_{I}\b_{I}.$$ 
Observe that the set $\{e_I\colon |I|=n\}$ is an orthonormal basis of $(\mathbb{R}^{d})^{\otimes n}$ with respect to this $\langle\cdot,\cdot\rangle$. In particular, $\b_I=\langle e_I,\textbf{b}\rangle$.	\\

Finally, denote by $(\fdot)^{(1)}:T((\mathbb{R}^{d}))\to( T((\mathbb{R}^{d})))^{d}$ and $(\fdot)^{(2)}:T((\mathbb{R}^{d}))\to (T((\mathbb{R}^{d})))^{d\times d}$ the shifts given by
\begin{align*}
    {\bf u}^{(1)}&:=\Big(\sum_{|I|\geq1}\u_I e_{I'}\Big)
(1_{\{i_{|I|}=1\}},\ldots,1_{\{i_{|I|}=d\}})^\top,\\
{\bf u}^{(2)}&:=\Big(\sum_{|I|\geq2}\u_I e_{I''}\Big)
\begin{pmatrix}
1_{\{i_{|I|-1}=i_{|I|}=1\}} & \cdots & 1_{\{i_{|I|-1}=1,i_{|I|}=d\}}\\
\vdots & \cdots &\vdots \\
1_{\{i_{|I|-1}=d,i_{|I|}=1\}}&\cdots& 1_{\{i_{|I|-1}=i_{|I|}=d\}}
\end{pmatrix}.
\end{align*}
Moreover, we define
$$\langle \a,\u^{(1)}\rangle:
=\sum_{|I|\geq 1}\big( \a_{I}\u_{(I1)},\ldots, \a_{I}\u_{(Id)} \big)^\top
\quad\text{and}\quad\langle \a,\u^{(2)}\rangle:
=\sum_{|I|\geq 2}\begin{pmatrix}
\a_{I}\u_{(I11)} & \cdots & \a_{I}\u_{(I1d)}\\
\vdots & \cdots &\vdots \\
\a_{I}\u_{(Id1)}&\cdots& \a_{I}\u_{(Idd)}
\end{pmatrix}
$$
for each $\a\in T(\R^d)$. Here, we use the notation $Ik$ for the multiindex $I$ prolonged by $k$.

Throughout the paper we consider a filtered probability space $(\Omega, \Fcal, (\Fcal_t)_{t\in[0,T]},\P)$ with a right continuous filtration. Unless otherwise specified, stochastic processes are always supposed to be defined thereon.
We are now ready to define the signature of an $\mathbb{R}^d$-valued continuous semimartingale.

\begin{definition}\label{def2}
	Let $(X_t)_{t\in[0,T]}$ be a continuous $\mathbb{R}^{d}$-valued semimartingale. Its \emph{signature} is the $T((\R^d))$-valued process  $t\mapsto \X_{t}$ whose components are recursively defined as
\begin{equation*}
	\langle e_{\emptyset},\mathbb{X}_{t}\rangle:=\textup{1}, \qquad \langle e_{I}, \mathbb{X}_{t}\rangle:=\int_{0}^{t}\langle e_{I'},\mathbb{X}_{r}\rangle\circ \mathrm{d} X_{r}^{i_{|I|}},
\end{equation*}
for each $I=(i_1,\ldots, i_n)$ and $t\in[0,T]$, where $\circ$ denotes the Stratonovich integral. 
Its projection $\X^{\leq N}$ on $T^{(N)}(\mathbb{R}^{d})$ is given by
\begin{equation*}
	\mathbb{X}_{t}^{\leq N}=\sum_{|I|\leq N} \langle e_{I}, \mathbb{X}_{t}\rangle e_{I}
\end{equation*}
and is called \emph{signature of $X$ truncated at level $N$}. 
\end{definition}
Observe that the signature of $X$ and the signature of $X-c$ coincide for each $c\in \R^d$, and are given by iterated Stratonovich integrals
$$
\langle e_{I}, \mathbb{X}_{t}\rangle = \int_{0 \leq t_1\leq \ldots t_n \leq t} \circ dX^{i_1}_{t_1} \cdots \circ dX^{i_n}_{t_n} \quad .
$$

In order to describe the algebraic properties of the signature we need the shuffle product, which is nothing else than a concise expression of the classical integration by parts formula, which  holds for  Stratonovich integrals.

\begin{definition}\label{shuffle-product}
	For any two multi-indices $I:=(i_1,\ldots,i_n)$ and $J:=(j_1,\ldots,j_m)$ the \emph{shuffle product} is defined recursively as
	\begin{align*}
		e_{I}\shuffle e_{J}:= (e_{I'}\shuffle e_{J})\otimes e_{i_n}+(e_{I}\shuffle e_{J'})\otimes e_{j_m},
	\end{align*}
	with $e_{I}\shuffle e_{\emptyset}:= e_{\emptyset}\shuffle e_{I}= e_{I}$. It extends to $\textbf{a},\textbf{b}\in T(\R^d)$ as
	$$\textbf{a}\shuffle\textbf{b}=\sum_{|I|,|J|\geq0}\a_I\b_J (e_I\shuffle e_J)$$
	and to $\textbf{a},\textbf{b}\in T((\R^d))$ as 
	$\langle e_I,\textbf{a}\shuffle\textbf{b}\rangle=\langle e_I ,\mathbf a^{\leq |I|}\shuffle \mathbf b^{\leq |I|}\rangle.$
\end{definition}
\begin{remark}\label{rem1}
The shuffle product has many important properties:
\begin{enumerate}\label{shuffprop}
    \item  $(T(\R^d),+,\shuffle)$ is a commutative algebra, which in particular means that the shuffle product is associative and commutative.
    \item\label{it6ii} $e_I\shuffle e_J\in(\R^d)^{\otimes(|I|+|J|)}$, which in particular implies that $e_I\shuffle e_J=\sum_{|H|=|I|+|J|}\lambda_He_{H}$  for some  $\lambda_H\in \R$.
    \item\label{it6iii} For each multi-index $I$ it holds
$\sum_{J^{ord}=I^{ord}}e_J
=\frac 1 {I!}(
e_{i_1}\shuffle\ldots \shuffle e_{i_{|I|}})$.
\item \label{th:shuffle}
Let $(X_t)_{t\in[0,T]}$ be a continuous $\mathbb{R}^{d}$-valued semimartingale and $I,J$ be two multi-indices. Then
$
		\langle e_{I},\mathbb{X}\rangle \langle e_{J}, \mathbb{X}\rangle =\langle e_{I}\shuffle e_{J}, \mathbb{X}\rangle\ a.s.$ (see \cite{R:58} or \cite{LCL:07}) and 
$$
\langle \textbf{a}, \mathbb{X} \rangle \langle \textbf{b}, \mathbb{X} \rangle =  \langle \mathbf{a} \shuffle \mathbf{b}, \mathbb{X}\rangle
$$
for $\mathbf{a},\mathbf{b} \in T(\mathbb{R}^d)$.
        
\end{enumerate}
\end{remark}

In the following definition we introduce the notion of the exponential shuffle, similarly as in Definition~6.1 of \cite{BHRS:21}. Note that in contrast to that definition there we consider it also for elements in the (complexified) extended tensor algebra and not only on $T(\mathbb{R}^d)$. 
The complexification of the extended tensor algebra is denoted by $T((\mathbb{R}^d)) + \text{i} T((\mathbb{R}^d))$. Note that we still apply $\langle\cdot,\cdot\rangle$ as introduced above to both the real and imaginary entries.

\begin{definition} \label{def:expshuffle}
For each $\u\in T((\R^d))+\text{i}T((\R^d))$ we set $\overline \u:=\u-\u_\emptyset e_\emptyset$ and consider the following quantities.
\begin{itemize}
    \item $\u^{\shuffle 0}=e_\emptyset$, $\u^{\shuffle 2}:=\u\shuffle \u$ and $\u^{\shuffle k}:=\u^{\shuffle (k-1)}\shuffle \u$.
    \item $\exp(\shuffle \u):=\sum_{k=0}^\infty \frac{\u^{\shuffle k}}{k! }$. Since the right hand side can be written as $\exp(\u_\emptyset)\sum_{k=0}^\infty \frac{\overline \u^{\shuffle k}}{k! }$, we can see that every component of $\exp(\shuffle \u)$ corresponds to a finite sum. 
    \item 
    {$\log(\shuffle\u):=\sum_{k=1}^\infty\frac {(-1)^{k-1}{\overline\u^{\shuffle k}}}k$, whenever $\u_\emptyset=1$.}
\end{itemize}
\end{definition}
 
\begin{lemma}\label{lem6}
For $\u_0:=0$ it holds $\exp(\shuffle \u_0)=e_\emptyset$ and for each $\u,\v$ it holds
\begin{align*}
\exp(\shuffle \u)\shuffle \exp(\shuffle \v)&=\exp(\shuffle (\u+ \v)).
\end{align*}
Moreover, for each $\u$ such that $\u_\emptyset=1$ it holds
$\exp(\shuffle\log(\shuffle \u))=\u.$
\end{lemma} 
One of the most important properties of the signature states that every continuous functional on a compact set of continuous semimartingales' trajectories can be approximated by a linear map of the time extended trajectory's signature, evaluated at its final time. This universal approximation result is based on the properties of the Lyons lift (see for instance Theorem~2.2.1 in \cite{L:98} and Theorem 9.5 and Exercise 17.2 in \cite{FV:10}) and on the Stone-Weierstrass theorem. It uses the fact that two trajectories cannot share the same value of the time-extended signature at final time and that linear maps on the signature form an algebra. As this result is not used directly in this paper and stating it would require to specify the involved spaces and topologies we direct the reader to the following references \cite{KLP:20,BHRS:21, CGS:22,CM:22, CPS:22}.

\section{Affine and polynomial processes on $T((\mathbb{R}^d))$}\label{subsec:affine} 

We introduce here the notion of affine or polynomial processes (which appear often \emph{both} just as linear processes) taking values in a subset of the extended tensor algebra $T((\mathbb{R}^d))$. 

\subsection{State space and spaces of dual elements}\label{sec:dualelements}

The goal of this section is to define a vector space of dual elements needed for the definition of affine and polynomial processes. We shall first introduce vector spaces of linear functionals $ \mathcal{S} \ni \mathbf{x} \mapsto \langle \u, \x \rangle $ restricted to some state space $\mathcal{S}$, which can be related in certain cases to power series. Sometimes these vector spaces already form a point-separating algebra, sometimes we have to form one, for instance by considering $ \mathcal{S} \ni \mathbf{x} \mapsto \exp \big ( \langle \u, \x \rangle \big) $: in both cases this will provide us with large enough sets of test functions for the martingale problem approach, which we use to define stochastic processes.

\begin{itemize}
\item State space of the stochastic process: $\mathcal{S} \subseteq T((\mathbb{R}^d))$. The set $\Scal$ is  endowed with the sigma-algebra induced by the product sigma-algebra on  $T((\mathbb{R}^d))$.
\item $\Pi:=\{\Pi_1,\Pi_2,\Pi_3,\ldots\}$ denotes a partition of the set $\{I\colon|I|\geq 0\}$ of all multi-indices. We assume that each $\Pi_k$ just contains indices of the same length.
\item For each $\x\in \Scal$ and $\u\in T((\mathbb{R}^d)) + \text{i} T((\mathbb{R}^d))$ we set
\begin{equation}\label{eqn5}
|\u|_\x:=\sum_{k=1}^\infty\big|\sum_{I\in \Pi_k}\co \u I \co \x I\big|.
\end{equation}

\item $\mathcal{S}^*:=\{ \mathbf{u} \in T((\mathbb{R}^d)) + \text{i} T((\mathbb{R}^d))\colon |\u|_\x<\infty \text{ for all } \mathbf{x} \in \mathcal{S} \}$.
\item For ${\bf u}\in \Scal^*$ and $\x\in \Scal$ we set 
\begin{align}\label{eq:uxana}
\langle {\bf u},{\bf x}\rangle:=\lim_{N\to\infty}\sum_{n=0}^N \langle \pr\u n,\pr\x n\rangle.
\end{align}
\end{itemize}

\begin{remark}\label{rem4}
 The choice of avoiding partitions with sets containing indices of  different length guarantees that 
the limit defining $
\langle \u,\x\rangle
$ is well defined for all $\u \in \Scal^*$  and all $\x \in \Scal$. With a different order of summation, $\Scal^*$ could potentially be larger but the limit might no longer be well-defined.

Observe that  $\Scal^*$ is a point-separating vector space. Point separation is already given by dual elements, where only finitely many components are non zero. Given $\u,\v\in \Scal^*$ and $\lambda,\mu\in \R$ it holds $\lambda\u+\mu\v\in \Scal^*$ and  $\lambda\langle \u,\x\rangle+\mu\langle\v,\x\rangle=\langle\lambda\u+\mu\v,\x\rangle$ for each $\x\in \Scal$.

  Finally, note that there is only one choice of partition for $d=1$.    
\end{remark}

\begin{example}\label{ex2}
 We discuss now some  choices for $\Pi$.
\begin{enumerate}
    \item \label{it3i} $\Pi:=\{\{I\}\colon |I|\geq 0\}$. In this case
$ |\u|_\x=\sum_{|I|\geq0}|\co \u I\co \x I|.$
If every entry of $\x$ is nonzero $ |\u |_{\x }$ is a norm. As we shall see, this partition is however not shuffle-compatible in the sense of Definition~\ref{def:shufflecomp}.
\item\label{it3ii} $\Pi:=\{\{I\colon I^{ord}=J\}\colon J=J^{ord}\}$. In this case
$$
|\u|_\x=\sum_{I\colon I=I^{ord}}\Big|\sum_{J\colon J^{ord}=I}\co \u J\co \x J\Big|,$$ where $I^{ord}$ is the ordered version of $I$. 
The advantage of this semi-norm lies in the connection with power series explained in Lemma~\ref{lem1}~\ref{it11iii} below.
In particular, this partition is the finest shuffle-compatible partition (see Definition~\ref{def:shufflecomp}).

 \item\label{it3iii} $\Pi:=\{\{I\colon |I|=n\}\colon n\geq 0\}$. In this case
$$
|\u|_\x=\sum_{n=0}^\infty|\langle \pr\u n,\pr\x n\rangle|.$$ This is the smallest semi-norm satisfying the necessary properties, thus the one corresponding to the largest $\Scal^*$. This partition is also shuffle-compatible (see Definition~\ref{def:shufflecomp}).
\end{enumerate}
\end{example}

\subsection{Martingale problems}
Let $M(\Scal) $ be the set of  maps $f:\Scal\to \C$  and fix a linear subset $D(\Lcal)\subseteq M(\Scal)$. Consider a linear operator $\mathcal{L}:D(\mathcal L)\to M(\Scal)$ and  a filtered probability space $(\Omega, \mathcal{F},(\mathcal{F}_t)_{t\in[0,T]}, \mathbb{P})$ with right continuous filtration.
An $\mathcal{S}$-valued stochastic process $(\mathbb{X}_t)_{t\in[0,T]}$ is called a \emph{solution to the martingale problem for $\mathcal{L}$}  if 
\begin{enumerate}
\item\label{iti} $\mathbb{X}_0= \mathbf{x}$ for some initial value $\mathbf{x}\in \mathcal{S}$, 
\item\label{itii}
there exists a version of $(\mathbb{X}_t)_{t\in[0,T]}$ such that $(f( \mathbb{X}_t))_{t\in[0,T]}$ and $(\Lcal f( \mathbb{X}_t))_{t\in[0,T]}$ 
are adapted  c\`adl\`ag processes for every $f\in D(\Lcal)$
and
\item\label{itiii} the process $N^f$ given by
\begin{equation}\label{eqnNnew}
N^f_t := f(\mathbb{X}_t)- f(\mathbb{X}_0) - \int_0^t \mathcal{L}f(\mathbb{X}_s) ds
\end{equation}
defines a local martingale for every $f\in D(\Lcal)$. 
\end{enumerate}

Note that we write $(\mathbb{X}_t)_{t\in[0,T]}$  for a general $\mathcal{S}$-valued  process which in the present context does not necessarily denote the signature process of some underlying process.

\subsection{The affine case}

We now use the martingale problem formulation to introduce affine processes on $\mathcal{S}$. {Observe that for each $\mathbf{x} \in \mathcal{S}$, we have $\langle e_\emptyset, \mathbf{x}\rangle = 1$.
Consequently, constant maps on $\mathcal{S}$ are linear and, without loss of generality, we can focus on linear processes in analogy with Theorem~\ref{th:affineRic1d}.}

\begin{definition}\label{def:primalaffine}
Fix $D(\Lcal):=\Span\{\mathbf{x} \mapsto \exp(\langle \mathbf{u}, \mathbf{x} \rangle) \colon  \mathbf{u} \in \mathcal{U}\}$ for some subset $\mathcal{U} \subseteq \Scal^*$  and let $\mathbf{\mathcal{L}}: D(\Lcal)\to M(\Scal)$ be a linear operator. 
We say that $\Lcal$ is
of $\Ucal$-\emph{affine type} if there exists 
a map $R: \mathcal{U} \to \mathcal{S}^*, \mathbf{u} \mapsto R(\mathbf{u})$
such that 
\begin{align}\label{eq:primalaffine}
\mathcal{L} \exp(\langle \mathbf{u}, \fdot \rangle)(\mathbf{x})= \exp(\langle \mathbf{u}, \mathbf{x} \rangle) \langle R(\mathbf{u}), \mathbf{x} \rangle
\end{align}
for each ${\bf u}\in \Ucal$ and ${\bf x} \in \Scal$.
Any solution $(\mathbb{X}_t)_{t\in[0,T]}$ of a martingale problem corresponding to an operator $\Lcal$ of $\Ucal$-affine type  is then called  $\mathcal{S}$-valued  $\Ucal$-\emph{affine process.}
\end{definition}

\begin{remark}
Observe that a solution to the
martingale problem is supposed to take values in $\Scal$ and that there is no cemetery state which would allow for a killing. This assumption has some useful consequences.
Suppose for a moment that $0\in \Ucal$ and thus that the map $f\equiv 1$ belongs to $\Dcal(\Lcal)$. Letting $(\X_t)_{t\in[0,T]}$ be the corresponding affine process with initial condition $\bf x$,   the process $(\int_0^t \Lcal f(\X_s)ds)_{t\in [0,T]}$
is a local martingale of finite variation and thus a constant process. 
Since $(\Lcal f(\X_t))_{t\in[0,T]}$ is assumed to be c\`adl\`ag, we can conclude that 
$\Lcal f({\bf x})=0$  and hence that $\langle R(0),{\bf x }  \rangle=0$ for all ${\bf x}\in \Scal$. 
This motivates the assumption $R(0)=0$, whenever this is needed.
\end{remark}

With these definitions at hand and certain integrability conditions we can now prove a first version of the affine transform formula in the present setting.

\begin{theorem}\label{thm:affinetrans2}
Fix $T >0 $ and let $(\mathbb{X}_t)_{t\in[0,T]}$  be an $\mathcal{S}$-valued  $\Ucal$-affine process with initial value $\mathbf{x}_0$ and  operator $\mathcal{L}$ of $\Ucal$-affine type as of Definition~\ref{def:primalaffine}.
Fix $\u\in \Ucal$ and suppose that  the Riccati ODE
\begin{equation}\label{eqn2}
{\bm \psi}(t) = \u+\int_0^t R({\bm \psi}(s)) ds, 
\end{equation}
admits an  $\mathcal{U}$-valued  solution $({\bm \psi}(t))_{t\in [0,T]}$ such that
$\int_0^T|R(\bm\psi(s))|_\x ds<\infty$ 
for all $\mathbf{x} \in \mathcal{S}$.
Moreover, assume that
\begin{equation} \label{eq:intcond1}
\begin{aligned}
&\mathbb{E}[\sup_{s,t \leq T} |\exp(\langle {\bm \psi}(s), \mathbb{X}_t \rangle)|] < \infty,\quad\text{ and} \\
 &\mathbb{E}[\sup_{s, t \leq T} |\langle R({\bm \psi}(s)), \mathbb{X}_t\rangle \exp(\langle {\bm \psi}(s), \mathbb{X}_t\rangle)|] < \infty.
\end{aligned}
\end{equation}
Then,  it holds
$
\mathbb{E}[\exp(\langle \mathbf{u}, \mathbb{X}_T \rangle) ]= \exp(\langle {\bm \psi}(T) , \mathbf{x}_0 \rangle).
$
\end{theorem}

\begin{proof}
We verify the conditions of Lemma~\ref{lemA} for 
$$X(s,t)=\exp(\langle{\bm \psi}(s),\X_t\rangle)\quad\text{and}\quad Y(s,t)= \langle R({\bm \psi}(s)), \mathbb{X}_t \rangle \exp( \langle{\bm \psi}(s), \mathbb{X}_t\rangle).$$
As $(\X_t)_{t\in[0,T]}$ is an $\Ucal$-affine process, we have by definition that
\begin{equation}\label{eqn11new}
M^{\mathbf{u}}_t =\exp( \langle \mathbf{u}, \mathbb{X}_t\rangle)-\exp( \langle \mathbf{u}, {\bf x}_0\rangle) - \int_0^t \langle R(\mathbf{u}), \mathbb{X}_r \rangle \exp( \langle \mathbf{u}, \mathbb{X}_r\rangle) dr,
\end{equation}
is a local martingale. 
Using that
\begin{align*}
\mathbb{E}[\sup_{t \leq T} |M^{\mathbf{u}}_t |] &\leq 2 \mathbb{E}[\sup_{t \leq T} |\exp(\langle \mathbf{u}, \mathbb{X}_t\rangle)|] 
 +T \mathbb{E}[\sup_{t \leq T} |\langle R(\mathbf{u}), \mathbb{X}_t\rangle \exp(\langle \mathbf{u}, \mathbb{X}_t\rangle)|]\\
& \leq 2 \mathbb{E}[\sup_{s,t \leq T} |\exp(\langle {\bm \psi}(s), \mathbb{X}_t \rangle)|] 
 +T \mathbb{E}[\sup_{s, t \leq T} |\langle R({\bm \psi}(s)), \mathbb{X}_t\rangle \exp(\langle {\bm \psi}(s), \mathbb{X}_t\rangle)|] < \infty
\end{align*}
holds by \eqref{eq:intcond1}, we can conclude that $(M^{\mathbf{u}}_t)_{t\in[0,T]}$ is a c\`adl\`ag  martingale.
Furthermore, since $(\bm\psi(s))_{s\in[0,T]}$ solves the Riccati ODE \eqref{eqn2} componentwise we have that
\[
\langle {\bm \psi}(s), \mathbf{x}^{\leq N} \rangle  = \langle \mathbf{u}, \mathbf{x}^{\leq N} \rangle  + \int_0^s \langle R({\bm \psi}(r)), \mathbf{x}^{\leq N} \rangle  dr,
\]
for each $N$. {The dominated convergence theorem with the bound $|R(\bm\psi(r))|_\x$ yields then
\[
\langle {\bm \psi}(s), \mathbf{x} \rangle  = \langle \mathbf{u}, \mathbf{x} \rangle  + \int_0^s \langle R({\bm \psi}(r)), \mathbf{x} \rangle  dr.
\]
Since $\int_0^s |\langle R({\bm \psi}(r)), \mathbf{x} \rangle|  dr\leq \int_0^s |R(\bm\psi(r))|_\x dr<\infty$, we can apply the It\^o formula to the deterministic process $(\langle {\bm \psi}(s), \mathbf{x} \rangle)_{s\in [0,T]}$ and obtain that}
\[
\exp(\langle {\bm \psi}(s), \mathbf{x} \rangle ) - \exp(\langle \mathbf{u}, \mathbf{x} \rangle ) - \int_0^s \langle R({\bm \psi}(r)), \mathbf{x} \rangle \exp( \langle {\bm \psi}(r), \mathbf{x}\rangle) dr=0,
\]
which is thus a constant martingale for each ${\bf x} \in \Scal$.
Due to \eqref{eq:intcond1} all the assumptions of Lemma~\ref{lemA} are  satisfied and the assertion follows.
\end{proof}

\begin{remark}\label{rem2new}
An inspection of the proof and Lemma~\ref{lemA} shows that \eqref{eq:intcond1} can be replaced by the following two properties.
\begin{enumerate}
\item The process $(M_t^{{\bf u}})_{t\in[0,T]}$ given by \eqref{eqn11new}
is a true martingale for each ${\bf u}\in \Ucal$. 
\item $\int_0^T\int_0^T \E[|\langle R({\bm \psi}(s)), \X_t \rangle \exp( \langle {\bm \psi}(s),\X_t\rangle)|]ds dt<\infty$.
\end{enumerate}
Similarly, conditions \eqref{eqn2} and $\int_0^T|R(\bm\psi(s))|_\x ds<\infty$ can be replaced by 
$$\langle{\bm \psi}(t),\x\rangle = \langle\u,\x\rangle+\int_0^t \langle R({\bm \psi}(s)),\x\rangle ds,\qquad \x\in \Scal.$$
\end{remark}

 Next, we show that $t\mapsto\mathbb{E}[\exp(\langle \mathbf{u}, \mathbb{X}_t \rangle) ]$ is in fact a solution of a transport equation along $s\mapsto {\bm \psi}(s)$. Mathematically, this means that setting $v(t,\u):=\mathbb{E}[\exp(\langle \mathbf{u}, \mathbb{X}_t \rangle) ]$ we get
   $$\partial_t v(t,{\bf u})
 = \partial_s v(t,{\bm \psi}(s))|_{s=0}.$$

\begin{corollary}\label{cor1}
Fix $\u\in \Ucal$ and suppose that the conditions of Theorem~\ref{thm:affinetrans2} hold for some $\mathcal{U}$-valued solution $({\bm \psi}(t))_{t\in [0,T]}$ of \eqref{eqn2} with ${\bm \psi}(0)={\bf u}$.
Then 
$
\mathbb{E}[\exp(\langle \mathbf{u}, \mathbb{X}_T \rangle) ]= v(T,\mathbf{u}),
$
for some map $v:[0,T]\times \Ucal\to \R$ such that the following holds.
\begin{enumerate}
\item\label{it5i}
$v(t,{\bm \psi}(s))$ is absolutely continuous in $t$ for each fixed $s\in[0,T]$ and in $s$ for each fixed $t\in[0,T]$. Moreover, 
$$\int_0^T\int_0^T |\partial_tv(t,{\bm \psi}(s))|dtds<\infty,\quad\text{and }\quad \int_0^T\int_0^T |\partial_sv(t,{\bm \psi}(s))|dtds<\infty.$$
\item\label{it5ii} $\partial_t v(t,{\bm \psi}(s))
 = \partial_s v(t,{\bm \psi}(s))$ for each $s, t \in [0,T]$ and  $v(0,{\bf u}) =\exp (\langle \mathbf{u},\mathbf{x}_0\rangle )$.
\end{enumerate}

\end{corollary}

\begin{proof}
Proceeding as in the proof of Theorem~\ref{thm:affinetrans2} we get that  the process $(M_t^{{\bf u}})_{t\in[0,T]}$ given by \eqref{eqn11new} is a true martingale satisfying $M_0^{{\bf u}}=0$. Together with \eqref{eq:intcond1} this in particular implies that 
$\E[M_t^{{\bm \psi}(s)}]=0$ and thus that $v(t,{\bm \psi}(s))$ is absolutely continuous in $t$ for each fixed $s\in[0,T]$. The corresponding derivative $\partial_t v(t,{\bm \psi}(s))$ is given by
\begin{equation}\label{eqn15}
\E[\langle R({\bm \psi}(s)), \mathbb{X}_t \rangle \exp( \langle {\bm \psi}(s), \mathbb{X}_t\rangle)].
\end{equation}
Similarly, \eqref{eqn2},   
$\int_0^T|R(\bm\psi(s))|_\x ds<\infty$, and  \eqref{eq:intcond1} yield that 
$v(t,{\bm \psi}(s))$ is absolutely continuous in $s$ for each fixed $t\in[0,T]$ and the corresponding derivative $\partial_s v(t,{\bm \psi}(s))$ is again given by \eqref{eqn15}. Hence, \ref{it5ii} follows now directly and \ref{it5i} follows by \eqref{eq:intcond1}.
\end{proof}

As in the one dimensional case we can now define the dual operator for affine processes.

\begin{remark}
Let $\mathcal{L}$ be an operator of $\Ucal$-affine type. The  \emph{dual $\Ucal$-affine operator} is the linear operator $\Bcal$ 
  acting on $D(\Bcal):=\Span\{g:\Ucal\to \R\colon g({\bf u})=  \exp(\langle \mathbf{u}, \mathbf{x} \rangle )$ for $\mathbf{x}\in \mathcal{\Scal}\}$ such that
\[
\mathcal{B}\exp(\langle \fdot,  \mathbf{x} \rangle)({\bf u)}:=\mathcal{L} \exp(\langle  \mathbf{u} , \fdot\rangle )({\bf x}), \quad \mathbf{u} \in \Ucal.
\]

Clearly
$\mathcal{B}\exp(\langle \fdot,  \mathbf{x} \rangle)({\bf u})=\exp(\langle \mathbf{u}, \mathbf{x} \rangle) \langle R(\mathbf{u}), \mathbf{x} \rangle$ for each ${\bf u} \in \Ucal$ and ${\bf x}\in \Scal$.
To show that
$\mathcal{B}$ can -- similarly as in the one dimensional situation -- be interpreted as the restriction of a transport operator,
suppose that for each ${\bf u }\in \Ucal$ and sufficiently small $\varepsilon >0$, there exists a solution $(\bm\psi(t))_{t\in[0,\e]}$
to the Riccati ODE \eqref{eqn2}, such that $\int_0^\e|R(\bm\psi(s))|_\x ds<\infty$.
Then the linear operator $\Bcal$ satisfies
$$\lim_{\e\to 0} \frac{|g({\bm \psi}(\e))-g({\bf u})-\e \Bcal g ({\bf u})|}{\e}
=\lim_{\e\to 0} \frac{|g({\bf u} + \int_0^\e R({\bm \psi}(s)) ds)-g({\bf u})-\e \Bcal g ({\bf u})|}{\e}=
0,$$
for  each $g\in \Dcal(\Bcal)$. 
This shows that $\Bcal g ({\bf u})$ can be seen as the directional derivative of $g$ at ${\bf u}$ in direction $R({\bf u})$.
\end{remark}

In order to get formulations that do not depend on the existence of the Riccati ODEs, we introduce the following property for $\mathcal{U}$ allowing for componentwise differentiability.

\begin{definition}\label{def1}
We say that the set $\Ucal$ is componentwise-semi-open  if for each ${\bf u}\in \Ucal$ and $|I|\geq0$ the set $\Ucal_{I,{\bf u}}:=\{y\in\R\colon {\bf u}+ye_I\in \Ucal\}$ is semi-open\footnote{A set $A$ is semi-open if there is an open set $O$ such that $O\subseteq A\subseteq \overline O$ where $\overline O$ denotes the closure of $O$.}. For a componentwise-semi-open set $\Ucal$ we say that a map $v:\Ucal\to\R$ is componentwise-differentiable at ${\bf u}$ if each function $v_{I,{\bf u}}:\Ucal_{I,{\bf u}}\to \R$ given by $v_{I,{\bf u}}(y):=v({\bf u}+ye_I)$ is differentiable at $y=0$ (if 0 lies in the interior of $\Ucal_{I,{\bf u}}$)   or if it is one-sided differentiable at $y=0$ (if 0 lies in the boundary of $\Ucal_{I,{\bf u}}$). In this case we write
$\nabla_{{\bf u}}  v({\bf u}):=\sum_{|I|\geq0}  v_{I,{\bf u}}'(0)e_I$.
\end{definition}

\begin{remark}\label{rem:dual}
For a componentwise-semi-open set $\mathcal{U}$, the dual $\mathcal{U}$-affine operator can therefore be written as
\[
\mathcal{B}g({\bf u}) =\langle R({\bf u}), \nabla_{{\bf u}} g({\bf u}) \rangle
\]
for $g \in \Dcal(\Bcal)$. As the right hand side makes sense for all componentwise-differentiable functions, $\Bcal$ can be extended to these functions.
\end{remark}

The following theorem provides now a second formulation of the affine transform formula, namely a classical transport version,  under slightly different assumptions. Note that in contrast to Corollary~\ref{cor1}, we here do not need to assume the existence of a solution to the Riccati ODE.

\begin{theorem}\label{thm:affinetrans}
Let $(\mathbb{X}_t)_{t\in[0,T]}$  be an $\mathcal{S}$-valued  $\Ucal$-affine process for some componentwise-semi-open set $\Ucal$. Let $\mathbf{x}_0$ be its initial value  and $\mathcal{L}$ the corresponding operator of $\Ucal$-affine type.
Set 
$$\rub(\mathbf{u}, \mathbf{x}):=\sup_{N\in \N}\Big|\sum_{|I|\leq N}\co \x I\langle e_I, R(\mathbf{u})\rangle \Big|,\qquad {\bf u}\in \Ucal, {\bf x}\in \Scal.$$
Fix $T >0$ and suppose that for each ${\bf u}\in \Ucal$ and $|I|\geq0$ it holds 
\begin{align}
\mathbb{E}[\sup_{ t\leq T}\rub(\mathbf{u}, \mathbb{X}_t) |\exp(\langle {\bf u},{\X_t}\rangle)|] &< \infty,\quad\text{and}\label{eqn2i}\\
\mathbb{E}[\sup_{ t\leq T}(1+|\langle e_I,\X_t\rangle|) |\exp(\langle {\bf u},{\X_t}\rangle)|]&<\infty.\label{eqn2ii}
\end{align}
Then  
$
\mathbb{E}[\exp(\langle \mathbf{u}, \mathbb{X}_T \rangle) ]= v(T,\mathbf{u}),
$ for all $\mathbf{u} \in \mathcal{U}$,
where $v(t,\mathbf{u})$ is a solution to the following transport equation 
\begin{align}\label{eq:transport}
\partial_t v(t,\mathbf{u})= \mathcal{B} v(t,\mathbf{u})
 = \langle R(\mathbf{u}), \nabla_{\bf u} v(t,\mathbf{u}) \rangle, \quad v(0,{\bf u}) =\exp (\langle \mathbf{u},\mathbf{x}_0\rangle ), \quad t \in [0,T].
\end{align}
\end{theorem}

\begin{proof}
Fix $\u\in\Ucal$. We know from the definition of an affine process that the process $M^{\mathbf{u}}$ given by
\begin{equation*}
dM^{\mathbf{u}}_t =d\exp( \langle \mathbf{u}, \mathbb{X}_t\rangle)- \mathcal{L}\exp( \langle \mathbf{u}, \mathbb{X}_t\rangle) dt, \qquad M^{\mathbf{u}}_0=0,
\end{equation*}
is a local martingale. Since conditions \eqref{eqn2i} and \eqref{eqn2ii} imply condition \eqref{eq:intcond1}, we can proceed as in the proof of Theorem~\ref{thm:affinetrans2} to conclude that 
$
\mathbb{E}[\sup_{t \leq T} |M^{\mathbf{u}}_t |] < \infty
$
and thus that $M^{\bf u}$ is a true martingale.
Therefore and due to Fubini-Tonelli's theorem, we have 
\begin{align}
\mathbb{E}\big[\exp(\langle \mathbf{u}, \mathbb{X}_t \rangle )\big]
&= \exp(\langle \mathbf{u}, \mathbf{x}_0 \rangle )+\int_0^t \mathbb{E}\Big[\langle R(\mathbf{u}), \mathbb{X}_s  \rangle \exp (\langle \mathbf{u}, \mathbb{X}_s \rangle  )\Big]  ds \notag\\
&= \exp(\langle \mathbf{u}, \mathbf{x}_0 \rangle )+\int_0^t \mathbb{E}\Big[\Big(
\lim_{N\to\infty}
\sum_{|I|\leq N}{\langle  R(\mathbf{u}),e_I\rangle}\langle e_I, \mathbb{X}_s  \rangle\Big)
\exp (\langle \mathbf{u}, \mathbb{X}_s \rangle  )\Big]  ds \notag\\
&=\exp(\langle \mathbf{u}, \mathbf{x}_0 \rangle )+\int_0^t \langle R(\mathbf{u}), \mathbb{E}\Big[\mathbb{X}_s \exp (\langle \mathbf{u}, \mathbb{X}_s \rangle  )\Big] \rangle ds \label{eq:Rint},
\end{align}
where the last equality follows from dominated convergence with \eqref{eqn2i}.
This justifies also the application of Fubini-Tonelli's theorem.
We now claim that the map 
\[
t \mapsto\langle R(\mathbf{u}), \mathbb{E}[\mathbb{X}_t \exp(\langle \mathbf{u}, \mathbb{X}_t \rangle)]\rangle 
= \mathbb{E}[\langle R(\mathbf{u}),\mathbb{X}_t \rangle \exp(\langle \mathbf{u}, \mathbb{X}_t \rangle)]
\]
is continuous. Let $t_n \to t$. Then, using \eqref{eqn2i} again we have by dominated convergence, 
\[
 \lim_{n \to \infty}\mathbb{E}[\langle R(\mathbf{u}),\mathbb{X}_{t_n} \rangle \exp(\langle \mathbf{u}, \mathbb{X}_{t_n} \rangle)]=\mathbb{E}[\lim_{n \to \infty}\langle R(\mathbf{u}),\mathbb{X}_{t_n} \rangle \exp(\langle \mathbf{u}, \mathbb{X}_{t_n} \rangle)].
\]
From the c\`adl\`ag assumption of  $(\langle R(\mathbf{u}),\mathbb{X}_{t} \rangle)_{t\in[0,T]}$  and $\langle \mathbf{u}, \mathbb{X}_{t}\rangle_{t\in[0,T]}$, it follows that\\
$
\langle R(\mathbf{u}),\mathbb{X}_{t_n} \rangle \exp(\langle \mathbf{u}, \mathbb{X}_{t_n} \rangle)$ converges to
$\langle R(\mathbf{u}),\mathbb{X}_{t} \rangle \exp(\langle \mathbf{u}, \mathbb{X}_{t} \rangle)$ if $t_n \downarrow t$  or to\\ $\langle R(\mathbf{u}),\mathbb{X}_{t-} \rangle \exp(\langle \mathbf{u}, \mathbb{X}_{t-} \rangle)$ if $t_n \uparrow t$. Since they are $\mathbb{P}$-a.s.~equal, we can conclude that
\[
\mathbb{E}[\lim_{n \to \infty}\langle R(\mathbf{u}),\mathbb{X}_{t_n} \rangle \exp(\langle \mathbf{u}, \mathbb{X}_{t_n} \rangle)]= \mathbb{E}[\langle R(\mathbf{u}),\mathbb{X}_{t} \rangle \exp(\langle \mathbf{u}, \mathbb{X}_{t} \rangle)],
\]
whence the asserted continuity follows. This together with \eqref{eq:Rint} implies 
\[ 
\partial_t\mathbb{E}[\exp(\langle \mathbf{u}, \mathbb{X}_t \rangle )]=
\langle R(\mathbf{u}), \mathbb{E}[\mathbb{X}_t \exp (\langle \mathbf{u}, \mathbb{X}_t \rangle  )] \rangle.
\]
Since $ \mathbf{u} \mapsto \exp(\langle \mathbf{u}, \mathbf{x} \rangle)$ is componentwise differentiable and $\mathbf{x} \exp(\langle \mathbf{u}, \mathbf{x} \rangle)=\nabla_{\mathbf{u}} \exp(\langle \mathbf{u}, \mathbf{x} \rangle)$,
it remains to argue that we can interchange expectation and differentiation, i.e.
$$\mathbb{E}[\nabla_{\bf u} \exp (\langle \mathbf{u}, \mathbb{X}_t \rangle  )]=\nabla_{\bf u}\mathbb{E}[ \exp (\langle \mathbf{u}, \mathbb{X}_t \rangle  )].$$
Indeed, consider the partial derivative with respect to  the component $\mathbf u_I$ of $\mathbf u$ corresponding to the multi-index $I$ and observe that   for each $\mathbf v\in \Ucal$ with $\v-\u=(\v_I-\u_I)e_I$, $|\v_I-\u_I|\leq \e$ and $\mathbf x \in \Scal$
\[
|\partial_{\mathbf u_I} \exp (\langle \mathbf{v}, \mathbf{x} \rangle)|\leq  |\langle e_I, \mathbf{x} \rangle || \exp (\langle \mathbf{v}, \mathbf{x} \rangle)|
\leq |\langle e_I, \mathbf{x} \rangle |(| \exp (\langle \mathbf{u}+\e e_I, \mathbf{x} \rangle)|+| \exp (\langle \mathbf{u}-\e e_I, \mathbf{x} \rangle)|).
\]
By \eqref{eqn2ii} and the Leibniz integral rule we can  conclude that
\[ 
\partial_t\mathbb{E}[\exp(\langle \mathbf{u}, \mathbb{X}_t \rangle )]=
\langle R(\mathbf{u}), \nabla_{\mathbf{u}}\mathbb{E}[\exp (\langle \mathbf{u}, \mathbb{X}_t \rangle  )] \rangle,
\]
proving that $\mathbb{E}[\exp(\langle \mathbf{u}, \mathbb{X}_t \rangle )]$ is a solution to \eqref{eq:transport}, which exists globally. Note that the equality with the dual operator follows from Remark~\ref{rem:dual}. This proves the assertion. 
\end{proof}

\subsection{The polynomial case}

In analogy to affine processes, polynomial processes are defined as follows.

\begin{definition}\label{def:primalpoly}
Fix $D(\Lcal):=\{\mathbf{x} \mapsto \langle \mathbf{u}, \mathbf{x} \rangle \colon  \mathbf{u} \in \mathcal{U}\}$ for some linear subset $\mathcal{U} \subseteq \Scal^*$  and let $\mathbf{\mathcal{L}}: D(\Lcal)\to M(\Scal)$ be a linear operator. 
We say that $\Lcal$ is
of \emph{$\Ucal$-polynomial type} if there exists 
a map $L: \mathcal{U} \to \mathcal{S}^*, \mathbf{u} \mapsto L(\mathbf{u})$
such that 
\begin{align}\label{eq:primalpoly}
\mathcal{L} (\langle \mathbf{u}, \fdot \rangle)(\mathbf{x})=  \langle L(\mathbf{u}), \mathbf{x} \rangle
\end{align}
for each ${\bf u}\in \Ucal$ and ${\bf x}\in \Scal$. Any solution $(\mathbb{X}_t)_{t\in[0,T]}$ of a martingale problem corresponding
to an operator $\mathcal{L}$  of $\Ucal$-polynomial type is called  $\mathcal{S}$-valued  \emph{$\Ucal$-polynomial process}.
\end{definition}

\begin{remark}\label{rem:poly}
\begin{enumerate}
\item 
Observe that in the literature of polynomial processes operators of polynomial type are typically called \emph{polynomial operators} and the operator $L$ appearing in Definition~\ref{def:primalpoly} is called \emph{1st-dual operator} (see e.g.~\cite{CS:21}). 
We here opted for the notion ``operator of \emph{polynomial type}''
 to be consistent with the terminology used for affine processes. Since in the present paper we are interested in spaces where higher order polynomials can be written as linear functions (see for instance Lemma~\ref{lem1}), we do not need to define higher order dual operators.
 
 \item  \label{rem:dualpoly}
If $\mathcal{L}$ is an operator of $\Ucal$-polynomial type, then we can also associate a dual operator. Indeed, the 
\emph{dual $\Ucal$-polynomial operator} is the linear operator $\Bcal$ 
  acting on $D(\Bcal):=\{g:\Ucal\to \R\colon g({\bf u})=  \langle \mathbf{u}, \mathbf{x} \rangle $ with $\mathbf{x}\in \mathcal{\Scal}\}$ such that
\[
\mathcal{B}(\langle \fdot,  \mathbf{x} \rangle)({\bf u)}:=\mathcal{L} (\langle  \mathbf{u} , \fdot\rangle )({\bf x}), \quad \mathbf{u} \in \Ucal.
\]
Thus $\mathcal{B}(\langle \fdot,  \mathbf{x} \rangle)({\bf u})=\langle L(\mathbf{u}), \mathbf{x} \rangle$ for each ${\bf u} \in \Ucal$ and ${\bf x}\in \Scal$. It can again be interpreted as the restriction of a transport operator which acts on componentwise-differentiable functions $g$ defined on componentwise-semi-open sets $\Ucal$ via
$
\mathcal{B}g(\u)= \langle L(\u), \nabla_{\u} g(\u) \rangle.
$
\end{enumerate}
\end{remark}

In the following proposition we  establish a relation between affine and polynomial processes. We will see (Proposition~\ref{prop2}), that for particular state spaces  affine and polynomial processes coincide. 
For general state spaces $\mathcal{S} \subseteq T((\mathbb{R}^d))$ this does not need to be the case.

\begin{proposition}\label{prop1_affpoly} 
Let $\mathbb{X}$  be an $\mathcal{S}$-valued continuous $\Ucal$-affine process 
with $R:\Ucal\to \Scal^*$ as defined in  \eqref{eq:primalaffine}. Assume that for each ${\bf u}\in \Ucal$ there is some $\lambda\notin \{0,1\}$ such that $\lambda {\bf u}\in \Ucal$. Then $\X$ is an $\mathcal{S}$-valued  $\Span(\Ucal)$-polynomial process 
with $L|_\Ucal:\Ucal\to \Scal^*$ given by
$$L({\bf u}):=\frac \lambda {\lambda-1}R({\bf u})-\frac 1 {\lambda(\lambda-1)}R(\lambda {\bf u}).$$ 
\end{proposition}

\begin{proof}
Define $D(\Lcal):=\{\mathbf{x} \mapsto \langle \mathbf{u}, \mathbf{x} \rangle \colon  \mathbf{u} \in \mathcal{U}\}$, $\Lcal:D(\Lcal)\to M (\Scal)$ as 
$\mathcal{L} (\langle \mathbf{u}, \fdot \rangle)(\mathbf{x})=  \langle L(\mathbf{u}), \mathbf{x} \rangle,
$
and note that it is of $\Span(\Ucal)$-polynomial type. To prove that $\X$ is a $\Span(\Ucal)$-polynomial process we verify that it solves the martingale problem for $\Lcal$. Properties~\ref{iti} and \ref{itii} are clear. For property~\ref{itiii}, note that since the process $(Y_t)_{t\in[0,T]}$ given by $Y_t:=\exp(\langle{\bf u},\X_t\rangle)$ is continuous, an application of the It\^o formula yields
$$dY_t^\lambda=\lambda Y_t^{\lambda-1}dY_t+\frac 1 2 \lambda(1-\lambda)Y_t^{\lambda-2}d[Y]_t.$$
On the other hand, since ${\bf u},\lambda {\bf u}\in\Ucal$ by definition of martingale problem we also have that
\begin{align*}
dY_t&=\exp(\langle \mathbf{ u}, \X_t \rangle) \langle R(\mathbf{u}), \X_t \rangle dt+d(\text{local martingale}),\\
dY_t^\lambda&=\exp(\langle \lambda\mathbf{ u}, \X_t \rangle) \langle R(\lambda\mathbf{u}), \X_t \rangle dt +d(\text{local martingale}).
\end{align*}
We can thus conclude that
\begin{align*}
d[Y]_t&=\frac 2 {\lambda(\lambda-1)Y^{\lambda-2}_t}\big(\exp(\langle \lambda{\bf u},\X_t\rangle)\langle R(\lambda{\bf u}),\X_t\rangle-\lambda Y_t^{\lambda-1}\exp(\langle {\bf u},\X_t\rangle)\langle R({\bf u}),\X_t\rangle\big)dt\\
&=\Big(\frac {2Y_t^2} {\lambda(\lambda-1)}\langle R(\lambda {\bf u}),\X_t\rangle -\frac {2 Y_t^2} {(\lambda-1)}\langle R({\bf u}),\X_t\rangle\Big)dt.
\end{align*}
An application of It\^o's formula yields 
$$d\log(Y_t)=\frac 1 {Y_t}\exp(\langle {\bf u},\X_t\rangle)\langle R({\bf u}),\X_t\rangle dt-\frac 1 {2Y_t^2}d[Y]_t+d\text{(local martingale)},$$
which can be rewritten as
$\langle {\bf u},\X_t\rangle-\langle {\bf u},{\bf x}_0\rangle
-\int_0^t\langle L({\bf u}),\X_s\rangle ds=\text{(local martingale)}.$
\end{proof}

The analog of the affine transform formula for polynomial processes is now given by the following moment formula.

\begin{theorem}\label{thm:momentpoly}
Fix $T >0 $ and let $(\mathbb{X}_t)_{t\in[0,T]}$  be an $\mathcal{S}$-valued  $\Ucal$-polynomial process with initial value $\mathbf{x}_0$ and  operator $\mathcal{L}$ of $\Ucal$-polynomial type as of Definition~\ref{def:primalpoly}.
Fix $\u\in \Ucal$ and suppose that  the extended tensor algebra valued linear ODE
\begin{equation}\label{eq:linODE}
 \mathbf{c}(t) = \u+\int_0^tL(\mathbf{c}(s)) ds
\end{equation}
admits an $\mathcal{U}$-valued  solution $({\bf c}(t))_{t\in[0,T]}$
such that $\int_0^T |L({\bf c}(t))|_\x dt<\infty$
for all ${\mathbf{x}} \in \mathcal{S}$.
If 
\begin{equation} \label{eq:intcond2}
\mathbb{E}[\sup_{s,t \leq T} |\langle \mathbf{c}(s), \mathbb{X}_t \rangle|] < \infty \qquad\text{and}\qquad
 \mathbb{E}[\sup_{s, t \leq T} | \langle L(\mathbf{c}(s)), \mathbb{X}_t \rangle|] < \infty,
\end{equation}
then it holds that
$
\mathbb{E}[\langle \mathbf{u}, \mathbb{X}_T \rangle ]= \langle \mathbf{c}(T) , \mathbf{x}_0 \rangle.
$
\end{theorem}

\begin{proof}
We verify the conditions of Lemma~\ref{lemA} for 
$$X(s,t)=\langle{\bf c}(s),\X_t\rangle\quad\text{and}\quad Y(s,t)= \langle L({\bf c}(s)), \mathbb{X}_t \rangle.$$
By definition of $\Ucal$-polynomial processes, the process $(N_t^{\mathbf{u}})_{t\in[0,T]}$ given by
\begin{equation}\label{eqn25}
dN_t^{\mathbf{u}}:=d\langle \mathbf{u}, \mathbb{X}_t\rangle - \langle L (\mathbf{u}), \mathbb{X}_t \rangle  dt,\qquad N_0^\u=0,
\end{equation}
is a local martingale.
Using that
\begin{align*}
\mathbb{E}[\sup_{t \leq T} |N^{\bf u}_t |] &\leq 2 \mathbb{E}[\sup_{t \leq T} |\langle {\bf u}, \mathbb{X}_t \rangle|]  +T \mathbb{E}[\sup_{t \leq T} |\langle L({\bf u}), \mathbb{X}_t\rangle |]\\
& \leq 2 \mathbb{E}[\sup_{s,t \leq T} |\langle {\bf c}(s), \mathbb{X}_t \rangle|] +T \mathbb{E}[\sup_{s, t \leq T} |\langle L({\bf c}(s)), \mathbb{X}_t\rangle|] < \infty
\end{align*}
holds by \eqref{eq:intcond2}, we can conclude that $(N^{{\bf u}}_t)_{t\in[0,T]}$ is a true martingale for all ${\bf u} \in \Ucal$.
Furthermore, for each ${\bf x}\in \Scal$ by \eqref{eq:linODE} and $\int_0^T |L({\bf c}(r))|_\x dt<\infty$ we have
\[
\langle \mathbf{c}(s), \mathbf{x} \rangle  - \langle {\bf u}, \mathbf{x} \rangle  - \int_0^s \langle L(\mathbf{c}(r)), \mathbf{x} \rangle  dr=0,
\]
which is thus a constant martingale.
Due to \eqref{eq:intcond2} all the assumptions of Lemma~\ref{lemA} are thus satisfied and the assertion follows.
\end{proof}

\begin{remark}\label{rem10}
Also in this case an inspection of the proof and Lemma~\ref{lemA} shows that \eqref{eq:intcond2} can be replaced by the following two properties.
\begin{enumerate}
\item The process $(N_t^{{\bf u}})_{t\in[0,T]}$ given by \eqref{eqn25}
is a true martingale for each ${\bf u}\in \Ucal$. 
\item $\int_0^T\int_0^T \E[|\langle L({\bf c}(s)), \X_t \rangle |]ds dt<\infty$.
\end{enumerate}
Similarly, conditions \eqref{eq:linODE} and $\int_0^T|L({\bf c}(s))|_\x ds<\infty$ can be replaced by assuming
$$
\langle{\bf c}(t),\x\rangle = \langle\u,\x\rangle+\int_0^t \langle L({\bf c}(s)),\x\rangle ds, \qquad \x\in \Scal.
$$
\end{remark}

\begin{remark}
Note that in the setting of polynomial processes we can also get analogous statements to Corollary~\ref{cor1} and Theorem~\ref{thm:affinetrans}.
\begin{enumerate}
\item The analog of Corollary~\ref{cor1} reads as follows.
Suppose that the conditions of Theorem~\ref{thm:momentpoly} are satisfied. Then $\mathbb{E}[\langle\u,\mathbb{X}_T\rangle]=v(T,\u)$ for some map $v:[0,T] \times \Ucal \to \mathbb{R}$ that satisfies
$\partial_t v(t, {\bf c}(s))=\partial_s v(t, {\bf c}(s))$ for all $s,t \in [0,T]$ and  $v(0, \u)=\langle \u, \x \rangle$.
This means that  $t \mapsto \mathbb{E}[\langle\u, \mathbb{X}_t\rangle]$ is a solution of a transport equation along $s \mapsto {\bf c}(s)$.
\item Under analogous conditions as of Theorem~\ref{thm:affinetrans}, we also get that $\mathbb{E}[\langle\u,\mathbb{X}_t\rangle]=v(t,\u)$, where $v(t, \u)$
satisfies
$
\partial_t v(t, \u)=\Bcal v(t, \u)=\langle L(\u), \nabla_{\u} v(t,\u) \rangle,
$
with $\Bcal$  given by Remark~\ref{rem:poly}~\ref{rem:dualpoly}. As in Remark~\ref{rem:bidual}, this just translates to
$
\partial_t \mathbb{E}[\langle\u,\mathbb{X}_t\rangle]=\langle L(\u),\mathbb{E}[\mathbb{X}_t]  \rangle
$
and -- if the set $\mathcal{U}$ is large enough -- to 
$
\partial_t \mathbb{E}[\mathbb{X}_t]= L'\mathbb{E}[\mathbb{X}_t],
$
where $L': \Scal \to \Scal$ is the adjoint operator of $L$, which in the terminology of \cite{CS:21}  corresponds again to the so-called bidual moment formula. 

\end{enumerate}
\end{remark}

\section{Signature SDEs and their signature process}\label{sec4}

We introduce signature SDEs as (It\^o) stochastic processes with characteristics depending in entire  way (made precise below) on its signature components. The first two subsections are devoted to a detailed analysis of the state space of signature processes and  functions thereon. 
Then we are led to our main results: such processes appear as linear projections of affine or polynomial processes taking values in subsets of so-called group-like elements. This class of processes is large (and to some extent universal) on the one hand; on the other hand affine or polynomial processes allow for very explicit descriptions of their laws in terms of Fourier-Laplace transforms.

\subsection{State space of signature processes, shuffle-compatible spaces of dual elements and entire functions}
\label{sigstatespace}
Fix some subset $S \subseteq \mathbb{R}^{d}$.
We start by introducing the state space of the signature process (see Definition~\ref{def2}) of a stochastic process $X$ satisfying $X_t-X_0\in S$ for each $t\in[0,T]$. 

\begin{definition}\label{def:Sgrouplike}
For a given subset $S\subseteq \R^d$ define
\begin{equation*}\label{eq:sigspace}
\begin{aligned}
\Scal(S)&:=\{ \mathbf x\in T((\R^d))\colon \pi_1(\x)\in S,\ 
 \co \x \emptyset=1,\text{ and } \co \x I\co \x J=\langle e_I\shuffle e_J,\mathbf x\rangle\ \quad  \forall |I|,|J|\geq0\}.
\end{aligned}
\end{equation*}
\end{definition}
The set $\Scal(S)$ corresponds to the set of group-like elements of $T((\R^d))$ whose first level lies in $S$. More precisely, define the commutator via $[e_i,e_j]:=e_i\otimes e_j-e_j\otimes e_i$
(which extends to all elements $\a \in T^{(N)}(\mathbb{R}^d)$ with $\a_\emptyset=0$) and consider
\[
g^N:=0 \oplus S \oplus [\mathbb{R}^d, \mathbb{R}^d]\oplus \cdots \oplus \underbrace{[\mathbb{R}^d, [\ldots,[\mathbb{R}^d, \mathbb{R}^d]]]}_{N-1 \text{ brackets}} \subseteq T^N(\mathbb{R}^d).
\]
Denoting by $G^N:=\{\sum_{k=0}^N\frac{\u^{\otimes k}}{k!}\colon \u \in g^N\}
$ its exponential image truncated at level $N$ we get
$$
\Scal(S)=\{{\x} \in T((\mathbb{R}^d)) : \x^{\leq N} \in G^N,\ \pi_1(\x)\in S\}.
$$
{Note that $G^N$ can also be defined via
\begin{equation}\label{eq:GN}
\begin{aligned}
G^N&:=\{ \mathbf x\in T^N(\R^d)\colon  \x_\emptyset=1,\text{ and } \co \x I\co \x J=\langle e_I\shuffle e_J,\mathbf x\rangle\ \quad  \forall |I|,|J|\geq0\}.
\end{aligned}
\end{equation}}
In the current setting we shall choose the semi-norms defined in \eqref{eqn5}  such that if $I\in\Pi_k$ then $\sigma(I)\in \Pi_k$ for each permutation $\sigma$.
Under this assumption the set $\Scal(S)^*$ is indeed closed with respect to the shuffle product (see Lemma~\ref{lem1}).
Recall that 
$\Pi=\{\Pi_1,\Pi_2,\Pi_3,\ldots\}$ denotes a partition of the set $\{I\colon|I|\geq 0\}$ and  all $\Pi_k$ just contain indices of the same length.

 \begin{definition}\label{def:shufflecomp}
We say that a partition $\Pi$ of multi-indices is \emph{shuffle-compatible} if for each $k_1,k_2$ there is a $k$ such that for each $I_1\in \Pi_{k_1},I_2\in \Pi_{k_2}$ we can write
$$ e_{I_1}\shuffle e_{I_2}=\sum_{I\in \Pi_k}\lambda_I e_{I}
$$
for some $\lambda_I\in \R$. In this case we write $k=k_1\shuffle k_2$.
\end{definition}

\begin{lemma} \label{lem1}
The following results hold.
\begin{enumerate}

\item\label{it11i} If the process $X-X_0$ takes value in $S$, the corresponding signature process $\mathbb{X}$ takes value in $\Scal(S)$.
\item\label{it11ii} Suppose that $\Pi$ is shuffle-compatible. Then  given ${\bf u},{\bf v}\in\Scal(S)^*$ it holds ${\bf u}\shuffle{\bf v}\in \Scal(S)^*$,
$|\u\shuffle\v|_\x\leq|\u|_\x|\v|_\x$, and $\langle \u,\x\rangle\langle\v,\x\rangle=\langle \u\shuffle\v,\x\rangle$ for each $\x\in \Scal(S)$.

\item\label{itlem3}
Suppose that $\Pi$ is shuffle-compatible.
Then given $\u\in \Scal(S)^*$ it holds $\exp(\shuffle \u)\in \Scal(S)^*$, $|\exp(\shuffle \u)|_\x\leq\exp(| \u|_\x)$, and 
$\exp(\langle \u,\x\rangle)
=\langle \exp(\shuffle \u),\x\rangle$ for each $\x\in \Scal(S)$.

\item\label{it11iii} {Suppose that $\Pi$ is given by Example~\ref{ex2}\ref{it3ii} and 
 $\mathbf u\in T((\R^d))+\text{i}T((\R^d))$ satisfies
\begin{equation}\label{eqn1}
 \mathbf u_{I}= \mathbf u_{I^{ord}}
 \end{equation}
 for every $|I|\geq0$. Setting
 $x^I:=x_{i_1}\cdots x_{i_{|I|}}$ we then get that 
 $\mathbf u\in \Scal(S)^*$ if and only if
$$g_{\mathbf u}({x}):=\sum_{I\colon I=I^{ord}}
{\bf u}_I\frac {x^I} {I!}$$
 converges absolutely for each $x\in S$. In this case,  for each ${\bf x}\in \Scal(S)$ it holds $$\langle \mathbf{u}, \mathbf{x}\rangle=g_{\mathbf u}(\pi_1(\x))\qquad \text{and}\qquad|\u|_\x=\sum_{I\colon I=I^{ord}}
|{\bf u}_I|\frac {|\pi_1(\x)^I|} {I!}.$$}
 
  \item\label{it11v} Suppose that $\Pi$ is shuffle-compatible, $\u$ satisfies  \eqref{eqn1},   and  $\u_\infty:=\sup_{I}|\u_I|<\infty$. Then $\u\in \Scal(S)^*$ and for each $\x\in \Scal(S)$ we get $|\u|_{\x}\leq \u_\infty|\x|_{\ell^1}$ for 
 $$|\x|_{\ell^1}:=\sum_{k=1}^\infty\Big|\sum_{I\in \Pi_k}\x_I\Big|$$
and
  $|\x|_{\ell^1}\leq \exp(|\x_{(1)}|+\cdots+|\x_{(d)}|).$
 \end{enumerate}
\end{lemma}

\begin{proof}
\ref{it11i} follows directly from Remark~\ref{shuffprop}\ref{th:shuffle}.
{ For \ref{it11ii} 
observe that  
\begin{align*}
    \langle e_{J},\u\shuffle \v\rangle&=\sum_{|J_1|,|J_2|\geq0}\co \u{J_1}\co \v{J_2}\langle e_{J},e_{J_1}\shuffle e_{J_2}\rangle,\\
   \co\x {J_1}\co \x {J_2}
    &=\langle e_{J_1}\shuffle e_{J_2},\x\rangle
    =\sum_{|J|\geq0}\co \x J\langle e_{J_1}\shuffle e_{J_2},e_{J}\rangle.
\end{align*}
For fixed $k$  we thus get
\begin{align*}
\sum_{I\in \Pi_k}\langle e_{I},\u\shuffle \v\rangle\co \x I
&=
\sum_{I\in \Pi_k}
\Big(\sum_{k_1,k_2=1}^\infty
\sum_{I_1\in \Pi_{k_1}}
\sum_{I_2\in \Pi_{k_2}}
\co \u {I_1}\co \v {I_2}
\langle e_{I},e_{I_1}\shuffle e_{I_2}\rangle1_{\{k={k_1\shuffle k_2}\}}\Big)
 \co \x I\\
 &=
 \sum_{k_1,k_2=1}^\infty
\sum_{I_1\in \Pi_{k_1}}
\sum_{I_2\in \Pi_{k_2}}
\co \u {I_1}\co \v {I_2}
\Big(\sum_{I\in \Pi_k}
 \co \x I\langle e_{I},e_{I_1}\shuffle e_{I_2}\rangle\Big)1_{\{k={k_1\shuffle k_2}\}}
\\
&=
\sum_{k_1,k_2=1}^\infty
\sum_{I_1\in \Pi_{k_1}}
\sum_{I_2\in \Pi_{k_2}}
\co \u {I_1}\co \v {I_2}\co \x {I_1}\co \x {I_2}1_{\{k={k_1\shuffle k_2}\}}.
\end{align*}

Since $\Pi$ is shuffle compatible for fixed $k_1,k_2$ it holds
$\sum_{k=1}^\infty1_{\{k_1\shuffle k_2=k\}}=1$
and hence
\begin{align*}
|\u\shuffle \v|_{\x}
&=\sum_{k=1}^\infty\Big|\sum_{I\in \Pi_k}\langle e_{I},\u\shuffle \v\rangle\co \x I\Big| \\
&\leq \sum_{k=1}^\infty
\sum_{k_1,k_2=1}^\infty\Big|
\sum_{I_1\in \Pi_{k_1}}
\sum_{I_2\in \Pi_{k_2}}
\co \u {I_1}\co \v {I_2}\co \x {I_1}\co \x {I_2}
\Big|1_{\{k_1\shuffle k_2=k\}}\\
&\leq 
\sum_{k_1,k_2=1}^\infty
\Big|
\sum_{I_1\in \Pi_{k_1}}
\sum_{I_2\in \Pi_{k_2}}
\co \u {I_1}\co \v {I_2}\co \x {I_1}\co \x {I_2}
\Big|=|\u|_{\x}| \v|_{\x}<\infty.
\end{align*}
The second part of the claim follows by Mertens' theorem for Cauchy products.

For \ref{itlem3}, observe that by \ref{it11ii} we have
    $$|\exp(\shuffle \u)|_\x
    \leq\sum_{n=1}^\infty\sum_{k=0}^\infty \frac 1 {k!}\Big|\sum_{I\in \Pi_n} \langle\u^{\shuffle k},e_I\rangle \x_I\Big|
    =\sum_{k=0}^\infty \frac 1 {k!} |\u^{\shuffle k}|_\x
    \leq \sum_{k=0}^\infty \frac 1 {k!} |\u|^k_\x
    \leq \exp(|\u|_\x).$$
The claim now follows by the dominated convergence theorem.

Concerning \ref{it11iii}, by Remark~\ref{rem1}\ref{it6iii} and Remark~\ref{shuffprop}\ref{th:shuffle}, for every ${\bf u} \in T((\R^d))+\text{i}T((\R^d))$ and ${\bf x}\in \Scal(S)$ such that \eqref{eqn1} holds  we have 
\begin{align*}
&|\u |_{\x}
=\sum_{I\colon I=I^{ord}}|\co \u I|\Big|\sum_{J\colon J^{ord}=I}\co \x J\Big|
=\sum_{I\colon I=I^{ord}}|\co \u I|\frac 1 {I! }|\pi_1(\x)^I|.
\end{align*}
Similarly, one also gets $\langle \mathbf{u}, \mathbf{x}\rangle=g_{\mathbf u}(\pi_1(\x))$ whenever the involved quantities are well-defined.}

For \ref{it11v} observe that the partition  given by Example~\ref{ex2}\ref{it3ii} is finer than each shuffle compatible partition. We can  thus assume without loss of generality that $\Pi$ is given by Example~\ref{ex2}\ref{it3ii}.
Set  $\v:=\sum_{|I|\geq0}e_I$ and note that  Lemma~\ref{lem1}\ref{it11iii} yields
$$
|\x|_{\ell^1}
=|\v|_\x
=\sum_{I\colon I=I^{ord}}
|\v_I|\frac {|x^I|} {I!}
=\sum_{I\colon I=I^{ord}}
\frac {|x^I|} {I!}
=\exp(|\x_{(1)}|+\cdots+|\x_{(d)}|).
$$
 As
$|\u|_\x
\leq \sum_{I\colon I=I^{ord}}|\co \u I||\sum_{ J\colon J^{ord}=I}\co \x J|
\leq|\u|_\infty|\x|_{\ell^1},
$
the claim follows.
\end{proof}

\begin{remark}
The above lemma shows that the choice of the partition considered in  Example~\ref{ex2}\ref{it3ii} leads to the usual definition of multivariate power series, a setting where -- due to the commutativity -- ordering does not matter.
\end{remark}

These considerations motivate the following notion of entire functions.

\begin{definition}\label{def:entire}
Consider the shuffle-compatible  partition
of Example~\ref{ex2}\ref{it3iii}
and for $\x \in \Scal(S)$ and $\u \in \Scal(S)^*$, let $\langle \u, \x \rangle$ be defined as in \eqref{eq:uxana}.
Then, we call the set of maps
$\{g: \Scal(S) \to \mathbb{R}, \,
\x \mapsto \langle \u, \x \rangle \, |\,  \u \in \Scal(S)^*\}$ \emph{entire maps on group-like elements}. 
\end{definition}

In the sequel we shall consider generic shuffle-compatible partitions and the corresponding pairing. The one of Example~\ref{ex2}\ref{it3iii} is however particularly interesting as it provides the largest space of entire functions, that  are also intrinsically connected with classical theory of real-analytic functions on path spaces (see \cite{CPST:24}).

\subsection{Signature SDEs as affine and polynomial processes}\label{sec:SIGSDE}

We consider $\Scal(S)\subseteq T((\mathbb{R}^{d}))$ as defined in \eqref{def:Sgrouplike} as state space of affine and polynomial processes and fix some shuffle-compatible partition $\Pi$. 
Before introducing signature SDEs, we start by a modification of Proposition~\ref{prop1_affpoly} showing that $\Scal(S)$-valued polynomial and affine processes coincide.
This is due to the fact that for each set $\Ucal\subseteq \Scal(S)^*$ which is closed under the shuffle product
and each ${\bf x}\in \Scal(S)$ it holds that the {set $\{\langle {\bf u}, \fdot \rangle \colon {\bf u}\in \Ucal\}$ forms an algebra over $\R$ and is thus closed with respect to sums, products and scalar multiplications.} This shows that polynomials on $\mathcal{S}$ appear as restrictions of linear maps.

\begin{proposition}\label{prop2}
Let $\mathcal{S}(S)$ be given by \eqref{eq:sigspace} for some subset $S \subseteq \mathbb{R}^d$.

    Let $\mathbb{X}$  be a continuous $\mathcal{S}(S)$-valued  $\Ucal$-affine process with 
$R:\Ucal\to \Scal(S)^*$ as defined in \eqref{eq:primalaffine}.
Then $\X$ is an $\mathcal{S}(S)$-valued  $\Vcal$-polynomial process 
with 
$\Vcal:=\Span\{\exp(\shuffle \u)\colon \u\in \Ucal\}$
and
$L:\Vcal\to \Scal(S)^*$ satisfying
$L(\exp(\shuffle \u))=\exp(\shuffle \u)\shuffle R(\u). $

Let $\mathbb{X}$  be a continuous $\mathcal{S}(S)$-valued  $\Ucal$-polynomial process with 
$L:\Ucal\to \Scal(S)^*$ as defined in \eqref{eq:primalpoly}.
Assume that for each ${\bf u}\in \Ucal$ we also have that $ {\bf u}\shuffle {\bf u}\in \Ucal$. Then $\X$ is an  $\mathcal{S}(S)$-valued  $\Ucal$-affine process 
with
$R:\Ucal\to \Scal(S)^*$ satisfying
$$R({\bf u})=L({\bf u})+\frac 12 L({\bf u}\shuffle{\bf u})-{\bf u}\shuffle L({\bf u}) .$$

\end{proposition}
\begin{proof}
    Define $D(\Lcal):=\{\mathbf{x} \mapsto \langle \exp(\shuffle\mathbf{u}), \mathbf{x} \rangle \colon  \mathbf{u} \in \mathcal{U}\}$, $\mathbf{\mathcal{L}}: D(\Lcal)\to M(\Scal)$ as
$$
\mathcal{L}( \langle \exp(\shuffle\mathbf{u}), \fdot \rangle)(\mathbf{x})= \langle \exp(\shuffle\mathbf{u})\shuffle R(\u), \mathbf{x} \rangle,
$$
and note that by
Lemma~\ref{lem1}\ref{it11ii}
it is of $\Vcal$-polynomial type.  Observe that by Lemma~\ref{lem1}\ref{itlem3} and Lemma~\ref{lem1}\ref{it11ii} we get
$\langle \exp(\shuffle\mathbf{u})\shuffle R(\u), \mathbf{x} \rangle=\exp(\langle \u,\x\rangle)\langle  R(\u), \mathbf{x} \rangle,$
for each $\u\in 
\Ucal$ and $\x\in \Scal(S)$. We can thus conclude that $\X$ is a $\Vcal$-polynomial process solving the martingale problem for $\Lcal$.

For the second part,
define $D(\Lcal):=\{\mathbf{x} \mapsto \exp(\langle \mathbf{u}, \mathbf{x} \rangle) \colon  \mathbf{u} \in \mathcal{U}\}$, $\mathbf{\mathcal{L}}: D(\Lcal)\to M(\Scal)$ as
$
\mathcal{L} \exp(\langle \mathbf{u}, \fdot \rangle)(\mathbf{x})= \exp(\langle \mathbf{u}, \mathbf{x} \rangle) \langle R(\mathbf{u}), \mathbf{x} \rangle,
$
and note that it is of $\Ucal$-affine type. To prove that $\X$ is an $\Ucal$-affine process we verify that it solves the martingale problem for $\Lcal$. Properties~\ref{iti} and \ref{itii} are clear. For property~\ref{itiii}, we apply the carr\'e du champ formula, i.e.,
\[
d [N^f]_t= \mathcal{L}f^2(\mathbb{X}_t) - 2 f(\mathbb{X}_t) \mathcal{L}f(\mathbb{X}_t) dt
\]
to obtain the quadratic variation of $N^f$ given by \eqref{eqnNnew}, when $t \mapsto \mathbb{X}_t$ is continuous. Take now $f(\mathbf{x})=\langle \mathbf{u}, \mathbf{x} \rangle$, set $Y_t:=\langle{\bf u},\X_t\rangle$ and note that $[Y]_t=[N^f]$. As $\langle {\bf u},\X_t\rangle^2=\langle {\bf u}\shuffle {\bf u},\X_t\rangle$, the quadratic variation of $Y$ is thus given by
\begin{align*}
d[Y]_t&=\big(\langle L({\bf u}\shuffle {\bf u}),\X_t\rangle-2\langle {\bf u},\X_t\rangle\langle L({\bf u}),\X_t\rangle\big)dt\\
&=\big(\langle L({\bf u}\shuffle {\bf u})-2 {\bf u}\shuffle L({\bf u}),\X_t\rangle\big)dt.
\end{align*}
An application of It\^o's formula yields 
$$d\exp(Y_t)=\exp(\langle {\bf u},\X_t\rangle)\langle L({\bf u}),\X_t\rangle dt+\frac {\exp(\langle {\bf u},\X_t\rangle)} {2}d[Y]_t+d\text{(local martingale)},$$
which yields the claim.
\end{proof}

 We now introduce \emph{signature SDEs} as the following generic class of diffusion models with state space $S$  
driven by some $d$ dimensional Brownian motion $B$ with initial condition $X_0=x_0$. The corresponding dynamics are  given by 
\begin{align}\label{eq:SigSDE}
dX_{t}={b}(\mathbb{X}_t) dt + \sqrt{{a}(\mathbb{X}_t)}dB_t,  \quad X_0=x_0,
\end{align}
where $(\mathbb{X}_t)_{t\in[0,T]}$ denotes the signature of $X$.  The coefficients ${b}: { \Scal}(S) \to \mathbb{R}^{d}$ and ${a}: {\Scal}(S) \to \mathbb{S}_+^{d}$  are supposed to satisfy 
$b_i({\bf x}) =\langle \el \b i, {\bf x} \rangle$ and  
${a}_{ij}({\bf x}) =\langle \el \a {ij}, {\bf x}\rangle$,
where $\mathbb{S}_+^{d} \subset \R^{d\times d}$ denotes the subset of positive semidefinite matrices and $\el \b i, \el \a {ij} \in \Scal(S)^*$.
This means that we deal here with coefficients that are entire functions of the signature in the sense of
Definition~\ref{def:entire}. 
Observe also that by the universal approximation theorem, choosing $ b$ and $ a$ appropriately allows us to approximate any continuous path functional arbitrarily well so that we deal here with a truly general class of diffusions whose  coefficients can depend on the whole path.

We do not treat here the question of existence and uniqueness of solutions to these equations in full generality but rather suppose that $X$ exists and solves \eqref{eq:SigSDE} in an appropriate sense on an appropriate state space $S$. 
{However, in the important case addressed in Proposition~\ref{prop:existence}, global existence and uniqueness can be established by relying on classical results.

\begin{proposition}\label{prop:existence}
Let $(X_t)_{t\in [0,T]}$ be given by \eqref{eq:SigSDE}. Let $N \in \mathbb{N}$, $f_i$ and  $f_{ij}$ be entire bounded maps on $\mathbb{R}$ and suppose that the characteristics $b$ and $a$ are of the following form $$b_i({\bf x})=f_i\bigg( \sum_{|I| \leq N} \alpha_I \langle e_I, \x  \rangle\bigg)\qquad \text{ and } \qquad a_{ij}({\bf x})=f_{ij}\bigg( \sum_{|I| \leq N} \alpha_I \langle e_I, \x  \rangle\bigg) 
$$
for all $i,j \in \{1, \ldots, d\}$ such that $a(\bf x)$ is positive semidefinite for each $\x\in G^N(\mathbb{R}^d)$. Then the signature SDE \eqref{eq:SigSDE} admits a global weak solution.
Moreover, if $\sqrt{a (\bf x)}$ is locally Lipschitz with respect to the Euclidean norm on $T^N(\mathbb{R}^d)$, a unique strong solution exists.
\end{proposition}

\begin{proof}
Consider the SDE system for $(\mathbb{X}^{\leq N}_t)_{t\in [0,T]}$. Since the characteristics only depend on the signature up to level $N$,  we obtain by the definition of the signature and the Stratonovich integral (similarly as in the proof of Lemma \ref{lem:sigdyn}) a standard SDE in the signature components up to level $N$ whose components satisfy
\begin{align} \label{eq:dyn}
d \langle e_I, \mathbb{X}_t \rangle=\Big(\frac 1 2\langle e_{I''},\mathbb{X}_t\rangle a_{i_{|I|-1}i_{|I|}}(\X_t) 
+\langle e_{I'},\mathbb{X}_t\rangle b_{i_{|I|}}(\mathbb{X}_t)\Big) dt
+\sum_{j=1}^{d} \langle e_{I'},\mathbb{X}_t \rangle \sigma (\mathbb{X}_t)_{i_{|I|}j} dB_t^j,
\end{align}
where $\sigma(\x)= \sqrt{a (\x)}$.
By continuously extending the characteristics to $T^N(\mathbb{R}^d)$ such that $a(\bf x)$ remains  bounded and positive semidefinite, we obtain global existence of the SDE system in $T^N(\mathbb{R}^d)$ by standard results for weak solutions. Indeed, the  drift and diffusion coefficients of $\mathbb{X}^{\leq N}$ are continuous on $T^N(\mathbb{R}^d)$ and of linear growth in the components of $\mathbb{X}^{\leq N}$ due to the boundedness assumptions on $b$ and $a$. This then yields a solution to \eqref{eq:SigSDE}, however with the modified coefficients extended to $T^N(\mathbb{R}^d)$. But since the truncated signature of this solution for $X$ satisfies exactly \eqref{eq:dyn}, we necessarily get a solution in $G^N(\mathbb{R}^d)$. This proves the first claim. 

The second claim concerning strong existence and uniqueness follows from the local Lipschitz property of  $\sqrt{a}$ and $b$ (where the latter automatically holds since it is a composition of an entire map with a linear function), which  translates to the whole system.
\end{proof}}

{
\begin{remark}Notice that the existence of a weak solution to the martingale problem taking values in \(G^N(\mathbb{R}^d)\) can also be verified directly by checking the positive maximum principle on \(G^N(\mathbb{R}^d)\). Indeed, due to \eqref{eq:GN}, $G^N$
can be described via \[
G^N(\mathbb{R}^d) = \{\mathbf{x} \in T^N(\mathbb{R}^d) \colon p(\mathbf{x}) = 0 \ \text{for all } p \in \mathcal{Q}\}
\]
for a suitable set of polynomials \(\mathcal{Q}\). 
Therefore verifying the positive maximum principle
 reduces to showing that the signature generator maps each \(p \in \mathcal{Q}\) to a function vanishing on \(G^N(\mathbb{R}^d)\) and that the diffusion coefficient,  supposed to be positive semidefinite, vanishes on \(G^N(\mathbb{R}^d)\) once multiplied with $\nabla p$ (see Theorem~5.3 in \cite{FL:16}). \end{remark}}

Even though the very general dynamics of \eqref{eq:SigSDE} have at first sight nothing in common with classical affine processes in finite dimension, they are \emph{always} the projection of an affine process, namely their signature process $(\mathbb{X}_t)_{t\in[0,T]}$. Before stating this result let us start with the following lemmas. Recall that $I'$ and $I''$ have been introduced in \eqref{eqn18} and the subsequent line, respectively.
\begin{lemma} \label{lem:sigdyn}
Let $X$ be given by \eqref{eq:SigSDE} and fix some multi-index $I$. Then it holds that
\begin{align*}
d\langle  e_{I}, \mathbb{X}_t\rangle
&=\langle
e_{I'}\shuffle \el \b {i_{|I|}}
+ \frac 1 2e_{I''}\shuffle \el \a {i_{{|I|}-1}i_{|I|}}
,\mathbb{X}_t\rangle dt
+\sum_{j=1}^{d} \langle e_{I'},\mathbb{X}_t \rangle{ \sigma (\mathbb{X}_t)}_{i_{|I|}j} dB_t^j,
\end{align*}
where $\sigma({\bf x})=\sqrt {a({\bf x})}$.
\end{lemma}
\begin{proof}
By definition of signature,  Stratonovich integral, and \eqref{eq:SigSDE}, the assertion follows by
\begin{align*}
d\langle  e_{I}, \mathbb{X}_t\rangle
&=\langle e_{I'},\mathbb{X}_t\rangle
\circ d X_t^{i_{|I|}}\\
&=\frac 1 2\langle e_{I''},\mathbb{X}_t\rangle d[X^{i_{|I|-1}},X^{i_{|I|}}]_t
 +\langle e_{I'},\mathbb{X}_t\rangle
 dX_t^{i_{|I|}}\\
&=\Big(\frac 1 2\langle e_{I''},\mathbb{X}_t\rangle \langle \el \a {i_{|I|-1}i_{|I|}}, \mathbb{X}_t\rangle 
+\langle e_{I'},\mathbb{X}_t\rangle \langle \el \b {i_{|I|}},\mathbb{X}_t\rangle\Big) dt
+\sum_{j=1}^{d} \langle e_{I'},\mathbb{X}_t \rangle \sigma (\mathbb{X}_t)_{i_{|I|}j} dB_t^j.
\end{align*}
 \end{proof}

\begin{remark}\label{rem:initialvalue}
    Observe that according to Definition~\ref{def2} the initial value of  $\X$ is given by $e_\emptyset$, as it corresponds to the signature of $X$ at time 0. This can be generalized considering the process $(\X^\x_t)_{t\in[0,T]}$ given by $\X_t^\x:=\x\otimes\X_t$ where $\X$ denotes the signature of the $S$-valued process $(X_t)_{t\in[0,T]}$ solving
    $$dX_{t}={b}(\X^\x_t) dt + {\sigma(\X^\x_t)}dB_t,  \quad X_0=x_0,$$
    for some $\x\in \Scal(S)$.
    Using that
    $\langle e_I,\X^\x_t\rangle=\sum_{e_{I_1}\otimes e_{I_2}=e_I}\langle e_{I_1},\x\rangle\langle e_{I_2}, \X_t\rangle$
    one can indeed show that $(\X^\x_t)_{t\in[0,T]}$ satisfies Lemma~\ref{lem:sigdyn}. This permits to see that $\X^\x$ solves the same martingale problem as $\X^{e_\emptyset}$ but with initial condition $\x$.
    
\end{remark}

\subsubsection{The affine case}\label{sec41}

We now show that the signature of signature SDEs as defined in \eqref{eq:SigSDE} is an $\Scal(S)$-valued affine process, for which we can express the map $R$ explicitly in terms of the coefficients of \eqref{eq:SigSDE}.
Recall that $\u\mapsto\u^{(1)}$ and $\u\mapsto\u^{(2)}$ denote the shifts defined in Section~\ref{sec22}.

\begin{theorem}\label{th:uniaffine}\phantomsection
\begin{enumerate}
\item\label{it7i} Fix $X$ as in \eqref{eq:SigSDE} and let $\X$ be its signature. Fix 
 $\Ucal\subseteq \Scal(S)^*$ such that  
 \begin{equation}\label{eqn27}
 \u^{(1)}\in(\Scal(S)^*)^{d}\qquad\text{and}\qquad\u^{(2)}\in (\Scal(S)^*)^{d\times d}
 \end{equation}
 for each $\u\in 
 \Ucal$ and consider the map $R:\mathcal{U} \to T((\R^d))$ given by
\begin{equation}\label{eq:R}
R(\mathbf u)={\bf b}^\top\shuffle {\bf u}^{(1)}+\frac 1 2 \tr\big({\bf a}\shuffle\big({\bf u}^{(2)}+{\bf u}^{(1)}\shuffle({\bf u}^{(1)})^\top\big)\big).
\end{equation}
Suppose  that 
$(\langle \u ,\X_t\rangle)_{t\in[0,T]}$ and $(\langle   R(\mathbf{u}), \X_t \rangle)_{t\in[0,T]}$   are continuous processes for all $\u\in \Ucal$ and 
$
(\sup_N|\langle \u^{\leq N},\X_t\rangle|)_{t\in[0,T]}
$ and $
(\sup_N|\langle  R(\mathbf{u}^{\leq N}), \X_t \rangle|)_{t\in[0,T]}$
are locally bounded. 
Then $\X$ is an $\Scal(S)$-valued $\Ucal$-affine process, whose linear operator is given by
 $$\mathcal{L} \exp(\langle \mathbf{u}, \fdot \rangle)(\mathbf{x}):= \exp(\langle \mathbf{u}, \mathbf{x} \rangle) \langle R(\mathbf{u}), \mathbf{x} \rangle.$$

\item\label{it7ii} If  the conditions of Theorem~\ref{thm:affinetrans2} are satisfied as well, then the affine transform formula
\[
\mathbb{E}[\exp(\langle \mathbf{u}, \mathbb{X}_T  \rangle)]= \exp(\langle {\bm \psi}(T), e_\emptyset  \rangle), 
\]
holds, 
where ${\bm \psi}$ is a  $\mathcal{U}$-valued solution of  the  Riccati ODE 
\begin{equation}\label{eqn111}
 {\bm \psi}(t)=\u+\int_0^tR({\bm \psi}(s))ds.
\end{equation}
\item\label{it7iii} If the conditions of Theorem~\ref{thm:affinetrans} are satisfied we also get
$
\mathbb{E}[\exp(\langle \mathbf{u}, \mathbb{X}_T \rangle) ]= v(T,\mathbf{u}),
$
where $v(t,\mathbf{u})$ is a solution to the following transport equation 
\begin{align*}
\partial_t v(t,\mathbf{u})= \mathcal{B} v(t,\mathbf{u})
 = \langle R(\mathbf{u}), \nabla_{\bf u} v(t,\mathbf{u}) \rangle, \quad v(0,{\bf u}) =\exp (\langle \mathbf{u},\mathbf{x}_0\rangle ), \quad t \in [0,T].
\end{align*}

\end{enumerate}
\end{theorem}

\begin{proof} 
Observe first that since $\Scal(S)^*$ is a linear set which is closed with respect to the shuffle product the pairing $\langle R(\u),\x\rangle$ is well defined for each $\u\in \Ucal$ and $\x\in \Scal(S)$.
We just prove part \ref{it7i}, since part \ref{it7ii} and part \ref{it7iii} then follow from Theorem~\ref{thm:affinetrans2} and Theorem~\ref{thm:affinetrans}, respectively. 
Since $\Lcal$ is of $\Ucal$-affine type, to conclude we just need to verify that $\X$ is a solution to the corresponding martingale problem. Properties~\ref{iti} and \ref{itii}  of the corresponding definition follow by assumption. 
 It\^o's formula and Lemma~\ref{lem:sigdyn} yield that for each ${\bf u}\in T(\R^d)$ the process $M^\u$ given by 
 \begin{equation}\label{eqn19}
dM_t^\u:=d\exp(\langle \mathbf{u}, \mathbb{X}_t \rangle)- \mathcal{L}\exp(\langle \mathbf{u}, \fdot \rangle)(\mathbb{X}_t) dt,\qquad M_0^{\u}=0
\end{equation}
is a local martingale and  property \ref{itiii} holds for each ${\bf u}\in T(\R^d)$.  In order to show that it holds on the whole $\Ucal$ we fix ${\bf u}\in \Ucal$ and let 
 $(\tau^n_{\bf u})_{n\in \N}$ be the localizing sequence making 
 $$
(\sup_N|\langle \u^{\leq N},\X_{t\land\tau_n}\rangle|)_{t\in[0,T]}
\qquad\text{and}\qquad
(\sup_N|\langle  R(\mathbf{u}^{\leq N}), \X_{t\land\tau_n} \rangle|)_{t\in[0,T]}$$
bounded. Fix $\x\in \Scal(S)$ and observe that \eqref{eqn27} yields
$$\lim_{N\to\infty}\langle (\u^{\leq N})^{(1)},\x\rangle=\lim_{N\to\infty}\sum_{n=0}^{N-1}\langle \pi_n(\u^{(1)}),\pi_n(\x)\rangle=\langle \u^{(1)},\x\rangle$$ 
and similarly $\lim_{N\to\infty}\langle (\u^{\leq N})^{(2)},\x\rangle=\langle \u^{(2)},\x\rangle$. Since $\langle R(\u^{\leq N}),\x\rangle$ can be written as
\begin{align*}
\langle{\bf b},\x\rangle^\top\langle (\u^{\leq N})^{(1)},\x\rangle+\frac 1 2 \tr\big(\langle{\bf a},\x\rangle\big(\langle{(\u^{\leq N})}^{(2)},\x\rangle+\langle{(\u^{\leq N})}^{(1)},\x\rangle\langle{(\u^{\leq N})}^{(1)},\x\rangle^\top\big)\big)
\end{align*}
this implies that $\lim_{N\to\infty}\langle R(\u^{\leq N}),\x\rangle=\langle R(\u),\x\rangle$, showing that 
$M^{\u^{\leq N}}_{\fdot\land \tau_n}$
converges to
$
M^\u_{\fdot\land \tau_n}$
in the bounded pointwise sense.
By the dominated convergence theorem we can conclude that $(M^\u_{t\land \tau_n})_{t\in[0,T]}$
is a bounded martingale proving the claim.
\end{proof}

 \begin{remark}
 The operator $R$ given by \eqref{eq:R} explicitly reads:
$$
R(\mathbf{u})= \sum_{|I|\geq0}
\Big(
e_{I'}\shuffle \el \b {i_{|I|}} 
+\frac 1 2 e_{I''}\shuffle \el \a {i_{|I|-1}i_{|I|}}
\Big){\bf u}_I+\frac 1 2 \sum_{|I|,|J|\geq0}\Big( e_{I'}\shuffle e_{J'}\shuffle \el \a {i_{|I|}i_{|J|}} \Big)
{\bf u}_I {\bf u}_J.
$$
Observe that by linearity of $\Scal(S)^*$ and Lemma~\ref{lem1}\ref{it11ii} we have that $R({\bf u})\in \Scal(S)^*$ for each ${\bf u}\in T(\R^d)$. 

Note that the condition that
$(\langle \u ,\X_t\rangle)_{t\in[0,T]}$ and $(\langle   R(\mathbf{u}), \X_t \rangle)_{t\in[0,T]}$ are continuous processes does not follow from the continuity of $ t \mapsto \langle e_I, \mathbb{X}_t \rangle$, as we do not necessarily have uniform convergence of $\x \mapsto \langle \u , \x \rangle$. In a follow up paper (see \cite{CPST:24}) we provide sufficient conditions to obtain locally uniform convergence everywhere.

Observe also that the local boundedness condition is satisfied if the processes
$(|\u|_{\X_t})_{t\in[0,T]}$, $(b(\X_t))_{t\in[0,T]}$, $(|\u^{(1)}|_{\X_t})_{t\in[0,T]}$, $(a(\X_t))_{t\in[0,T]}$, and $(|\u^{(2)}|_{\X_t})_{t\in[0,T]}$ 
are locally bounded (e.g.~if they are continuous) for each $\u\in \Ucal$. 
Inspecting the proof we can see that the same condition can be replaced by requiring the process $M^\u$ defined in \eqref{eqn19} to be a local 
martingale for each $\u\in \Ucal$. In view of Remark~\ref{rem2new} it could even be convenient to directly assume $M^\u$ to be a true martingale.

Finally, note that if we considered the process $\mathbb{X}^{\x}=\x \otimes \mathbb{X}$ as in 
Remark~\ref{rem:initialvalue}, the affine transform formula would read as
$
\mathbb{E}[\exp(\langle \mathbf{u}, \mathbb{X}_T^\x  \rangle)]= \exp(\langle {\bm \psi}(T), \x  \rangle).
$
 \end{remark}

The next corollary provides a generic methodology to obtain expressions 
 for the cumulant generating function of $X_T$  as defined in \eqref{eq:SigSDE}, whenever 
the corresponding Riccati ODE can be solved. Its proof is based on the observation that $u^\top X_T=u^\top x_0 +\langle u,\X_T\rangle$ 
for each $u\in \R^d$.

\begin{corollary}
Suppose that $X$ satisfies the conditions of Theorem~\ref{th:uniaffine}  for $X_0=x_0$  and some $\Ucal\subseteq \Scal(S)^*$. Fix $u\in \R^d$ such that 
$(u^\top x_0)e_\emptyset +u\in \Ucal.$
Then
$
\mathbb{E}[\exp( u^\top X_T)]= \exp({\langle \bm \psi}(T), e_{\emptyset}\rangle )
$
where ${\bm \psi}$ is an $\mathcal{U}$-valued  solution of the corresponding Riccati ODE \eqref{eqn111} with initial condition ${\bm \psi}(0)=(u^\top x_0)e_\emptyset +u$.
\end{corollary}

 \begin{example}
We show now how the presented techniques can be used to provide an alternative construction of the joint characteristic function of a two dimensional OU process
$$dX_t=\sigma dW_t, \quad X_0=x_0$$
 and its L\'evy area $L_t:=\frac 1 2 \int_0^t X^2_sdX^1_s-X^1_sdX^2_s$. 
 Here, $W$ denotes a standard Brownian motion, and $\sigma \in \mathbb{R}^{2 \times 2}$.
 This classical result is originally due to \cite{L:50} and was then generalized and studied by several authors, see e.g.~\cite{HS:83} and the recent paper \cite{lyons2024pde} as well as the references therein.
First observe that $X$ satisfies \eqref{eq:SigSDE} for
$
a(X_t)_{ij}=  a_{ij}=\langle a_{ij}e_\emptyset,\X_t\rangle,
$
where $a:=\sigma\sigma^\top$.
Set then
$\Ucal:=T^2(\R^2)+iT^2(\R^2)$
and observe that $\Ucal$ can be identified  with $\R^7+i\R^7$.
Next, note that the operator $R$ of \eqref{eq:R} can be rewritten as
\begin{equation*}
R(\mathbf{u})=
 \sum_{i_1i_2=1}^2
\frac 1 2  a_{i_{1}{i_{2}}} {\bf u}_{(i_1i_2)}e_\emptyset+\sum_{i_1i_2j_1j_2=1}^2a_{i_{2}{j_{2}}}
{\bf u}_{(i_1i_2)} {\bf u}_{(j_1j_2)}( e_{i_1}\otimes e_{j_1})
+\sum_{i_1j_1=1}^2 \frac 1 2a_{i_{1}{j_{1}}}
{\bf u}_{(i_1)} {\bf u}_{(j_1)}e_\emptyset,
\end{equation*}
and thus maps  $\Ucal$ to itself. By Theorem~\ref{th:uniaffine} we know that the  signature process $(\mathbb{X}_t)_{t\in[0,T]}$
of $X$ is an $\Ucal$-affine process taking values in $T((\mathbb{R}^d))$. In particular, letting ${\bm \psi}$ be a solution to the finite dimensional Riccati ODE
\[
\partial_t {\bm \psi}(t) = R({\bm \psi}(t)), \quad {\bm \psi}(0)=i\frac 1 2 \lambda(e_2\otimes e_1-e_1\otimes e_2)+\text{i}\gamma_1e_1+\text{i}\gamma_2e_2
\]
we get that
$\E[\exp(\text{i}\lambda L_t+\text{i} \gamma^\top X_t)]=\exp(\langle{\bm \psi}(T), e_{\emptyset}\rangle)\exp(\text{i}\gamma^\top x_0).$
Note that if $\sigma = \operatorname{Id}_2$, one can also consider the $4$-dimensional affine process $(X=(X^1,X^2), L, (X^1)^2 + (X^2)^2)$ to obtain the result.

The result generalizes to general $n$-dimensional OU processes
$dX_t=\theta (\mu-X_t) dt+\sigma dW_t.$
Also in this case one can indeed verify that setting $\Ucal:=T^2(\R^n)+iT^2(\R^n)$ we get that $R(\Ucal)\subseteq \Ucal$, for the operator $R$ given by \eqref{eq:R}. This in particular yields that the characteristic function of the level two signature process $\X^{\leq 2}$  of $X$ can be computed by solving a finite dimensional Riccati ODE.

 \end{example}

 \subsubsection{The polynomial case}
While Theorem~\ref{th:uniaffine} uses affine technology to compute $\mathbb{E}[\exp(\langle \mathbf{u}, \mathbb{X}_T  \rangle)]$ for generic diffusions of form \eqref{eq:SigSDE}, the following theorem establishes a formula for the \emph{expected signature} relying on the moment formula used for polynomial processes. Recall again that $\u\mapsto\u^{(1)}$ and $\u\mapsto\u^{(2)}$ denote the shifts defined in Section~\ref{sec22}.

\begin{theorem} \label{eq:expectedsig}
\begin{enumerate}
\item Fix $X$ as in \eqref{eq:SigSDE} and let $\X$ be its signature.
Fix 
 $\Ucal\subseteq \Scal(S)^*$ with  
 \begin{equation}\label{eqn23}
 \u^{(1)}\in(\Scal(S)^*)^{d}\qquad\text{and}\qquad\u^{(2)}\in (\Scal(S)^*)^{d\times d}
 \end{equation}
 for each $\u\in 
 \Ucal$ and consider the map $L:\mathcal{U} \to T((\R^d))$ given by
\begin{equation}\label{eq:L}
L(\mathbf u)={\bf b}^\top\shuffle {\bf u}^{(1)}+\frac 1 2 \tr({\bf a}\shuffle{\bf u}^{(2)}).
\end{equation}
Suppose also that 
$(\langle \u,\X_t\rangle)_{t\in[0,T]}$ and $
(\langle L(\u),\X_t\rangle)_{t\in[0,T]}$
are continuous processes for all $\u\in \Ucal$ and
$(\sup_N|\langle \u^{\leq N},\X_t\rangle|)_{t\in[0,T]}$ and $
(\sup_N|\langle L(\u^{\leq N}),\X_t\rangle|)_{t\in[0,T]}$
are locally bounded.
Then $\X$ is an $\Scal(S)$-valued $\Ucal$-polynomial process with respect to $\Ucal$, whose linear operator is given by
 $\mathcal{L} \langle \mathbf{u}, \fdot \rangle(\mathbf{x}):= \langle L(\mathbf{u}), \mathbf{x} \rangle$.
\item If  the conditions of Theorem~\ref{thm:momentpoly} are satisfied as well, then the moment formula
\[
\mathbb{E}[\langle \mathbf{u}, \mathbb{X}_T \rangle ]= \langle \mathbf{c}(T) , e_\emptyset \rangle,
\]
holds true, 
where ${\bf c}$ is a $\mathcal{U}$-valued  solution of  the  linear ODE 
$$
 \mathbf{c}(t) =\u+\int_0^t L(\mathbf{c}(s))ds.
$$

\end{enumerate}
\end{theorem}
\begin{proof}
The proof follows the proof of Theorem~\ref{th:uniaffine} step by step.
\end{proof}

\begin{remark}
 The operator $L$ given by \eqref{eq:L} explicitly reads:
\begin{equation*}
L \mathbf{u}= \sum_{|I|\geq0} (\frac 1 2 e_{I''}\shuffle \el \a {i_{|I|-1}i_{|I|}}
+e_{I'}\shuffle \el \b {i_{|I|}}){\mathbf{u}}_I.
\end{equation*}
Observe that by linearity of $\Scal(S)^*$ and Lemma~\ref{lem1}\ref{it11ii} we have that $L({\bf u})\subseteq \Scal(S)^*$ for each ${\bf u}\in T(\R^d)$.
 Inspecting the proof we can see again that the local boundedness condition can be replaced by requiring
$M_t^\u:=\langle \u ,\X_t\rangle
-\langle \u,\X_0\rangle
-\int_0^t\langle L(\mathbf{u}), \X_s \rangle ds$
to be a local martingale for each $\u\in \Ucal$. In view of Remark~\ref{rem2new} it could even be convenient to directly assume $M^\u$ to be a true martingale.
Observe also that if we considered the process $\mathbb{X}^{\x}=\x \otimes \mathbb{X}$ as in 
Remark~\ref{rem:initialvalue}, the moment formula would read as
$
\mathbb{E}[(\langle \mathbf{u}, \mathbb{X}_T^\x  \rangle)]= (\langle {\bf c}(T), \x  \rangle).$
 \end{remark}

\begin{example}\label{ex4}
Suppose now that $a_{ij}({\bf x})$ and $b_i({\bf x})$ do not depend on ${\bf x}_I$ for $|I|\geq 3$ and $|I|\geq 2$, respectively. Then the restriction of $L$ to $T^N(\R^d)$ satisfies $L(T^N(\R^d))\subseteq T^N(\R^d)$ for each $N\in \N$. This implies that the solution of \eqref{eq:linODE} can be written as a finite dimensional matrix exponential. Hence, provided enough integrability, the expectation of the corresponding truncated signature $\X^{\leq N}$ can also be computed via a finite dimensional matrix exponential. As a special case of this specification we have classical polynomial processes in the sense of \cite{CKT:12} or \cite{FL:16}. In this case  $a_{ij}$ is additionally supposed to be a quadratic function in the original state variables, i.e. in the signature truncated at level 1. This restriction  yields additional integrability properties which are very helpful for the application of Theorem~\ref{eq:expectedsig}. We refer to Section~4 in \cite{CGMS:23} for a complete analysis of this case.

As a specific example consider $\widehat S_t:=(t,S_t)$ where $S$ follows a Black-Scholes model under the risk neutral measure, i.e.~$dS_t=\sigma S_t dW_t$ for a Brownian motion $W$ and some $\sigma>0$. Then   $\widehat S$ satisfies \eqref{eq:SigSDE} for $b_i=1_{\{i=1\}}$ and
$a({\bf x})_{ij}=\langle\sigma^2(2e_2\otimes e_2+2S_0e_2+S_0^2),{\bf x}\rangle1_{\{i=j=2\}}.$
The corresponding operator $L$ is given by
\begin{align*}
L e_I&= 
\frac {\sigma^2} 2 \Big(2e_{I''}\shuffle (e_2\otimes e_2)
+2S_0e_{I''}\shuffle e_2
+S_0^2e_{I''}\Big)1_{\{i_{|I|}=i_{|I|-1}=2\}}
+e_{I'}1_{\{i_{|I|}=1\}}.
\end{align*}
By Remark~\ref{rem1}\ref{it6ii}, we can immediately see that $L$ is indeed mapping $T^N(\R^d)$ to itself and can thus be represented as a matrix $G$. The truncated expected signature can now be written in terms of $e^{tG}$ as in the classical setting. 
\end{example}

\subsubsection{Combining the approaches}
Working with affine processes we express semigroups by means of a solution of a Riccati ODE. This solution needs to be paired with $\x$ and composed with an exponential map.  In Section~\ref{sec41} these operations have been performed in this order leading to representations of the form $\exp(\langle \bm\psi(t),\x\rangle)$. The algebraic structure on $\Scal(S)$ permits however to gain generality by inverting the order and working with representations of the form
$\langle \exp(\shuffle \bm\psi(t)),\x\rangle.$

\begin{example}
Fix $\x\in\Scal(\R)$ and note that 
$$\x=\sum_{k=0}^\infty\frac{x^k}{k!}e_1^{\otimes k} =(1,x,\frac {x^2}{2!},\frac{x^3}{3!}\ldots),$$
for some $x\in \R$.
Suppose that $\u=\text{i}e_1$ and note that
$\exp(\langle \u,\x\rangle)=\exp(\text{i}x)$
and
$$\exp(\shuffle \u)=\sum_{k=0}^\infty \frac{\u^{\shuffle k}}{k!}
=\sum_{k=0}^\infty \text{i}^k \frac{e_1^{\shuffle k}}{k!}
=\sum_{k=0}^\infty  \text{i}^k {e_1^{\otimes k}}
=(1, \text{i},  \text{i}^2, \text{i}^3,\ldots).
$$
 We thus obtain that
$\langle \exp(\shuffle \u),\x\rangle
=\sum_{k=0}^\infty \text{i}^k\frac{x^k}{k!}{=\exp(\text{i}x)=\exp(\langle \u,\x\rangle).}$

Next, consider {$\u=(0,1,-1,2!,-3!,4!,\ldots)$} and note that $$|\u|_\x=\sum_{k=1}^\infty (k-1)! \frac {|x|^k}{k!}=\sum_{k=1}^\infty  \frac {|x|^k}{k}=\infty$$ for each $\x\in \Scal(\R)$  corresponding to $x\notin(-1,1)$. {Note that in this case
$ \u=\log(\shuffle (e_\emptyset +e_1))$,
implying that $\exp(\shuffle \u)= e_\emptyset+e_1$, hence
$$ | \exp(\shuffle \u)|_\x=| e_\emptyset+e_1|_\x=1+|x|<\infty,$$
and thus that $\langle\exp(\shuffle \u),
\x\rangle$ is well defined for each $\x\in \Scal(\R)$.}
The interpretation of this example can be found by considering $\x$ corresponding to $x\in(-1,1)$.  
For such $\x$ we indeed have $\langle \u,\x\rangle=\log(1+x)$.
\end{example}

Before stating the results we need a technical lemma of which we omit the proof.
\begin{lemma}\label{lem7}
For each $\u\in T((\R^d))$  it holds 
$$\exp(\shuffle \u)^{(1)}=\exp(\shuffle \u)\shuffle \u^{(1)},\qquad \exp(\shuffle \u)^{(2)}=\exp(\shuffle\u)\shuffle(\u^{(2)}+{\bf u}^{(1)}\shuffle({\bf u}^{(1)})^\top),$$ and hence
$L(\exp(\shuffle\u))=\exp(\shuffle\u)\shuffle R(\u)$
for $R$ and $L$ as in $\eqref{eq:R}$ and \eqref{eq:L}, respectively.
\end{lemma}

We provide now a generalization of the affine-transform formula. Observe that given $\u\in \Scal(S)^*$, by Lemma~\ref{lem1}\ref{itlem3} we recover the result of Theorem~\ref{th:uniaffine} with much weaker assumptions, since we require conditions on
$\langle \exp(\shuffle\u),\X\rangle$ rather than on $\langle \u,\X\rangle$.

\begin{theorem} \label{newthm}
Fix $X$ as in \eqref{eq:SigSDE} and let $\X$ denote its signature.
Fix 
 $\Ucal\subseteq T((\R^d))$ such that
 $\exp(\shuffle \u)\in\Scal(S)^*$,
 \begin{equation}\label{eqn23bis}
 \exp(\shuffle\u)\shuffle\u^{(1)}\in(\Scal(S)^*)^{d}\quad\text{and}\quad
 \exp(\shuffle\u)\shuffle(\u^{(2)}+{\bf u}^{(1)}\shuffle({\bf u}^{(1)})^\top)\in (\Scal(S)^*)^{d\times d}
 \end{equation}
 for each $\u\in 
 \Ucal$ and consider the map $R:\mathcal{U} \to T((\R^d))$ given by
\eqref{eq:R}.
Suppose also that 
$(\langle \exp(\shuffle\u),\X_t\rangle)_{t\in[0,T]}$ and $
(\langle \exp(\shuffle \u)\shuffle R(\u),\X_t\rangle)_{t\in[0,T]}$
are continuous and setting $L$ as in \eqref{eq:L} assume that
$(\sup_N|\langle \exp(\shuffle\u)^{\leq N},\X_t\rangle|)_{t\in[0,T]}$ and $
(\sup_N|\langle L(\exp(\shuffle\u)^{\leq N}),\X_t\rangle|)_{t\in[0,T]}$
are locally bounded for each $\u\in \Ucal$.
Fix $\u\in \Ucal$ and suppose then that the Riccati ODE
\begin{equation}\label{eqn2bis}
{\bm \psi}(t) =\u+\int_0^t R({\bm \psi}(s))ds,
\end{equation}
admits a  $\mathcal{U}$-valued  solution $({\bm \psi}(t))_{t\in [0,T]}$ such that
$\int_0^T|\exp(\shuffle \bm\psi(s))\shuffle R(\bm\psi(s))|_\x ds<\infty$
for all $\mathbf{x} \in \Scal(S)$.
Moreover, assume that
\begin{equation} \label{eq:intcond1bis}
\begin{aligned}
&\mathbb{E}[\sup_{s,t \leq T} |\langle(\exp(\shuffle {\bm \psi}(s)), \mathbb{X}_t \rangle)|] < \infty,\quad\text{ and} \\
 &\mathbb{E}[\sup_{s, t \leq T} |\langle \exp(\shuffle \bm \psi(s))\shuffle R({\bm \psi}(s)), \mathbb{X}_t\rangle |] < \infty.
\end{aligned}
\end{equation}
Then,  it holds
$
\mathbb{E}[\langle \exp(\shuffle\mathbf{u}), \mathbb{X}_T \rangle ]=  \exp(\langle{\bm \psi}(T) , e_\emptyset \rangle).
$
\end{theorem}
\begin{proof}
Let $({\bm \psi}(t))_{t\in [0,T]}$ be a  $\mathcal{U}$-valued  solution  of \eqref{eqn2bis}. For the sake of exposition we present this proof using the differential form of \eqref{eqn2bis}.
Using the `$\overline{\u}^{\shuffle k}$-notation' as introduced in Definition~\ref{def:expshuffle}, 
observe that 
\begin{align*}
    \partial_t\langle\overline{\bm\psi(t)}^{\shuffle k},e_I\rangle
&=\partial_t\Big(\sum_{|I_1|,\ldots,|I_k|\geq0}
\langle \overline{\bm\psi(t)},e_{I_1}\rangle
\cdots
\langle\overline{\bm\psi(t)},e_{I_k}\rangle
\langle e_{I_1}\shuffle\cdots\shuffle e_{I_k},e_I\rangle\Big)\\
&=k\sum_{|I_1|,\ldots,|I_k|\geq0}
\langle \overline{\bm\psi(t)},e_{I_1}\rangle
\cdots
\langle \overline{\bm\psi(t)},e_{I_{k-1}}\rangle
\langle \overline{R(\bm\psi(t))},e_{I_k}\rangle
\langle e_{I_1}\shuffle\cdots\shuffle e_{I_k},e_I\rangle\\
&=k\langle\overline{\bm\psi(t)}^{\shuffle(k-1)}\shuffle \overline{R(\bm\psi(t))},e_I\rangle,
\end{align*}
where summation and differentiation can be exchanged since the sum has finitely many nonzero elements.
This in particular yields
\begin{align*}
    \partial_t\langle\exp(\shuffle\overline{\bm\psi(t)}),e_I\rangle
&=\partial_t\langle\sum_{k=0}^\infty\frac{\overline{\bm\psi(t)}^{\shuffle k}}{k!},e_I\rangle
=\sum_{k=0}^\infty\partial_t\langle\frac{\overline{\bm\psi(t)}^{\shuffle k}}{k!},e_I\rangle\\
&=\sum_{k=1}^\infty\langle\frac{\overline{\bm\psi(t)}^{\shuffle (k-1)}}{(k-1)!}\shuffle \overline{R(\bm\psi(t))},e_I\rangle
=\langle\exp(\shuffle\overline{\bm\psi(t)})\shuffle \overline{R(\bm\psi(t))},e_I\rangle,
\end{align*}
where again we use that the appearing sum has finitely many nonzero elements. Since 
$$\langle \exp(\shuffle\bm\psi(t)),e_I\rangle
=\langle\exp(\shuffle\overline{\bm\psi(t)}),e_I\rangle \exp(\bm\psi(t)_\emptyset)$$
and 
$\partial_t \exp(\bm\psi(t)_\emptyset)=\exp(\bm\psi(t)_\emptyset) (R(\bm\psi(t)))_\emptyset$
we get
\begin{align*}
    \partial_t\langle\exp(\shuffle\bm\psi(t)),e_I\rangle
&=
\langle\exp(\shuffle\overline{\bm\psi(t)})\shuffle \overline{R(\bm\psi(t))},e_I\rangle
\exp(\bm\psi(t)_\emptyset) \\
&\qquad+
\langle\exp(\shuffle\overline{\bm\psi(t)}),e_I\rangle
\exp(\bm\psi(t)_\emptyset) (R(\bm\psi(t)))_\emptyset\\
&=\langle\exp(\shuffle\bm\psi(t))\shuffle (\overline{R(\bm\psi(t))}+(R(\bm\psi(t)))_\emptyset e_\emptyset),e_I\rangle\\
&=\langle\exp(\shuffle\bm\psi(t))\shuffle R(\bm\psi(t)),e_I\rangle.
\end{align*}
The result now follows directly from Theorem~\ref{eq:expectedsig} and Lemma~\ref{lem7}.
\end{proof}

Finally, we explain how to move from the moment formula to the generalized affine-transform formula provided in the last theorem.
\begin{theorem}\label{eq:combining}
Fix $\u\in T((\R^d))+ i T((\R^d))$ and suppose that $({\bf c}(t))_{t\in[0,T]}$ satisfies the conditions of Theorem~\ref{eq:expectedsig} for ${\bf c}(0)=\exp(\shuffle \u)$. Assume that ${\bf c}(t)_\emptyset\neq0$ 
for each $t\in[0,T]$ and set $$\bm\psi(t)_\emptyset=\u_\emptyset+\int_0^t\frac{L{\bf c}(s)_\emptyset}{{\bf c}(s)_\emptyset}ds$$
and 
$\overline{\bm\psi(t)}:=\log(\shuffle {\bf d}(t))$ for ${\bf d}(t):={{\bf c}(t)}/{{\bf c}(t)_\emptyset}$.
Then $\bm \psi$ satisfies the conditions of Theorem~\ref{newthm}.
\end{theorem}
\begin{proof}
First observe that
$\exp(\bm\psi(t)_\emptyset)={\bf c}(t)_\emptyset$ and  Lemma~\ref{lem6} yields
$
\exp(\shuffle \overline{\bm\psi(t)})={\bf d}(t).
$
This implies that
\begin{equation}\label{eqn24}
\exp(\bm\psi(t)_\emptyset)\exp(\shuffle \overline{\bm\psi(t)})={\bf c}(t)
\end{equation}
which together with Lemma~\ref{lem7}
yield
$$\bm\psi(t)_\emptyset
=\bm\psi(0)_\emptyset+\int_0^t \frac{\exp(\bm\psi(s)_\emptyset)R(\bm \psi(s))_\emptyset}{\exp(\bm\psi(s)_\emptyset)} ds
=\bm\psi(0)_\emptyset+\int_0^t R(\bm \psi(s))_\emptyset ds.
$$
Finally, proceeding as in the proof of Theorem~\ref{newthm}, for each $|I|>0$ we can conclude that
\begin{align*}
    \partial_t\bm\psi(t)_I
    &=
 \sum_{k=1}^\infty\frac{(-1)^{k+1} }{k}\partial_t\langle\overline{{\bf d }(t)}^{\shuffle k},e_I\rangle\\
& = \sum_{k=1}^\infty{(-1)^{k+1} }\langle\overline{{\bf d }(t)}^{\shuffle (k-1)}\shuffle\Big(-\overline{{\bf d }(t)}R(\bm\psi(t))_\emptyset
+
\overline{{\bf d}(t)\shuffle R(\bm\psi(t))}
\Big),e_I\rangle\\
& =\langle \sum_{k=1}^\infty{(-1)^{k+1} }
\overline{{\bf d }(t)}^{\shuffle (k-1)}\shuffle
\Big((e_\emptyset+\overline{{\bf d }(t)})
\shuffle
\overline{R(\bm\psi(t))}
\Big),e_I\rangle\\
& =\langle \Big(\sum_{k=0}^\infty{(-1)^{k} }
\overline{{\bf d }(t)}^{\shuffle k}
-\sum_{k=1}^\infty{(-1)^{k} }
\overline{{\bf d }(t)}^{\shuffle k}\Big)
\shuffle
\overline{R(\bm\psi(t))},e_I\rangle\\
&=
R(\bm\psi(t))_I.
\end{align*}
The claim now follows by Lemma~\ref{lem7} and \eqref{eqn24}.
\end{proof}

\section{Power sequences of one dimensional diffusions with  real-analytic characteristics}\label{sec5}

In this section we consider the case  $d=1$ and a process $\X$  given as the sequence of powers of some stochastic process $X$, i.e.
\begin{align}\label{eq:X1d}
\X=(1, X, X^2, \ldots, X^n, \ldots ).
\end{align}
In order to simplify the notation for $\a\in T((\R))$ we set $\a_n:=\langle e_1^{\otimes n},\a\rangle$. Moreover, since in this context there is no risk of misinterpretation, we use upper indices to denote  powers, and not to denote semimartingale components.

Fix $S \subseteq \mathbb{R}$ and consider a one-dimensional diffusion process $X$ on $S$ given by
\begin{align}\label{eq:SDE1d_new}
dX_t=b(X_t)dt +  \sqrt{a(X_t)} dB_t, \quad X_0=x_0,
\end{align}
where 
\begin{equation}\label{eqn4new}
b(x)=\sum_{n=0}^{\infty} {\bf b}_n x^n\qquad\text{and}\qquad a(x)=\sum_{n=0}^{\infty} {\bf a}_n x^n,\qquad x\in S,
\end{equation}
and $\mathbf{b}$, $\mathbf{a}\in T((\mathbb{R}))$. Assume that $|b(x)|,|a(x)|<\infty$ for each $x\in S+B_\e(0)$, for some $\e>0$.
We then denote by $\Acal:C^2(\R)\to C(\R)$ the corresponding extended generator
$$\Acal f:=bf'+\frac 1 2 af''.$$
Following Lemma~\ref{lem1}, observe that letting
$$\Ss:=\{ \mathbf x\in T((\R))\colon {\bf x}=(1,x,x^2,\ldots),\ x\in S\}$$
we get that $\mathbf u\in \Tcal(S)^*$ if and only if
 $$\gg_{\bf u}(x):=\sum_{n=0}^\infty {\bf u}_nx^n$$
 converges absolutely for each $x\in S$. If this is the case, then we set $\langle \mathbf{u}, \mathbf{x}\rangle:=\gg_{\mathbf u}(\x_1)$ for each ${\bf x}\in \Tcal(S)$. Let then $\star$ denote the discrete convolution, i.e.
$$({\bf u}\star{\bf v})_n:=\sum_{n_1+n_2=n}{\bf u}_{n_1} {\bf v}_{n_2},\qquad {\bf u},{\bf v}\in T((\R))$$
and observe that power series analysis yields ${\bf u}\star{\bf v}\in \Tcal(S)^*$ and $\gg_{\mathbf u}(x)\gg_{\mathbf v}(x)=\gg_{{\bf u}\star {\bf v}}(x)$ for each ${\bf u},{\bf v}\in \Tcal(S)^*$.

Finally, we often consider the component-wise open set 
\begin{equation}\label{eqn20}
   \Vcal:=\{ \mathbf u\in T((\R))\colon |\gg_{\mathbf u}(x)|<\infty\text{ for all }
x\in S+B_\e(0),\ \e>0\}
\end{equation}
and for each $\u$ in $\Vcal$ we set \begin{equation}\label{eqn21}
\begin{aligned}
  \u^{[1]}&:=(\u_1,2 \u_2, 3\u_3,\ldots,k\u_k,\ldots) \\
  \u^{[2]}&:=(2\u_2,6\u_3,12 \u_4,\ldots, k(k-1)\u_k,\ldots).
\end{aligned}
\end{equation}
Observe that these elements of the extended tensor algebra correspond  to $\u^{(1)}$ and $\u^{(2)}$ when the shuffle product is replaced by the discrete convolution.
\begin{remark}\label{rem5}
By the properties of power series in their convergence domain we have that $\Vcal\subseteq \Tcal(S)^*$. Moreover,  for each $\u\in \Vcal$ the map $h_\u$ is smooth on $S$, both $\u^{[1]}$ and $\u^{[2]}$ are elements of $ \Vcal$ and thus of $\Tcal(S)^*$, and it holds
$$h'_\u=h_{\u^{[1]}}\qquad\text{and}\qquad h''_\u=h_{\u^{[2]}}.$$

Note also that it can be of interest to consider power series expansions around points different than 0. The whole theory of this section can then be applied considering the shifted process $X^\alpha:=X-\alpha$ for some $\alpha\in \R$. This is for instance used to apply the results of  Theorem~\ref{lem4new} to the  shifted Jacobi diffusion in Section~\ref{ex1}.

Finally, note  that
the processes analyzed in this section do not correspond to signatures of one dimensional processes, which would take  values in 
$$\Scal(S-x_0)=\{ \mathbf x\in T((\R))\colon {\bf x}=(1,x-x_0,\frac{(x-x_0)^2}2,\frac{(x-x_0)^3}{3!},\ldots),\ x\in S\}.$$
Moreover, since property \eqref{eqn1} is always satisfied for $d=1$, for each ${\bf x}\in \Scal(S)$ and ${\bf u}\in \Scal(S)^*$, by Lemma~\ref{lem1} we can write
$\langle \mathbf{u}, \mathbf{x}\rangle=g_{\mathbf u}(\x_1)$ for
$
g_{\mathbf u}(x-x_0)=\sum_{n=0}^{\infty} \u_n \frac {(x-x_0)^n}{n!}.
$
With this parametrization, the assertions of Remark~\ref{rem5} read
$g'_\u=g_{\u^{(1)}}$ and $g''_\u=g_{\u^{(2)}}.$
\end{remark}

\subsection{The affine case}

Similarly to Section~\ref{sec41}, we
show that $\X$ given by \eqref{eq:X1d} for a diffusion $X$ of the form \eqref{eq:SDE1d_new}
is a $\Tcal(S)$-valued affine process, for which we can express the map $R$ explicitly in terms of the coefficients of \eqref{eq:SDE1d_new}. We start here by formulating the affine transform formula via the transport equation, similarly as in Theorem~\ref{thm:affinetrans}.
Recall also that  $\nabla$ denotes the component-wise derivative in the sense of Definition~\ref{def1}.

\begin{theorem}\label{thm1}
Fix $T>0$ and  $\Vcal$ as in \eqref{eqn20}. Then
the process
$\Xx:=(1, X, X^2, \ldots, X^n, \ldots)$
is $\Vcal$-affine with respect to the operator $R:\Vcal\to T((\R))$ given by 
\begin{equation}\label{eqn30}
R(\mathbf u)={\bf b}\star {\bf u}^{[1]}+\frac 1 2 {\bf a}\star\big({\bf u}^{[2]}+{\bf u}^{[1]}\star{\bf u}^{[1]}\big).
\end{equation}
Fix then a component-wise semi-open set $\Ucal\subseteq\Vcal$ and assume  that for all $n\in \N_0$ and ${\bf u}\in \Ucal$ 
\begin{equation}\label{eqn3}
\E[\sup_{t\leq T}|X_t|^n{ \exp(\gg_{\bf u}(X_t))}] < \infty\quad \text{and}\quad \E\Big[\sup_{t\leq T, n\in \N} |\gg_{\tru {R(\u)} n}(X_t)|\exp(\gg_{\bf u}(X_t))\Big] < \infty.
\end{equation}
Then
$
\mathbb{E} [\exp(\gg_{\bf u} (X_T))]=v(T,\mathbf{u}),
$
where $v(t,{\bf u})$ is a solution to the transport equation
\begin{equation}\label{eqn12}
\partial_t v(t,\mathbf{u}) 
=\langle R({\bf u}),\nabla_{\bf u}v(t,\mathbf{u})\rangle,
 \quad v(0,\mathbf{u})
=\exp(\gg_{\bf u} (x_0)).
\end{equation}
\end{theorem}

\begin{proof}
Observe first that for each ${\bf u}\in \Ucal$ by definition of $\X$ it holds
$$\Lcal(\exp(\langle {\bf u},\fdot\rangle))(\X_t)=\Acal(\exp(\gg_{\bf u}))(X_t).$$
Since $\Tcal(S)^*$ is closed with respect to $\star$,  Remark~\ref{rem5} yields $R(\u)\in \Tcal(S)^*$ for each $\u\in \Vcal$ and $\gg_{R(\mathbf u)}=b\gg_{\bf u}'+\frac 1 2 a(\gg_{\bf u}''+(\gg_{\bf u}')^2)$. 
To show that $\Xx$ is a $\Vcal$-affine process we then need to verify that $\Xx$ solves the martingale problem for 
$$\mathcal{L} \exp(\langle \mathbf{u}, \fdot \rangle)({\bf x}):= \exp(\langle \mathbf{u}, \mathbf{x} \rangle) \langle R(\mathbf{u}), \mathbf{x} \rangle.$$
Note that continuity of $X$ yields property \ref{itii} and since $\Xx_0=(1,x_0, x_0^2, \ldots,x_0^n, \ldots)\in \Ss$
property \ref{iti} is satisfied as well.
 For property \ref{itiii} we already showed that $\gg_{{\bf u}}\in C^2(S)$. Since $\mathcal{L}( \exp(\langle \mathbf{u}, \fdot \rangle))(\Xx_t)=\Acal(\exp(\gg_{{\bf u}}))(X_t)$ for each ${\bf u} \in \Vcal$, the claim then follows from It\^o's formula.
Since condition  \eqref{eqn3} implies conditions \eqref{eqn2i} and \eqref{eqn2ii}, the claimed representation follows by Theorem~\ref{thm:affinetrans}.
\end{proof}

We also have a version involving the Riccati ODE, similarly as in Theorem~\ref{thm:affinetrans2}.

\begin{theorem}\label{thm2}
Let $T >0$, $\Vcal$ be as in \eqref{eqn20},  and fix $R$  as in Theorem~\ref{thm1} and  ${\bf u}\in \Vcal$. 
Suppose that the sequence valued Riccati ODE
\begin{equation}\label{eqn7aff}
 {\bm \psi}(t) =\u+\int_0^t R({\bm \psi}(s))ds,
\end{equation}
admits a  $\Vcal$-valued solution $({\bm \psi}(t))_{t\in[0,T]}$ such that $\int_0^T|R(\bm\psi(s))|_\x ds<\infty$ for each $\x\in \Ss$.
Suppose also that
\begin{equation}\label{eqn14}
\begin{aligned}
&\mathbb{E}[\sup_{s,t \leq T}|\exp(\gg_{{\bm \psi}(t)}(X_s))|] < \infty, \qquad
 \mathbb{E}[\sup_{s,t \leq T}|\exp(\gg_{{\bm \psi}(t)}(X_s))||\gg_{R({\bm \psi}(t))}(X_s)|] < \infty.
\end{aligned}
\end{equation}
Then $\E[\exp(\gg_{{\mathbf u}}(X_T))]=\exp(\gg_{{\bm \psi}(T)}(x_0))$, i.e.
$
\mathbb{E}[\exp(\sum_{n=0}^{\infty} {\bf u}_n  X_T^n)]= \exp(\sum_{n=0}^{\infty} {\bm \psi}_n(T)x_0^n).
$
\end{theorem}

\begin{proof}
Recall that $\Xx$ is a $\Ss$-valued $\Vcal$-affine process by Theorem~\ref{thm1} and  ${\bf u}\in \Vcal$ if and only if $|\gg_{\mathbf u}(x)|<\infty$ for all $x\in S+B_\e(0)$ for some $\e>0$. The claim
follows from Theorem~\ref{thm:affinetrans2}.
\end{proof}

\begin{remark}\label{rem2}
Observe that Remark~\ref{rem2new} in this particular framework states that condition \eqref{eqn14} can be replaced by the following two properties.
\begin{enumerate}
\item\label{it4i} The process $(M_t^{{\bm \psi}(s)})_{t\geq0}$ given by 
\begin{equation}\label{eqn11}
M_t^{{\bm \psi}(s)}:=\exp(\gg_{{\bm \psi}(s)}(X_t))-\exp(\gg_{{\bm \psi}(s)}(x_0))-\int_0^t \mathcal{A}\left(\exp(\gg_{{\bm \psi}(s)})\right)(X_r) dr
\end{equation}
is a true martingale for each $s\in[0,T]$. Since bounded local martingales are martingales, this condition is always satisfied when $S$ is a compact set.
\item\label{it4ii} 
$\int_0^T\int_0^T \E[|\Acal(\exp(\gg_{{\bm \psi}(s)})(X_t))|]ds dt<\infty$.
\end{enumerate}
It also states that
\eqref{eqn7aff} and $\int_0^T|R(\bm\psi(s))|_\x ds<\infty$ can be replaced by
\begin{equation}\label{eqn32}
    h_{{\bm\psi}(t)}(x)=h_\u(x)+\int_0^th_{R({\bm \psi}(s))}(x)ds
\end{equation}
for each $x\in S$.
\end{remark}

\subsection{The polynomial case}
We are now ready to state the moment formula. Again we will use the notation introduced at the beginning of the section.
\begin{theorem}\label{th:polyinfinite}
Fix $T>0$ and set $\Vcal$ as in \eqref{eqn20}.
Then
the process
$\Xx:=(1, X, X^2, \ldots, X^n, \ldots)$
is $\Vcal$-polynomial with respect to the linear operator $L:\Vcal\to T((\R))$ given by
\begin{equation}\label{eqn31}
L({\bf u}):={\bf b}\star {\bf u}^{[1]}+\frac 1 2 {\bf a}\star{\bf u}^{[2]}.
\end{equation}
Fix then ${\bf u}\in \Vcal$
and suppose that the linear ODE
\begin{equation}\label{eqn7}
{\bf c}(t)=\u+\int_0^tL({\bf c}(s))ds,
\end{equation}
admits a $\mathcal{V}$-valued solution  with $\int_0^T|L({\bf c}(s))|_\x ds<\infty$ for all $\x\in \Ss$. Suppose also that
\begin{align}\label{eq:technial}
\mathbb{E}[\sup_{s,t \leq T}| \gg_{{\mathbf c}(t)}(X_s)|] < \infty, \quad\text{and}\quad \mathbb{E}[\sup_{s,t \leq T}|\gg_{L(\mathbf{c}(t))}(X_s)|] < \infty.
\end{align}
Then $\E[\gg_{{\mathbf u}}(X_T)]=\gg_{{\mathbf c}(T)}(x_0)$, i.e.~$
\mathbb{E}\Big[\sum_{n=0}^{\infty} {\bf u}_n  X_T^n\Big]= \sum_{n=0}^{\infty} {\bf c}_n(T)x_0^n.
$
\end{theorem}

\begin{proof}
Observe first that for each ${\bf u}\in \Ucal$ by definition of $\X$ it holds
$$\Lcal(\langle {\bf u},\fdot\rangle)(\X_t)=\Acal(\gg_{\bf u})(X_t)$$
for each ${\bf u}\in \Vcal$.
Since $\Tcal(S)^*$ is closed with respect to $\star$,  Remark~\ref{rem5} yields $L(\u)\in \Tcal(S)^*$ for each $\u\in \Vcal$ and $\gg_{L(\mathbf u)}=\Acal \gg_{\bf u}=b h'_{\bf u} + \frac{1}{2} a h''_{\bf u} $.
To show that $\Xx$ is a $\Vcal$-polynomial process we then need to verify that $\Xx$ solves the martingale problem for 
$\mathcal{L} (\langle \mathbf{u}, \fdot \rangle)({\bf x}):= \langle L(\mathbf{u}), \mathbf{x} \rangle.$
Since $\mathcal{L}(\langle \mathbf{u}, \fdot \rangle)(\Xx_t)=\Acal(\gg_{{\bf u}})(X_t)$ for each ${\bf u} \in \Vcal$, the claim follows as in the proof of Theorem~\ref{thm1}.
Since condition  \eqref{eq:technial} implies condition \eqref{eq:intcond2}, the claimed representation follows by Theorem~\ref{thm:momentpoly}.
\end{proof}

\begin{remark}\label{rem6}
Observe that the operator $L$ corresponds to the infinite dimensional matrix given by 
$$L_{ij}= \Big(j\b_{i-(j-1)}  1_{\{i\geq j-1\}}+\frac {j(j-1)}2 \a_{i-(j-2)}1_{\{i\geq j-2\}}\Big),$$
or equivalently
\[L=
\begin{pmatrix}
0 & \b_0 & \a_0 & 0 &  \cdots&  &&\\
0 & \b_1 & \a_1+2 \b_0 & 3 \a_0 &  & \cdots&&&\\
0 & \b_2 & \a_2+2 \b_1 & 3 \a_1+3 \b_0  & &\cdots&&\\
0 & \vdots&\vdots&\vdots&&&&\\
0 & \b_i &  \a_i+2 \b_{i-1}  & \cdots & \frac {j(j-1)}2 \a_{i-(j-2)}+j \b_{i-(j-1)}&\cdots  & \frac{(i+1)(i+2)}{2}\a_0 &0& \cdots\\
 0 & \vdots&&&&&&\\
\end{pmatrix}.
\]
Observe also that Remark~\ref{rem10} in this particular framework states that condition \eqref{eq:technial} can be replaced by the following two properties.
\begin{enumerate}
\item The process $(M_t^{{\bf c}(s)})_{t\geq0}$ given by 
$
M_t^{{\bf c}(s)}:=\gg_{{\bf c}(s)}(X_t)-\gg_{{\bf c}(s)}(x_0)-\int_0^t \mathcal{A}\gg_{{\bf c}(s)}(X_r) dr
$
is a true martingale for each $s\in[0,T]$. Since bounded local martingales are martingales, this condition is always satisfied when $S$ is a compact set.
\item
$\int_0^T\int_0^T \E[|\Acal\gg_{{\bm c}(s)}(X_t)|]ds dt<\infty$.
\end{enumerate}
It also states that
\eqref{eqn7} and $\int_0^T|L({\bf c}(s))|_\x ds<\infty$ can be replaced by
\begin{equation}\label{eqn28}
    h_{{\bf c}(t)}(x)=h_\u(x)+\int_0^th_{L({\bf c}(s))}(x)ds
\end{equation}
for each $x\in S$.

    Finally, note that the operator $R$ defined in Theorem~\ref{thm1} and the operator $L$ defined in Theorem~\ref{th:polyinfinite} satisfy the relations described in Proposition~\ref{prop1_affpoly} with the shuffle operator $\shuffle$ replaced by the discrete convolution $\star$. Indeed, for each $\lambda\notin \{0,1\}$ and ${\bf u}\in\Ucal$ we can write
\begin{align*}
\frac \lambda {\lambda-1}R({\bf u})-\frac 1 {\lambda(\lambda-1)}R(\lambda {\bf u})
&=L({\bf u}).
\end{align*}
The same is true for Proposition~\ref{prop2}. Indeed, letting $\exp(\star \u)$ be the coefficients of the power series $\exp(\sum_{k=0}^\infty \u_kx^k)$ we get that
$L(\exp(\star \u))=\exp(\star \u)\star R(\u).$
Moreover, assuming that ${\bf u}\star{\bf u}\in\Ucal$ for each ${\bf u} \in \Ucal$ and using that $({\bf u}\star{\bf u})^{[j]}=2({\bf u}\star{\bf u}^{[j]}+1_{\{j=2\}}{\bf u}^{[1]}\star{\bf u}^{[1]})$, we also get
\begin{align*}
    L({\bf u})+\frac 12 L({\bf u}\star{\bf u})-{\bf u}\star L({\bf u})
&=R({\bf u}).
\end{align*}
\end{remark}

\subsection{Concrete specifications and a ``counterexample''} \label{sec:concrete}

This section is dedicated to considering concrete specifications, where we obtain 
existence of the Riccati and linear ODEs of 
Theorem~\ref{thm1}
and Theorem~\ref{th:polyinfinite}, respectively. We start with Brownian motion, consider then affine processes and finally a rich specification of SDEs on compact state spaces, in general beyond the polynomial class, but including also Jacobi-type diffusions. Each of these specifications comes with a class of entire functions for which the moment formula and the affine transform formula hold true. 

We conclude this section with an example of an SDE with entire drift and diffusion coefficients, for which $x \mapsto \mathbb{E}_x[X_t]$ is not real analytic. This means that the moment formula and the affine transform
formula can only hold asymptotically as outlined in the introduction. Nevertheless, we can find a dense class of functions for which we can prove their exact validity.

Throughout this section we work here in the one-dimensional setup of Section~\ref{sec5} and consider  the power sequence expansion, denoted by 
$\mathbb X_t^{pow}:=(1,X_t,X_t^2,\ldots)$. 

\subsubsection{Brownian motion}

Let $X$ be a one-dimensional Brownian motion (starting at $0$) and denote by $\Acal$ its generator
$\Acal f(x)=\frac 1 2 f''(x)$.
Let $\Bcal$ be the set of entire functions $f:\R\to\C$ with $|f|\leq\exp(a(|\fdot|+1))$ for some $a\in \R$ and set
$\Dcal:=\{ f\in \Bcal \colon f',f''\in \Bcal\}.$
Consider the semigroup $(P_t)_{t\in[0,T]}$ given by $P_tf(x_0)=\E[f(x_0+X_t)]$ for  $f\in \Dcal$.
For this class of functions we get the following results whose proof can be found in Appendix~\ref{appB1}.

\begin{lemma}\label{lem111}
The following holds. 
\begin{enumerate}
\item\label{itnew1} $P_t$ maps $\Dcal$ to itself for each $t\in[0,T]$. 
\item\label{itnew2} The map $t\mapsto P_tf$ is continuous for each $f\in \Bcal$.
\item\label{itnew3} $\int_0^t P_s|g|(x_0) ds<\infty$
for each $t>0$, $x_0\in \R$, and $g\in \Bcal$.
\item\label{itnew4} For each $f\in \Dcal$ the map $s\mapsto P_{s}f(x_0)$ is continuously differentiable and it holds 
$\partial_s P_{s}f(x_0)= P_s\Acal f(x_0)
=\Acal P_sf(x_0).$
\end{enumerate}
\end{lemma}
Observe that the Riccati operator and the linear operator in the Brownian case  are
\begin{align*}
R^{pow}(\u)_k &= \frac{1}{2} \Big((k+1)(k+2)\u_{k+2} + \sum_{i+j=k}(i+1)(j+1) \u_{i+1}\u_{j+1}\Big),\\
L^{pow}(\u)_k &= \frac{1}{2} (k+1)(k+2)\u_{k+2}.
\end{align*}
We now provide an existence result for solutions of the corresponding Riccati and linear ODEs with respect to the power sequence expansion. A similar result for the signature representation of Section~\ref{sec4} can be proved analogously. Recall that we use the notation $\gg_{\bf u}(x):=\sum_{n=0}^\infty {\bf u}_nx^n$ and that 
by Lemma~\ref{lem111}\ref{itnew1} for each $f\in \Dcal$ the map $P_tf$ is entire.

\begin{proposition}\label{prop1}
Fix $\Ucal^{pow}:=\{\u\in T((\R))\colon h_\u\in\Dcal\}$ and  $\u\in \Ucal^{pow}$. Let ${\bf c}(t)$ denote the coefficients of $P_th_\u$, meaning that $P_th_\u=h_{{\bf c}(t)}$. {Then $({\bf c}(t))_{t\in[0,T]}$ solves the linear ODE   \eqref{eqn28} for each $x\in S$.}

Fix $\Wcal^{pow}:=\{\u\in T((\R))\colon \exp(h_\u)\in\Dcal\}$,  $\u\in \Wcal^{pow}$, and set $f(x):=\exp(h_\u(x))$. Let
${\bm \psi(t)}$ be the coefficients of the expansion around 0 of $x\mapsto\log(P_tf(x)),$
meaning that $\log(P_tf(x))=h_{{\bm \psi(t)}}(x)$ for all $|x|$ small enough.
 {Then $({\bm \psi}(t))_{t\in[0,T]}$ solves the Riccati ODE  \eqref{eqn32} for all $x$ in the domain of convergence of each $h_{\bm\psi(t)}$.}
\end{proposition}

\begin{remark}\label{rem:notentire}
Observe that for notational simplicity we implicitly assume that $h_\u$ in the definition of $\Wcal^{pow}$ is entire. This is however not needed as we could define
$$\Wcal^{pow}:=\{\u\in T((\R))\colon h_\u(x)=\log f(x)\text{ for }f\in\Dcal  \text{ and $x$ small enough}\}.$$
\end{remark}

\begin{proof}
The first statement follows by Lemma~\ref{lem111}\ref{itnew4} and the observation that $\Acal h_\u=h_{L(\u)}$ for each $\u\in \Ucal$. For the second statement, observe that for each $x$ in the convergence domain of $h_{\bm\psi(t)}(x)$  by Lemma~\ref{lem111}\ref{itnew4} we get
$
\exp(h_{{\bm \psi}(t)}(x)) \partial_th_{{\bm \psi}(t)}(x)
=\partial_t P_{t}f(x)
=\Acal P_tf(x)
=\exp(h_{{\bm \psi}(t)}(x)) h_{R({\bm \psi}(t))}(x).$
\end{proof}
\begin{remark}
    Observe that Proposition~\ref{prop1} just provides sufficient conditions for the existence of solutions of the linear ODE and the Riccati ODE, respectively. Additionally, verifying condition \eqref{eq:technial} would permit to apply Theorem \eqref{th:polyinfinite}. 
 For fixed $\u\in\Ucal^{pow}$, by Remark~\ref{rem6} and the definition of ${\bf c}(t)$, for the moment formula it suffices to check that
    $\int_0^T P_s|\Acal  h_\u|(x)ds<\infty$ and $\int_0^T\int_0^TP_{t}P_s|\Acal h_\u|(x)dsdt<\infty$ for all $x \in S$. By the definition of $\Dcal$ and Lemma~\ref{lem111}\ref{itnew3} this is the case.

    To  apply Theorem   \eqref{thm2} we would instead first need to check that \eqref{eqn32} holds for each $x\in S$ and then to verify \eqref{eqn14}. This is clearly more involved.  Applying one of these theorems would guarantee that the constructed solutions are the only solutions satisfying these technical conditions. 
\end{remark}

\subsubsection{Affine processes}

In the case of affine processes we can identify another family of entire functions for which we get existence of the Riccati and linear ODEs. Similarly to the Brownian case, we provide this for the power sequence expansion, but a similar result also holds for the signature expansion. 

The next proposition 
is an adaptation of Proposition 3.33
in  \cite{CPS:24}, where the statement can be found in the more general setup of  affine jump diffusions. Since the assumptions on the function class as well as on the affine process are slightly different, we also provide here the essential parts of the proof.

\begin{proposition}\label{prop:affine}
      Let $X=(X_t)_{t\in[0,T]}$ be an affine diffusion on $\mathbb{R}$, with state space $S=\mathbb{R}_+$ or $\mathbb{R}$, starting at $X_0=x_0$.    Let $\mathcal{E}$ be the set of entire functions 
    $h$ that can be bounded on $\mathbb{R}$ by
    $$|h(x)| \leq C(1 +|x|^n)$$ for some $n \in \mathbb{N}$. Define
    \begin{align}\label{eq:Dcal,Upow}
    \Dcal&:=\{h \in \mathcal{E}\colon h',h'' \in \mathcal{E} \}, \quad \text{and} \quad \Ucal^{pow}=\{\u\in T((\R))\colon h_\u\in\Dcal\}.
\end{align}
    Fix $\u\in \Ucal^{pow}$ and  assume further that $h_\u(x)=\int_{-\tau}^\tau \exp((\varepsilon+iu)x)g(u) du$ for some $g:[-\tau,\tau]\to \R$ with $\int_{-\tau}^\tau |g(u)|du<\infty$, $\varepsilon,\tau\in[0,\infty)$ and 
    
 $$\E[\exp(\varepsilon X_T)] < \infty.$$
 Then there exists a $\Ucal^{pow}$-valued solution $({\bf c}(t))_{t\in[0,T]}
$ of \eqref{eqn7} with initial condition $\u$ that also satisfies \eqref{eqn28} for each $x \in S$ and the moment formula holds true, i.e.,
    \begin{align*}
\E[\gg_{{\mathbf u}}(X_T)]=\gg_{{\mathbf c}(T)}(x_0).
\end{align*}
Moreover, if $\exp(h_{\v})$ for some entire function $h_{\v}$ satisfies the above conditions, then there also exists a solution to the Riccati equation \eqref{eqn7aff}
with ${\bm \psi}(0)=\v$.
\end{proposition}

\begin{remark}
As in the case of Brownian motion  $h_{\v}$ does not need to be entire in the last statement, but we can rather
use $\v \in \mathcal{W}^{pow}$,
similarly as in Remark \ref{rem:notentire}.
\end{remark}

\begin{proof}
By Lemma \ref{lemaffine} and Lemma 3.22 of \cite{CPS:24} it holds that 
$P_t\mathcal{A}h_{\u}=\mathcal{A}P_th_{\u}$. By the martingality of $N^{h_{\u}}$ as proved in  Lemma \ref{lemaffine}\ref{itc2}  and the integrability proved in Lemma~\ref{lemaffine}\ref{itc3}, an application of Fubini yields
\[
P_t h_{\u}(x_0)-h_{\u}(x_0)-\int_0^t P_s\mathcal{A}h_{\u}(x_0) ds=P_t h_{\u}(x_0)-h_{\u}(x_0)-\int_0^t \mathcal{A}P_sh_{\u}(x_0) ds=0,
\]
 and thus by Lemma \ref{lemaffine}\ref{itc3} also $\partial_t P_t h_{\u}=\mathcal{A}P_th_{\u}$. Moreover, by  
Lemma \ref{lemaffine}\ref{itc1} $P_th_{\u}$ is entire and since $\mathcal{A}h_{\u}= h_{L(\u)}$, the coefficients of $P_th_{\u}$, meaning $P_th_{\u}=h_{{\bf c}(t)}$, solve the linear ODE \eqref{eqn7} such that \eqref{eqn28} is satisfied for every $x \in S$. This proves the existence of a solution. To conclude that the moment formula holds true it remains to show that for any  solution of \eqref{eqn7} with $\int_0^T|L({\bf c}(s))|_\x ds<\infty$ we have   $
\E[\gg_{{\mathbf u}}(X_T)]=\gg_{{\mathbf c}(T)}(x_0)$.
This is a consequence of Lemma~\ref{lemaffine}\ref{itc4} and Remark~3.23 in \cite{CPS:24}.

The assertion concerning the Riccati equations follows from an adaptation of Theorem~\ref{eq:combining} to the current setup, as stated precisely in Proposition 4.15 in \cite{CPS:24}, where the coefficients of ${\bm \psi}$ are derived from ${\mathbf c}$.
\end{proof}

The essential parts for proving Proposition \ref{prop:affine}
are summarized in the next lemma, whose proof can be found in Appendix~\ref{appB2}.

\begin{lemma}\label{lemaffine}
Let  the conditions of Proposition \ref{prop:affine} on the function $h_{\u}$ and the affine process $X$ be satisfied. Denote its infinitesimal generator by $\mathcal{A}$ and consider the associated semigroup $(P_t)_{t \in [0,T]}$ given by $P_t h_{\u}(x_0)=\mathbb{E}[h_{\u}(X_t)]$ with $X$ starting at $X_0=x_0$. Then the following assertions hold true:
\begin{enumerate}
\item\label{itc1} $P_th_{\u}$  is an entire function.
\item\label{itc2}  The process $N^{h_{\u}}_t=h_{\u}(X_t)-h_{\u}(x_0)- \int_0^t \mathcal{A}h_{\u}(X_s)ds$ for $t \in [0,T]$
is a true martingale.
\item\label{itc3} The map $x \mapsto P_t\mathcal{A}h_{\u}(x)$ is continuous for every $t \in [0,T]$.
\item\label{itc4}  $\int_0^T P_s|\Acal  h_\u|(x)ds<\infty$ and $\int_0^T\int_0^TP_{t}P_s|\Acal h_\u|(x)dsdt<\infty$ for all $x \in S$.
\end{enumerate}
\end{lemma}
{\subsubsection{Diffusion processes on $\mathbb{R}$ with real analytic coefficients} \label{sec:elliptic}

We now consider generic SDEs of the form \eqref{eq:SDE1d_new}, where $b$ and $a$ are just real analytic functions. To prove that the moment and the affine transform formula hold, we consider the case of uniform ellipticity. Thus, the state space is necessarily  the whole of $\mathbb{R}$. However, in contrast to the conditions below \eqref{eqn4new}, we do not require $a$ and $b$ to be entire but just real-analytic. We then get the following version of the  moment and the affine transform formula. For its proof we mainly use existence and uniqueness assertions of parabolic equations and the identity theorem for analytic functions.

We use here the same notation as in the rest of the paper, with the only difference that $h_{\u}: \mathbb{R} \to \mathbb{C}$ now denotes a real analytic (instead of entire) function with Taylor coefficients at $0$ equal to $\u$.

\begin{theorem}\label{thm_real_analytic}

Consider an SDE of the form \eqref{eq:SDE1d_new} with initial condition $X_0=x_0$,
where $a,b:\mathbb R\to\mathbb R$ are real-analytic such that $a,b$ and $\partial_x a$  are bounded. Assume, furthermore, the following uniform ellipticity condition
\[
a(x)\ge \delta >0 \qquad \text{for all } x\in\mathbb R.
\]
Let $h_{\u}: \mathbb{R} \to \mathbb{C}$ be a real analytic function of polynomial growth (on $\mathbb{R}$) such that $\mathcal{A} h_{\u} $ also has polynomial growth.
Fix $t>0$, $\bar x\in\mathbb R$, and define
\[
U(t,x_0) =  \mathbb{E}\big[h_{\u}(X_t-\bar x)\mid X_0=x_0\big].
\]
Then the following assertions hold true.
\begin{enumerate}
\item\label{Exit1} The map $x\mapsto U(t,x)$ is real analytic on $\mathbb R$, too. In particular, there exists a radius
$r=r(t,\bar x)>0$ such that for all $|x_0-\bar x|<r$,
\[
U(t,x_0)=  \mathbb{E}\big[h_{\u}(X_t-\bar x)\mid X_0=x_0\big]=\sum_{n=0}^\infty {\bf c}_n(t,\bar x)\,(x_0-\bar x)^n,
\]
with Taylor coefficients at $0$
${\bf c}_n(t,\bar x)=\frac{1}{n!}\partial_x^n U(t,\bar x)$. Those coefficients identify $U$ uniquely.

\item\label{Exit2} Let ${\bf a}(\bar x)=({\bf a}_n(\bar x))_{n \geq 0}$ and ${\bf b}(\bar x)=({\bf b}_n(\bar x))_{n \geq 0}$ denote the coefficients of the Taylor series of
$a(\cdot+\bar x)$ and $b(\cdot+\bar x)$ at $0$, i.e.,
\[
a(x+\bar x)=\sum_{n=0}^\infty {\bf a}_n(\bar x)x^n,
\qquad
b(x+\bar x)=\sum_{n=0}^\infty {\bf b}_n(\bar x)x^n.
\]
Define the linear operator $L(\bar x)$ by
\begin{equation*}
L(\bar x){\bf u}={\bf b}(\bar x)\star {\bf u}^{[1]}+\frac 1 2 {\bf a}(\bar x)\star{\bf u}^{[2]}.
\end{equation*}

Then the coefficient vector ${\bf c}(t, \bar x)=({\bf c}_n(t,\bar x))_{n\ge0}$ satisfies the
infinite-dimensional linear ODE
\[
\partial_t {\bf c}(t,\bar x) = L(\bar x){\bf c}(t,\bar x),
\text{ for all } t\geq 0, \text{ where }
{\bf c}(0,\bar x)=\bf u.
\]
\item\label{Exit3} Denote the sequence space
$$
\mathcal{W} = \operatorname{Pow} \Big (\big \{ f \text{ real analytic on } \mathbb{R} \text{ such that $f$ and $\mathcal{A}f$ are of polynomial growth on $\mathbb{R}$} \big \} \Big ) \, ,
$$
of all Taylor coefficients $\u$ for functions $h_{\u}$ such that there is a real analytic continuation $ f $ on $\mathbb{R}$ which together with $ \mathcal{A} f$ (the analytic continuation of $\mathcal{A}h_{\u}$) is of polynomial growth. The linear map $\operatorname{Pow}$ maps a function to its power series coefficients, i.e.~$\operatorname{Pow}(h_\u)=\u $.

Then the linear ODE system is well posed within the class of coefficient sequences in $\mathcal{W}$.

\item\label{Exit4} Let $h_{\v}$ locally uniquely represent a real analytic function such that $ \exp(h_{\v}) \in \mathcal{W} $.
Assume furthermore that $\mathbb{E}[\exp(h_{\v}(X_t-\bar x))\mid X_0=\bar x] \neq 0$ for all $t \in [0,T]$ at $ \bar x$. Then, there exists a radius $r>0$ (which can depend on time $t$ and $\bar x$) such that for all $|x_0-\bar x|<r$,
\[
U(t,x_0)=  \mathbb{E}\big[\exp(h_{\v}(X_t-\bar x))\mid X_0=x_0\big]=\exp\left(\sum_{n=0}^\infty {\bm \psi}_n(t,\bar x)\,(x_0-\bar x)^n \right),
\]
where ${\bm \psi}$ solves the Riccati equation
\[
\partial_t {\bm \psi}(t,\bar x)=R(\bar x) \big ( {\bm \psi}(t,\bar x) \big), \qquad t \leq T, \qquad \psi(0,\bar x)=\v,
\]
for $t \in [0,T] $ with 
$
R(\bar x)(\u)={\bf b}(\bar x)\star {\bf u}^{[1]}+\frac 1 2 {\bf a}(\bar x)\star\big({\bf u}^{[2]}+{\bf u}^{[1]}\star{\bf u}^{[1]}\big).
$
This Riccati ODE system is locally (in time) well posed as long as taking consistent logarithms is justified and the logarithm remains in $\mathcal{W}$ in space.
\end{enumerate}
\end{theorem}

\begin{proof}
\ref{Exit1}
Since $h_{\u}$ has polynomial growth and the diffusion coefficients satisfy the 
linear growth condition, standard moment estimates imply
$ \mathbb{E}[h_{\u}(X_t-\bar x)| X_0=x_0]<\infty$.
Thus $U(t,x_0)= \mathbb{E}[h_{\u}(X_t-\bar x)| X_0=x_0]$ is well defined
and solves the Kolmogorov equation
\[
\partial_t U(t,x)
=
b(x)\partial_x U(t,x)+\tfrac12 a(x)\partial_{xx}U(t,x),
\qquad
U(0,x)=h_{\u}(x-\bar x).
\]
By uniform ellipticity, the real analytic coefficients and the respective boundedness conditions on $a, b$ and $\partial_x a$, the function $x\mapsto U(t,x)$ is real analytic on $\mathbb{R}$ for every $t \leq T$. This is a direct consequence of \cite [Theorem 1.1]{Takac2012}. 
Therefore, for each $\bar x$ there exists $r=r(t,\bar x)>0$ such that
\[
U(t,x)=\sum_{n=0}^\infty c_n(t,\bar x)(x-\bar x)^n
\quad\text{for }|x-\bar x|<r
\]
with $c_n(t,\bar x)$ being the coefficients of the Taylor series corresponding to $U(t, \cdot)$ at $\bar x$.
This proves the first claim.

\ref{Exit2} To prove the second part, define 
$w(t,\xi)=U(t,\bar x+\xi)=  \sum_{n=0}^\infty {\bf c}_n(t,\bar x)\xi^n$. Then
\begin{equation}\label{eq:linear_xi}
\partial_t w(t,\xi)
=
b(\bar x+\xi)\partial_\xi w(t,\xi)
+\frac 12 a(\bar x+\xi)\partial_{\xi\xi}w(t,\xi),
\qquad
w(0,\xi)=h_{\u}(\xi).
\end{equation}
Expanding all terms into the respective power series around $\xi=0$ and matching the coefficients, yields, for each $n\ge0$,
\[
\partial_t {\bf c}_n(t,\bar x)= \sum_{n_1+n_2=n}{\bf b}_{n_1}(\bar x) {\bf c}_{n_2}^{[1]}(t,\bar x)+\frac 1 2 \sum_{n_1+n_2=n} {\bf a}_{n_1}(\bar x){\bf c}_{n_2}^{[2]}(t,\bar x)
 \]
with initial condition ${\bf c}_n(0,\bar x)=\u_n$. By the definition of the $\star$ product as discrete convolution, the claim follows.

\ref{Exit3} 
By injectivity of the power series coefficient map $ h_{\u} \mapsto \u$ we can lift the solution semigroup, which acts on real analytic functions with polynomial growth, such that also $\mathcal{A} h_{\u} $ has polynomial growth, to the sequence space $\mathcal{W}$: it constitutes a smooth semigroup of linear maps which we denote by $(Q_t)$, in the sense that $(Q_t \u)_n$, i.e.~the $n!$ times the $n$th derivative of $P_t h_\u$ at $\bar x$, is smooth in $t$ for $ n \geq 0$. This yields existence on $\mathcal{W}$.

Let us now assume that $d:\mathbb{R}_+ \to \mathcal{W}$ is another differentiable (componentwise) solution of the linear system with initial value $ d(0) \in \mathcal{W}$. Fix $T>0$ and consider
$$
t \mapsto Q_{T-t} d(t) \, .
$$
We now prove that this map is constant by transferring it back to real analytic functions. Indeed
$$
t \mapsto \mathbb{E}\big[h_{d(t)}(X_{T-t}-\bar x)\mid X_0=x_0\big]
$$
has vanishing derivative since $ \frac{d}{dt} h_{d(t)} = \mathcal{A} h_{d(t)}$ for $ 0 \leq t \leq T$ by assumption on $d$. This proves uniqueness.

\ref{Exit4}
As $U(t,\bar x)= \mathbb{E}[\exp(h_{\v}(X_t-\bar x))| X_0=\bar x] \neq 0$ for all $t \in [0,T]$ by assumption, it follows from continuity that there exists some $\bar r >0$ such that
\[
U(t,x)\neq 0 \qquad \text{for all } |x-\bar x|<\bar r  \text{ and } t \in [0,T].
\]
Thus, on  $(\bar x-\bar r,\bar x+\bar r)$ we  may choose a continuous branch of
the complex logarithm and define
\[
V(t,x):=\log(U(t,x)).
\]
Since $U(t,\cdot)$ is real analytic by the same arguments as above,  the function $V(t,\cdot)$ is real analytic on $(\bar x-\bar r,\bar x+\bar r)$. Therefore $V(t,\cdot)$ admits a convergent Taylor expansion around $\bar x$, such that
\[
V(t,x)=\sum_{n=0}^\infty {\bm \psi}_n(t,\bar x)\,(x-\bar x)^n,
\]
for a small neighborhood of $\bar x$ and for uniquely determined coefficients ${\bm \psi}_n(t,\bar x)\in\mathbb C$. Since $U=\exp(V)$ on
$(\bar x-\bar r,\bar x+\bar r)$, it follows from the Kolmogorov equation for $U$ that  $V$
satisfies
\begin{equation}
\label{eq:HJ_full}
\partial_t V(t,x)
=b(x)\,\partial_x V(t,x)+\tfrac12 a(x)\Big(\partial_{xx}V(t,x)+(\partial_x V(t,x))^2\Big) \, .
\end{equation}
The initial condition is $V(0,x)=h_{\v}(x-\bar x)$. 

Set now $\xi=x-\bar x$ and write $v(t,\xi):=V(t,\bar x+\xi)$. Then 
\begin{equation}
\label{eq:HJ_shifted_full}
\partial_t v(t,\xi)
=b(\bar x+\xi)\,\partial_\xi v(t,\xi)
+\tfrac12 a(\bar x+\xi)\Big(\partial_{\xi\xi}v(t,\xi)+(\partial_\xi v(t,\xi))^2\Big).
\end{equation}
Expanding all terms  into power series at $\xi=0$ and matching the coefficients we obtain
\[
\partial_t {\bm \psi}(t,\bar x)=
{\bf b}(\bar x)\star {\bm \psi}^{[1]}(t,\bar x)+\frac 1 2 {\bf a}(\bar x)\star\big({\bm \psi}^{[2]}(t,\bar x)+{\bm \psi}^{[1]}(t,\bar x)\star{\bm \psi}^{[1]}(t,\bar x)\big).
\]
Since $v(0,\xi)=h_{\v}(\xi)$,  the initial condition is
$
{\bm \psi}(0,\bar x)=\v.
$

Well-posedness follows by the same arguments as in the linear case. 
\end{proof}

\begin{remark}
Note that the essential part of the above proof is the real-analyticity of $U(t,\cdot)$ for each $t \in [0,T]$. There is a huge body of classical literature where this property has been established, usually by constructing a fundamental solution for uniformly parabolic operators relying on the so-called parametrix method. This in turn yields the necessary estimates on the solution and its derivatives required to establish analyticity, see e.g., \cite{eidelman2004analytic, friedman2008partial, friedman1958regularity} and the references therein. For this approach the standard conditions are boundedness on the coefficients and sometimes even non-spatial dependence of $a$ as e.g., in \cite{friedman1958regularity}. In \cite[Chapter 2.3]{eidelman2004analytic} the parametrix construction is extended to equations with linear drift part, but still bounded diffusion and further conditions on the derivatives. The more recent paper \cite{DeckKruse2002} also extends the parametrix construction to the case of unbounded coefficients.  
Applied to our setting, the respective conditions translate to sub-linear growth for the drift, i.e., $|b(x)|\leq C (1 +|x|^{\beta})$ for $\beta <1$, while $a$ has to be globally Lipschitz.

Note that fundamental solutions have been constructed without recourse to the parametrix method (see, e.g., \cite{aronson1967parabolic}), thereby accommodating more general growth conditions; however, this approach does not provide the derivative estimates needed to infer analyticity.

The approach applied in \cite{Takac2012} is also different from the parametrix method. It rather relies on an
$L^2$-a priori bound for $U(t,x)$
when the initial condition is some holomorphic function in a Hardy space. This then yields analytic smoothing  effects such that analyticity in space \emph{and time} holds when the initial condition is only in $L^2$. In a similar context, \cite{escauriaza2017analyticity} provides precise quantitative estimates for space-time analyticity. We also refer to this paper for an extensive historical literature review.
\end{remark}

\begin{remark}
Note that uniform ellipticity is a crucial assumption here, however, not a necessary one. Without it several examples with real analytic characteristics and degenerate diffusion part are known, where solutions of the Kolmogorov equation can even get rougher than the initial condition, see \cite{hairer2015}. On the other hand there are well known diffusion processes, namely affine ones, which can be degenerate but real analyticity is preserved to some extent.

The precise condition which is needed is actually \emph{analytic hypoellipticity} which was first considered in  \cite{Treves1971,BaouendiGoulaouic1972} and systematically analyzed in \cite{Treves1992}. Analytic hypoellipticity just  means that the solution of the PDE is analytic whenever this holds for the operator applied to it. This condition is hard to characterize. In general it does not hold for hypoelliptic operators and it is also not implied by Hörmander’s bracket condition.

Microlocal theory as considered, e.g., in \cite{Treves1992} allows to localize analytic hypoellipticity. An important object in microlocal theory is the analytic wave front set $\operatorname{WF}_a(U)$ of the PDE-solution $U$. In our time-space situation a point $(t,x)$ and a direction $(\kappa, \eta)$ is not an element of $\operatorname{WF}_a(U)$ if -- loosely speaking -- $U$ is real analytic at $(t,x)$ in direction $(\kappa, \eta)$ (expressed in terms of Fourier analysis). For linear differential operators of the current form with real analytic coefficients, that is,
\[
P=-\partial_t  +b(x) \partial_x + \frac{1}{2} a(x) \partial_{xx} 
\]
with principal symbol $\sigma_P((t,x), (\kappa, \eta)) =\frac{1}{2} a(x)\eta^2$ and characteristic variety
$$
\operatorname{char} P = \{ ((t,x) , (\kappa,\eta) ) : (\kappa,\eta) \neq 0 ,\, \sigma_P((t,x), (\kappa, \eta))  = 0 \} , 
$$
a main result (see \cite[Theorem 9.5.1 and Chapter 9.6]{hormander}) states that 
\begin{align}\label{eq:analytic_wave_front}
\operatorname{WF}_a(PU) \subset  \operatorname{WF}_a(U), \quad \text{and} \quad \operatorname{WF}_a(U) \subseteq \operatorname{WF}_a(PU) \cup \operatorname{char} P.
\end{align}
Applying this to the current situation where $a$ can vanish at certain points $x$ and where the initial condition of the Kolmogorov equation is analytic in space, then analyticity in space can only be lost at the degeneracy points of $a$.  This means that Theorem   \ref{thm_real_analytic} also holds true in this case as long as $\bar x$ is not in the degeneracy set. In Section \ref{ex:counterexample} we show the failure of real-analyticity at such degeneracy points by means of a concrete example.

Observe that Theorem \ref{thm_real_analytic} extends under analogous conditions also to $\mathbb{R}^n $-valued SDEs with real-analytic coefficients, which then correspond to specific signature SDEs. As this section deals with one-dimensional equations, we here stated it only in this setting.

In this context, note that the second part of \eqref{eq:analytic_wave_front} gives a powerful criterion in multivariate situations to analyze in which directions non-analyticities propagate and in which not. So even in cases with degenerate diffusion  matrix on the whole state space, there can be directions where analyticity persists and where the above theory is -- after a change of coordinates -- applicable. Moreover,
this microlocal theory also allows 
to consider 
coefficients that are not globally real analytic, but only locally in a neighborhood of $\bar{x}$ and to apply our results in such a situation.
\end{remark}

  }

\subsubsection{Diffusion processes on bounded state spaces with entire coefficients}

This section is dedicated to specify a class of SDEs of the form \eqref{eq:SDE1d_new} in terms of their infinitesimal generators and a class of entire functions given by coefficients 
$\u$ such that
 a solution to the linear ODE \eqref{eqn7} exists, which also satisfies \eqref{eqn28} for each $x\in S$. 
 We here work again with power sequence expansion and consider  operators $\mathcal{A}$ which act on $x \mapsto \lambda_k x^k$ as
\[
\mathcal{A}(\lambda_k(\cdot)^k)(x)=\mathcal{G} (\lambda_{\cdot} x^{(\cdot)})(k), \quad x \in S, \, k \in \mathbb{N}_0,
\]
for some operator 
$\Gcal$ acting on functions $f:\N_0\to \R$ as
\[
\mathcal{G}f(k)=\sum_{j=0}^{\infty} \beta_{kj}(f(j)- f(k))+ \beta_k f(k).
\]
Assuming that $\beta_{kj}\geq 0$ and $\beta_k\leq 0$, we can see that $\Gcal$ corresponds to the generator of a jump process $Z$ with state space $\mathbb{N}_0$ that jumps from $k$ to $j$ with intensity $\beta_{kj}$ and can be killed with rate $\beta_k$ depending on its position $k$. 
The process $Z$ is thus  dual to the original process $X$ and under appropriate integrability conditions (see Theorem~4.11 in \cite{EK:05}) we can conclude that
\begin{align}\label{eq:duality}
\mathbb{E}[\lambda_{Z_0}X_T^{Z_0}]=\mathbb{E}[\lambda_{Z_T}x_0^{Z_T}].
\end{align}
Set then $\mathbf{\mu}_k(t):=\mathbb{P}[Z_t=k]$,  $\mathbf{u}_k=\lambda_k \mu_k(0)$, and note that
\begin{align*}
\mathbb{E}[\lambda_{Z_0}X_T^{Z_0}]=\sum_{k=0}^{\infty}\mathbb{E}[\lambda_{Z_0}X_T^{Z_0}|Z_0=k] P(Z_0=k)=
\sum_{k=0}^{\infty} \lambda_k \mu_k(0)\mathbb{E}[X_T^{k}]=
\sum_{k=0}^{\infty}\mathbf{u}_k \mathbb{E}[X_T^{k}]
\end{align*}
and similarly
$
\mathbb{E}[\lambda_{Z_T}x_0^{Z_T}]=\sum_{k=0}^{\infty} \lambda_k \mu_k(T) x_0^k.
$
By \eqref{eq:duality} we therefore have 
\[
\sum_{k=0}^{\infty}\mathbf{u}_k \mathbb{E}[X_T^{k}]=\sum_{k=0}^{\infty} \lambda_k \mu_k(T) x_0^k.
\]
If the sum and the expected value can be interchanged  this then implies that
$
\mathbf{c}_k(T)= \lambda_k\mu_k(T).
$
Since $\mu(t)$ solves the Fokker Planck equation 
$
\partial_t \mu(t) = \mathcal{G}^* \mu(t),
$
which reduces to infinite dimensional linear ODE in this case, we can identify $L$ with $\mathcal{G}^*$.

Let again $\Vcal$ be as defined in \eqref{eqn20} and in order to simplify the notation set
\begin{equation}\label{Rcal}
\Rcal_C:=\{{\bf u}\in T((\R))\colon {\bf u}_k\geq0,\ \sum_{k=0}^\infty{\bf u}_k\leq C\},\qquad \Rcal:=\bigcup_{C>0}\Rcal_C.
\end{equation}

\begin{theorem}\label{lem4new}
Fix $T >0$,  $\lambda\in T((\R))$, and  $\mu\in \Rcal$ such that
$\u:=(\lambda_0\mu_0,\lambda_1\mu_1,\ldots)\in \Vcal$.
Assume that
$$\Acal (\lambda_k(\fdot)^k)(x)=\sum_{j=0}^\infty \beta_{kj}(\lambda_jx^j-\lambda_kx^k)+\beta_k\lambda_kx^k$$
for some $\beta_{kj}\in \R_+$ and $\beta_k\in \R$ such that $\sup_{k\geq 0}\beta_k^+<\infty$ and $\lim_{j\to\infty}\beta_{kj}=0$ for each $k\in \N_0$. 
Assume also that for all $x\in S+B_\e(0)$ the element
\begin{equation}\label{eqn6Z}
\Big((\lambda_kx^k)_{k\in\N_0},\Big(\sum_{j=0}^\infty \beta_{kj}(\lambda_jx^j-\lambda_kx^k)\Big)_{k\in\N_0}\Big)\in T((\R))\times T((\R))
\end{equation}
lies in the bounded pointwise\footnote{A sequence $((\u^n,\v^n))_{n\in \N}\subseteq T((\R))\times T((\R))$ converges to $(\u,\v)\in T((\R))\times T((\R))$ in the bounded pointwise sense if $\sup_{n,k}|\u^n_k|,\sup_{n,k}|\v^n_k|<\infty$ and $\lim_{n\to\infty}(\u^n_k,\v^n_k)=(\u_k,\v_k)$ for each $k$.} closure of the set
\begin{equation*}
\Big\{\Big(\v,\Big(\sum_{j=0}^\infty \beta_{kj}(\v_j-\v_k)\Big)_{k\in \N_0}\Big)\colon \v\in T(\R)\Big\}\subseteq T((\R))\times T((\R)) \, .
\end{equation*}
{Then there exists a $\Vcal$ valued map $({\bf c}(t))_{t\in[0,T]}$, of the form  ${\bf c}_k(t)=\lambda_k\mu_k(t)$ for some $\mu_k(t)\geq0$, satisfying the linear ODE \eqref{eqn28} for each $x\in S$.} If $\beta_k\leq 0$ for each $k$, then  $\sum_{k=0}^\infty\mu_k(t)\leq \sum_{k=0}^\infty\mu_k$.
If we additionally have that
\begin{align*}
\int_0^T\int_0^T \E[|\Acal(\gg_{\mathbf{c}(t)})(X_s)|]ds dt
&
<\infty,
\end{align*}
we can conclude that the moment formula provided by Theorem~\ref{th:polyinfinite} is satisfied.
\end{theorem}

\begin{proof}
Without loss of generality assume that $\sum_{k=0}^\infty {\mu}_k=1$.
 Set  $D:=C_c(\N_0)$ and consider  the linear operator $\Gcal:D\to C_0(\N_0)$ given by
 $\Gcal v(k):=\sum_{j=0}^\infty \beta_{kj}(v(j)-v(k)).$
Observe that  $D$ is a dense subset of $C_0(\N_0)$  and $\Gcal$ satisfies the positive maximum principle on $\N_0$. By Theorem~4.5.4 in \cite{EK:05} we can conclude that there exists a solution $(Z_t)_{t\in[0,T]}$ to the corresponding martingale problem with initial condition $\mu$, meaning that
\begin{equation}\label{eqn100Z}
v(Z_t)-v(Z_0)-\int_0^t  \Gcal v(Z_s) ds
\end{equation}
is a martingale for all $v\in D$.
Fix  $x\in S+B_\e(0)$ and set 
$$v_x(k):=\lambda_kx^k
\qquad \text{and}\qquad \Gcal v_x(k):=\sum_{j=0}^\infty \beta_{kj}(\lambda_jx^j-\lambda_kx^k).$$
Since $(v_x,\Gcal v_x)$ lies in the bounded pointwise closure of $\{(v,\Gcal v)\colon v\in D\}$ by assumption for each $x\in S+ B_\e(0)$, we can conclude that \eqref{eqn100Z} holds for each $v_x$. Noting that $\exp(\int_0^t\beta_{Z_s} ds)$ is a bounded process of bounded variation, we can conclude that
$$e^{\int_0^t\beta_{Z_r}dr}v_x(Z_t)-v_x(Z_0)-\int_0^t  e^{\int_0^s\beta_{Z_r}dr}(\Gcal v_x(Z_s)+\beta_{Z_s}v_x(Z_s)) ds$$
is a martingale as well. Set now
$\mu_k(t):=\E[\exp(\int_0^t\beta_{Z_s} ds)1_{\{Z_t=k\}}]$ and note that computing the expectation of the last equation we get
$$\sum_{k=0}^\infty v_x(k)\mu_k(t)=\sum_{k=0}^\infty v_x(k)\mu_k(0)+\int_0^t \sum_{k=0}^\infty (\Gcal v_x(k)+\beta_kv_x(k))\mu_k(t).$$
Define ${\bf c}_k(t):=\lambda_k\mu_k(t)$ and note that since
 $\gg_{{\mathbf c}(t)}(x)=\sum_{k=0}^\infty  {\mu}_{k}(t)\lambda_kx^k=\langle v_x,{\mu}(t)\rangle<\infty$ for each $x\in S+B_\e(0)$ we also have that $\gg_{{\mathbf c}(t)}$ is smooth on $S+B_\e(0)$. Since power series analysis yields $\Acal(\gg_{{\mathbf c}(t)})(x)=\sum_{k=0}^\infty \mu_k(t)\Acal (\lambda_k(\fdot)^k)(x)$ we thus get
$$\Acal(\gg_{{\mathbf c}(t)})(x)=\sum_{k=0}^\infty  {\mu}_{k}(t)\sum_{j=0}^\infty \beta_{kj}(\lambda_jx^j-\lambda_kx^k)-\beta_k\lambda_kx^k
=\langle \Gcal v_x,{\mu}(t)\rangle.$$
Combining these observations we obtain
\begin{align*}
\gg_{{\mathbf c}(t)}(x)-\gg_{\u}(x)
&=\sum_{k=0}^\infty  {\mu}_{k}(t)\lambda_kx^k-\sum_{k=0}^\infty  {\mu}_{k}\lambda_k x^k
=\langle v_x,{\mu}(t)\rangle-\langle v_x,{\mu}\rangle\\
&=\int_0^t\langle \Gcal v_x,{\mu}(s)\rangle ds
=\int_0^t \Acal(\gg_{{\mathbf c}(s)})(x)ds
=\int_0^t  \gg_{L({\bf c}(s))}(x)ds,
\end{align*}
for each $x\in S+B_\e(0)$, proving \eqref{eqn28}.
\end{proof}

{
\begin{remark}
   An inspection of the proof of Theorem~\ref{lem4new} provides a representation of the solution $({\bf c}(t))_{t\in[0,T]}$ of the linear ODE \eqref{eqn28} with ${\bf c}_k(0)=(\lambda_0\mu_0,\lambda_1\mu_1,\ldots)\in \Vcal$ for $\mu\in \Rcal$. One indeed has that ${\bf c}_k(t)=\lambda_k\mu_k(t)$ 
where 
\[
\mu_k(t) := \Big(\sum_{k=0}^\infty \mu_k\Big) 
\E\left[\exp\left(\int_0^t \beta_{Z_s} ds\right) 
\mathbf{1}_{\{Z_t = k\}}\right].
\]
Here \( (Z_t)_{t \ge 0} \) is an \(\mathbb{N}_0\)-valued pure jump Markov process 
with generator 
\[
\Gcal v(k) := \sum_{j=0}^\infty \beta_{kj}(v(j) - v(k)), 
\qquad v \in C_c(\mathbb{N}_0).
\]
That is, \(Z\) jumps from state \(k\) to state \(j\) with intensity \(\beta_{kj}\). 
Hence, \(Z\) can be interpreted (and simulated) as a pure jump Markov process 
whose transition rates are given by the matrix \((\beta_{kj})_{kj \in\N_0}\).
\end{remark}
}

Next, we consider three different examples which satisfy the conditions of Theorem~\ref{lem4new}.

\paragraph{The Jacobi diffusion}
Consider the Jacobi diffusion $X$, i.e.~the  process $X$ on $[0,1]$ given by \eqref{eq:SDE1d_new} for $b=0$ and $a(x)= x-x^2$.
Observe that setting  $f_k(x)=(x/2)^k$ we get
$$\Acal(f_k)(x)=\frac {k(k-1)}4(f_{k-1}(x)-f_k(x))-\frac{k(k-1)}4f_k(x),$$
which implies that $\Acal$ is of the form described in Theorem~\ref{lem4new} for $\lambda_k:=(1/2)^k$, $\beta_{kj}=\frac {k(k-1)}41_{\{j=k-1\}}$, and $\beta_k=-\frac {k(k-1)}4$.
Moreover, for each $x\in[-\e,1+\e]$ setting $u^n_{x,k}:=\lambda_kx^k1_{\{k\leq n\}}$ we obtain that
\begin{equation*}
\Big((u^n_{x,k})_{k\in \N_0},\Big(\sum_{j=0}^\infty \beta_{kj}(u^n_{x,j}-u^n_{x,k})-\beta_ku^n_{x,k}\Big)_{k\in \N_0}\Big)_{n\in \N_0}\subseteq T(([0,1]))\times T(([0,1]))
\end{equation*}
converges to \eqref{eqn6Z} uniformly for $n$ going to infinity.

Next, recall that $\Rcal$ is defined in \eqref{Rcal} and fix
$\mu\in \Rcal$ such that setting ${\u}_k:=\mu_k/2^k$ it holds that
$|\gg_{\u}(x)|<\infty$
 for each $x\in [-\e,1+\e]$. By Theorem~\ref{lem4new} we can conclude that 
\eqref{eqn7}  admits a $(-\e,1+\e)$-solution $({\bf c}(t))_{t\in[0,T]}$ with ${\bf c}_k(t)=\mu_k(t)/2^k$,  $\mu_k(t)\geq0$, and  $\sum_{k=0}^\infty\mu_k(t)\leq \sum_{k=0}^\infty\mu_k$.
Moreover, since $\Acal(\gg_{\mathbf{c}(t)})(x)\geq 0$ for each $x\in [0,1]$ we can compute
\begin{align*}
\int_0^T\int_0^T \E[|\Acal(\gg_{\mathbf{c}(t)})(X_s)|]ds dt
&\leq\sup_{x\in[0,1]}(\gg_{{\mathbf c}(T)}(x)-\gg_{\mathbf c}(x))
<\infty.
\end{align*}
This implies that the conditions of Remark~\ref{rem2} are satisfied and by Theorem~\ref{th:polyinfinite} we can conclude that
\begin{equation}\label{eqn8}
\mathbb{E}\Big[\sum_{n=0}^{\infty} {\u}_n  X_T^n\Big]= \sum_{n=0}^{\infty} {\bf c}_n(T)x_0^n.
\end{equation}
Finally, fix $u\geq0$ and observe choosing $\mu\in T((\R))$ with $\mu_k=\frac{(2u)^k}{k!}$ for each $k\in\N_0$ it holds $\mu\in \Rcal$ and $|\gg_{\u}(x)|=\sum_{k=0}^\infty \frac {(ux)^k}{k!}=\exp(ux)<\infty $ for each $x\in [-\e,1+\e]$. We can thus conclude that \eqref{eqn8} provides a representation of
$\E[\exp(uX_T)]$.

Next, choose ${\u}\in T((\R))$ such that $|\gg_{\u}(x)|<\infty$ for each $x\in [-\e,1+\e]\cup \{2+\e\}$ and 
define  $\mu\in T((\R))$ as $\mu_k= 2^k{\bf c}_k$
 for each $k\in\N_0$. 
 Here we need to include $\{2+\e\}$ because of the choice of the factor $\lambda_k=\u_k/\mu_k$. Using $\lambda_k=(1+\e/2)^{-k}$ we would not need to impose this auxiliary assumption.
 {Set then ${\u}^+\in T((\R))$ as $\u_n^{+}=\u_n\lor 0$ and ${\u}^-\in T((\R))$ similarly.}
 Observe that
$
|\gg_{{\u}^+}(x)|+|\gg_{{\u}^-}(x)|
\leq 
    \sum_{k=0}^\infty |{\u}_k|x^k
$
 for each $x\in (-\e,1+\e)$, which is finite since power series converge absolutely in the interior of their domain of convergence.
The same argument for $x=2$ yields $\mu^+,\mu^-\in \Rcal$. We can thus conclude that \eqref{eqn8} provides a representation of
$\E[h_{\u}(X_T)]$ as
$$\E[h_{\u}(X_T)]
=\E[h_{\u^+}(X_T)]-\E[h_{\u^-}(X_T)]
=h_{{\bf c}^+(T)}(x_0)-h_{{\bf c}^-(T)}(x_0).
$$
In particular, choosing ${\bf u}\in T((\R))$ such that $|\gg_{\mathbf u}(x)|<\infty$ for each $x\in [-\e,1+\e]\cup \{2+\e\}$ and setting $\bf c$ such that $h_{\bf c}=\exp(h_{\bf u})$ we get that $|\gg_{\mathbf c}(x)|<\infty$ for each $x\in [-\e,1+\e]\cup \{2+\e\}$. Such ${\bf c}$ is explicitly given by
${\bf c}=e^{\star \bf u}$, i.e.
${\bf c}_k=\frac {1} {k! } \sum_{j=1}^\infty\sum_{i_1+\ldots+i_j=k} u_{i_1}\cdots u_{i_j}$
 for all $k\in\N_0$. 
We can  conclude that \eqref{eqn8} provides a representation of
$\E[\exp(h_{\bf u}(X_T))]$ as
$$\E[\exp(h_{\bf u}(X_T))]
=\E[h_{\bf c^+}(X_T)]-\E[h_{\bf c^-}(X_T)]
=h_{{\bf c}^+(T)}(x_0)-h_{{\bf c}^-(T)}(x_0).
$$

\paragraph{The shifted Jacobi diffusion} Let $X$ be the Jacobi process and  set $X^1:=X-1$. Observe that the corresponding extended generator $\Acal^1$ is given by \eqref{eq:SDE1d_new} for $b=0$ and $a(x)= x+x^2$.
Setting $f_k(x)=(-{x}/2)^k$ we thus have that
$$\Acal^1(f_k)(x)=\frac {k(k-1)}4(f_{k-1}(x)-f_k(x))-\frac{k(k-1)}4f_k(x),$$
which implies that $\Acal^1$ is of the form described in Theorem~\ref{lem4new} for  $\lambda_k:=(-1/2)^k$, $\beta_{kj}=\frac {k(k-1)}41_{\{j=k-1\}}$, and $\beta_k=-\frac {k(k-1)}4$. By symmetry, all the arguments of the previous example work and we can conclude that \eqref{eqn7}  admits a $ [-\e,1+\e]$-solution $({\bf c}(t))_{t\in[0,T]}$ such that
\begin{equation*}
\mathbb{E}\Big[\sum_{n=0}^{\infty} {\bf u}_n  (X_T-1)^n\Big]
= \sum_{n=0}^{\infty} {\bf c}_n(T) (x_0-1)^n,
\end{equation*}
for each ${\bf u}\in T((\R))$ such that ${\bf u}_k=\mu_k(-1/2)^k$ for some ${\bf \mu}\in \Rcal$ and
$|\gg_{\mathbf u}(x)|$
is finite for each $x\in [-\e,1+\e]$. Choosing $\mu\in T((\R))$ with $\mu_k=\frac{(2u)^k}{k!}$ for each $k\in\N_0$ we can thus conclude that 
$$\mathbb{E}[\exp(-uX_T)]=\mathbb{E}[\exp(-u(X_T-1))]\exp(-u)=\exp(-u) \sum_{n=0}^{\infty} {\bf c}_n(T) (x_0-1)^n$$
for each $u\geq0$.

\paragraph{Non polynomial diffusions on the unit interval} Similarly, we can deal with other examples. Consider the process 
$X$ on $[0,1]$ given by \eqref{eq:SDE1d_new} for $b=0$ and $a(x)= x(1-x)(1-x/2)$.
Observe that setting  $f_k(x)=(x/2)^k$ we have that
$$\Acal(f_k)(x)=\frac {k(k-1)}4(f_{k-1}+2f_{k+1}-3f_k),$$
which implies that $\Acal$ is of the form described in Theorem~\ref{lem4new} for  $\lambda_k:=(1/2)^k$, $\beta_{kj}=\frac {k(k-1)}41_{\{j=k-1\}}+\frac {k(k-1)}21_{\{j=k+1\}}$, and $\beta_k=0$.
Following the first example we can then apply Theorem~\ref{th:polyinfinite} to conclude that
$
\mathbb{E}[\sum_{n=0}^{\infty} {\bf u}_n  X_T^n]= \sum_{n=0}^{\infty} {\bf c}_n(T)x_0^n
$
for each ${\bf u}\in T((\R))$ such that ${\bf u}_k=\mu_k(1/2)^k$ for some ${\bf \mu}\in \Rcal$ and
$|\gg_{\mathbf u}(x)|$
is finite for each $x\in [-\e,1+\e]$. Again, observe that this in particular provides a representation for $\E[\exp(uX_T)]$ for each $u\geq0$.

 The same type of approach generalizes to more elaborate types of processes on $[0,1]$, as for instance the  Fleming Viot type process from population genetics  given by
\[
dX_t=\sum_{n=1}^{\infty} b_n(X^n_t - X_t^{n+1}) dt+ \sqrt{X_t(1-X_t)} dB_t
\]
for some appropriately chosen coefficients $b_n$ (see \cite{CSp:18}). This is classically neither an affine nor a polynomial process. It corresponds to a diffusion of the form \eqref{eq:SDE1d_new} where $b$ and $a$ are given by \eqref{eqn4new} for
$$
 {\bf b}_n =b_n-b_{n-1},\quad \a_1=1,\quad \a_2=-1\quad\text{and}\quad \a_n=0\quad\text{for each $n\neq 1,2$.}
$$
By Theorem~\ref{thm1} the process $X$ can be seen as the projection of a $\Vcal$-affine process $\X$ whose operator $R$ is given by
$R(\u):={\bf b}\star {\bf u}^{[1]}+\frac 1 2 {\bf a}\star\big({\bf u}^{[2]}+{\bf u}^{[1]}\star{\bf u}^{[1]}\big)$
and in this case satisfies
\begin{align*}
R(\mathbf{u})_n
&=\sum_{k_1+k_2=n}\Big( (b_{k_1}-b_{k_1-1})(k_2+1)\u_{k_2+1}\Big)
+\frac 1 2 
(n(n+1)\u_{n+1}-(n+1)(n+2)\u_{n})
\\
&+\frac 1 2\Big(
\sum_{k_2+k_3=n-1}
(k_2+1)\u_{k_2+1}(k_3+1)\u_{k_3+1}
-
\sum_{k_2+k_3=n-2}
(k_2+1)\u_{k_2+1}(k_3+1)\u_{k_3+1}\Big).
\end{align*}
{In contrast to results above, where real analytic transition kernels are considered, in \cite{CP:17} the authors are directly constructing diffusion semigroups on spaces of real analytic functions, which is in the spirit of our work.

Consider the Hilbert space
$$H^2:=\{ \u\in \Tcal(S)^*\colon \sum_{n=0}^\infty |\u_n|^2<\infty\}
$$
which corresponds to a Hardy space, as introduced in \cite{CP:17} in our notation.

\begin{lemma}\label{lemhardy}
    Fix $S\subseteq [-1,1]$ and consider a one-dimensional diffusion $X$ given by \eqref{eq:SDE1d_new} with extended generator $\mathcal A$. Assume that  $\Acal$ satisfies
    $$\Acal f(x)=  (h_{\tilde\a}(x)-x^2) f''(x),\qquad x\in S,$$
    for some $\tilde  \a\in \Tcal(S)$ such that \begin{equation}\label{hardybound}
       \Big|\sum_{n=0}^\infty\tilde \a_nz^n\Big|\leq 1/2 
    \end{equation}
    for each $z\in \C$ with $|z|< 1$. 
    Then for each $\u\in H^2$ there exists ${\bf c}(t)\in H^2$ such that
    $\E[h_\u(X_t)|X_0=x_0]=h_{{\bf c}(t)}(x_0).$
\end{lemma}

Note that an operator $\mathcal A$ of such form satisfies the positive maximum principle and can be the extended generator of a diffusion process on  $S$ if and only if $h_{\tilde\a}(x)\geq x^2$ on $S$. In view of \eqref{hardybound} this implies that $S\subseteq [-2^{-1/2},2^{-1/2}]$.  

\begin{corollary}
    Consider the setting introduced in Lemma~\ref{lemhardy}. Then there exist an $H^2$-valued solutions ${\bf c}(t)$ of \eqref{eqn7} for
    $$
    L({\bf u})=\frac 1 2 (\tilde {\bf a}-e_{11})\star{\bf u}^{[2]},$$
    such that the results of Theorem~\ref{th:polyinfinite} holds true.
\end{corollary}

\begin{proof}
Proceeding as in the proof of Theorem~\ref{thm:momentpoly} we get that ${\bf c}(t)$ solves \eqref{eqn28}. Moreover, condition (i) of Remark~\ref{rem6} is satisfied by compactness of $S$ and  condition (ii) of the same remark follows by
        $$\int_0^T\int_0^T \E[|\Acal\gg_{{\bm c}(s)}(X_t)|]ds dt
        \leq
        \int_0^T\int_0^T \E[|\Acal h_\u(X_{t+s})|]ds dt
        \leq 
        \sup_S|\mathcal Ah_{\u}|\frac {T^2} 2
        <\infty,$$
        where the first inequality follows by noticing that 
        $$\Acal h_{{\bm c}(s)}=(\Acal \E[h_\u(X_s)|X_0=\cdot])(\cdot)=\E[\Acal h_\u(X_s)|X_0=\cdot]$$
         and applying Jensen inequality.
\end{proof}

Examples for $a(x)$ satisfying these assumptions include Jacobi diffusions as $1/2-x^2$ or $x(1/2-x)$, but also more involved examples as $e^{x-1}/2-x^2$. It is interesting to note that, as in Theorem~\ref{lem4new}, a special role is played by the coefficient of $x^k$ in $\mathcal A ((\cdot)^k)(x)$. 

By tensorization we do also obtain a multivariate corollary.

\begin{corollary}

Fix $S\subseteq [-1,1]^d$ and consider a $d$-dimensional diffusion $X$ with extended generator $\mathcal A$. Assume that  $\Acal$ satisfies
    $$\Acal f(x)=  \sum_{i=1}^d (h_{\tilde\a_i}(x)-x_i^2) \partial_i^2 f(x),\qquad x\in S,$$
    for some $\tilde  \a_i \in \Tcal(S)$ such that 
    \begin{equation}
       \Big|\sum_{n=0}^\infty\tilde \a_{i,n} z^n\Big|\leq 1/2 
    \end{equation}
    for each $z\in \C$ with $|z|< 1$. 
    Then for each $\u\in H^2 \otimes \ldots \otimes H^2$ there exists ${\bf c}(t)\in H^2 \otimes \ldots \otimes H^2 $ such that
    $\E[h_\u(X_t)|X_0=x_0]=h_{{\bf c}(t)}(x_0).$

\end{corollary}
}

{\subsubsection{An example showing the failure of real-analyticity 
$x \mapsto \mathbb{E}[f(X_t)|X_0=x]$ at degeneracy points}
\label{ex:counterexample}

The next example (derived from \cite{IT:97}) illustrates the
failure of real-analyticity of
$x \mapsto \mathbb{E}[f(X_t)|X_0=x]$ at degeneracy points
 even though the corresponding stochastic flows are real analytic. }

Fix $d=1$ and $S=\R$.
Let $X_t^{x_0}$ be a process given by \eqref{eq:SDE1d_new} for $b(x)=\cos(2x)\sin(2x)$ and $a(x)=\sin^2(2x)$ and initial condition $X_0=x_0$. The vector fields are obviously entire functions and globally bounded on $\mathbb{R}$. The solution of this SDE is given by $X_t^{x_0}=x_0$ for each $X_0=x_0=k\pi/2$ and 
$$X_t^{x_0}= \arctan(\exp(2B_t)\tan(x_0))+k\pi$$
for all $x_0 \in(-\frac \pi 2+k\pi,\frac \pi 2 +k\pi)$. For all $t>0$ the map $x_0\mapsto X_t^{x_0}$ is almost surely not entire, but real analytic with a positive radius of convergence depending on the Brownian path $B$. The expectation of this quantity is, however, not real analytic at {$x_0=0$ where $a(x_0)=0$,} since otherwise also its derivative would be real analytic, which is the expectation of
$$
\frac{\partial}{\partial x_0} X_t^{x_0} = \frac{\exp(2 B_t)}{1 + (\exp(4B_t)-1) \sin^2(x_0)} \, \text{ for } t > 0 \text{ and } x_0\in(-\frac \pi 2+k\pi,\frac \pi 2 +k\pi).
$$
Indeed the functional series, which represents the flow's derivative for small $x_0$
$$
\sum_{j=0}^\infty \exp(2B_t) {(-1)}^j{\big( \exp(4B_t) -1 \big)}^j \sin^{2j}(x_0)
$$
is absolutely minorized by
$
\sum_{j=0}^\infty | \exp(2B_t) {(-1)}^j{( \exp(4B_t) -1 )}^j | {(\lambda x_0)}^{2j}
$
for some $0 < \lambda < 1$ and $x_0$ in an interval around $0$ such that $ (\lambda x_0)^2 \leq (\sin(x_0))^2 $. The expectation of the minorizing series, however, has a radius of convergence $0$, since the highest order term diverges to infinity like ${ \exp(4tj+2t) }$ as $ j \to \infty$.

Even though real analyticity of the expected value fails for the identity map, there is a real-analytic family of functions on $(-\pi/2, \pi/2)$ (and similarly for every shift of $k\pi$), dense in $C(K)$ for compact sets $K \subset (-\pi/2, \pi/2)$ where it works. Indeed,
\[
\text{span}\{ x \mapsto \tan^j(x) \, | \, j=0,1, \ldots\}
\]
is by the Stone--Weierstrass Theorem dense in $C(K)$ and since
\[ \mathbb{E}[\tan^j(X_t^{x_0})]=\tan^j(x_0)\mathbb{E}[\exp(2j B_t)]=\tan^j(x_0)\exp(2j^2t), \]
the real analyticity of the expected values is clear.

{
In general, note that  at the degeneracy points of the diffusion coefficient, 
we can still construct -- by means of the radius of convergence, which depends measurably on the Brownian path -- an increasing sequence of stopping times $\tau_n \nearrow \infty$ such that $x\mapsto \mathbb{E}_x[f(X_{t \wedge \tau_n})]$ is real analytic 
(at least if the $\omega$-dependent coefficients of the power series of $f(X_{t \wedge \tau_n})$ are sufficiently integrable). On the one hand this yields a sequence of real analytic functions converging locally uniformly to 
$x \mapsto \mathbb{E}_x[f(X_{t})]$.
On the other hand, approximate affine or polynomial structures exist for the stopped processes by Dynkin's formula. Indeed, as the generator $\mathcal{A}$ maps real analytic functions to real analytic ones, for each $n \in \mathbb{N}$, we obtain from
\begin{align*}
\mathbb{E}_x[f(X_{t\wedge \tau_n})] &= f(x) + \mathbb{E}_x\big[ \int_0^{t \wedge \tau_n} \mathcal{A} f (X_s) ds \big] \\&= f(x)+\mathbb{E}_x\big[ \int_0^{t} \mathcal{A} f (X_{s \wedge \tau_n}) ds \big] - \mathbb{E}_x\big[ \mathcal{A} f (X_{\tau_n}) (t- (\tau_n \wedge t)) \big]
\\&= f(x)+ \int_0^{t} \mathcal{A} \Big ( \mathbb{E}_{\cdot} \big[   f (X_{s \wedge \tau_n})  \big] \Big)(x) ds -  \mathbb{E}_x\big[ \mathcal{A} f (X_{\tau_n}) (t- (\tau_n \wedge t)) \big]
\end{align*}
(under the assumption that we can exchange $\mathcal{A}$ and $\mathbb{E}_x$ and the time integral) an inhomogeneous infinite dimensional linear ODE
\[
\partial_t \mathbf{c}(n,t)= L(\mathbf{c}(n,t))   -  \mathbb{E}_x\big[ \mathcal{A} f (X_{\tau_n}) (t- (\tau_n \wedge t)) \big] , \quad  \mathbf{c}(n,0)=\mathbf{u},
\]
where $\mathbf{c}(n,t)$ denotes the coefficients of the real analytic map $x \mapsto \mathbb{E}_x[f(X_{t\wedge \tau_n})]$ and
$L$  the linear operator corresponding to the action of the generator $\mathcal{A}$ on real-analytic maps. 
Note that the inhomogeneity $\mathbb{E}_x\big[ \mathcal{A} f (X_{\tau_n}) (t- (\tau_n \wedge t) )\big] $ tends to $0$ as 
 $n \to \infty$.}

\section{Numerical illustrations}\label{sec:numerics}

We consider now several concrete examples for which we solve the corresponding Riccati, transport and/or linear ODEs numerically. 
The corresponding code is available at \href{GitHub-Repo}{https://github.com/sarasvaluto/AffPolySig}.

Recall that we can consider two different approaches. The first one follows Section~\ref{sec4} and is based on the signature process
$\mathbb X_t^{sig}:=(1,X_t,\frac{X_t^2}{2!},\ldots)$. The second one follows  Section~\ref{sec5} and is based  on the power sequence expansion 
$\mathbb X_t^{pow}:=(1,X_t,X_t^2,\ldots)$.

\subsection{Riccati ODE for a one dimensional Brownian motion}\label{sec:Riccatinum}

Usually, the infinite dimensional Riccati ODE given by \eqref{eqn7aff} cannot be solved analytically and one needs to resort to numerical schemes. We start here with the simple setup of a one dimensional Brownian motion $X$ and the corresponding signature process $\X^{sig}$ or power expansion $\X^{pow}$. 
As a first numerical procedure we consider Scheme \ref{sch1}  to obtain a representation of $\E[\exp(\langle \u^{sig},\X_T^{sig}\rangle)]$ and $\E[\exp(\langle \u^{pow},\X_T^{pow}\rangle)]$. Of course other options are possible, starting from the choice of the truncation in the second item of the scheme.

 Recall that in the case of Brownian motion the Riccati operators  corresponding to the two settings $\X^{sig}$ and $\X^{pow}$ are given by
\begin{align*}
R^{sig}(\u)_k &= \frac{1}{2} \Big(\u_{k+2} + \sum_{i+j=k}\binom k i \u_{i+1}\u_{j+1}\Big),\\
R^{pow}(\u)_k &= \frac{1}{2} \Big((k+1)(k+2)\u_{k+2} + \sum_{i+j=k}(i+1)(j+1) \u_{i+1}\u_{j+1}\Big).
\end{align*}
\begin{scheme}\label{sch1}\phantomsection
\begin{enumerate}
\item Set a truncation level $K$.
    \item Define the truncated operators $R^{sig,K}$ and $R^{pow,K}$ as
    $$R^{sig,K}(\u)_k:=R^{sig}(\u)_k1_{\{k\leq K\}}\qquad\text{and}\qquad R^{pow,K}(\u)_k:=R^{pow}(\u)_k1_{\{k\leq K\}}.$$
    Observe that $R^{sig,K}(\u)=R^{sig}(\u)^{\leq K}$ and $R^{pow,K}(\u)=R^{pow}(\u)^{\leq K}$.
    \item Consider the truncated initial conditions $\u^{sig,\leq K}$ and $\u^{pow,\leq K}$.
    \item Employ a standard ODE solver to solve the ODEs
   \begin{align*}
       \partial_t \bm\psi^{sig,K}(t)&=R^{sig,K}(\bm\psi^{sig,K}(t)),& \bm\psi^{sig,K}(0)&=\u^{sig,\leq K}\\
    \partial_t \bm\psi^{pow,K}(t)&=R^{pow,K}(\bm\psi^{pow,K}(t)),& \bm\psi^{pow,K}(0)&=\u^{pow,\leq K} 
   \end{align*}
    \item The expected representations are then given by
    \begin{align*}
        \E[\exp(\langle \u^{sig},\X_T^{sig}\rangle)]&\approx \exp(\bm\psi^{sig,K}(T)_\emptyset)\qquad\text{and}\\
        \E[\exp(\langle \u^{pow},\X_T^{pow}\rangle)]&\approx \exp(\bm\psi^{pow,K}(T)_\emptyset).
    \end{align*}
\end{enumerate} 
\end{scheme}

\begin{example}[Laplace transform of the geometric Brownian motion]
Let now $(Y_t)_{t\in[0,T]}$ be a geometric Brownian motion with parameters $\mu=1/2$ and $\sigma=1$ such that
$Y_t=Y_0e^{X_t}$
for a standard Brownian motion $(X_t)_{t\in[0,T]}$. The corresponding Laplace transform is given by
$\E[\exp(-cY_T)]=\E[\exp(-cY_0\exp(X_{T}))]$
and is finite if and only if $cY_0\geq 0$.

Since the map  $f(x)=\exp(-cY_0\exp(x))$ lies in $\Dcal$ for each $c,Y_0>0$ we have that $\u^{pow}_k:=-cY_0/k!$ lies in $\Wcal^{pow}$
as defined in Proposition~\ref{prop1}. This in particular implies that  
$\E[\exp(-cY_T)]=\exp(\bm \psi(T)_\emptyset)$
where ${\bm \psi}$ is a $\Wcal^{pow}$-valued solution  of the Riccati ODE 
\begin{equation}\label{eqn16}
\partial_s{\bm \psi}(s)={R^{pow}({\bm \psi}(s))},\qquad {\bm \psi}(0)=\u^{pow}.
\end{equation}
Comparing the results of the numerical scheme proposed above with a Monte Carlo approximation for $c=T=Y_0=1$ we obtain the results of Figure~\ref{fig1}.
The same procedure using the signature's extension leads to the same result.
\begin{figure}[h!]
\centering
\includegraphics[width=0.7\textwidth]{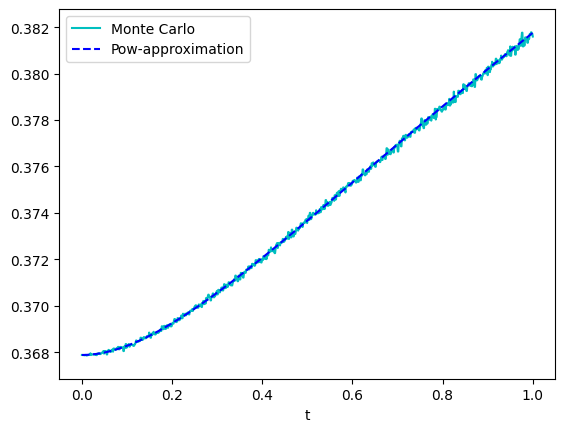}
\caption{$\E[\exp(-\exp(X_t))]$ computed with Scheme~\ref{sch1} for $K=20$.}\label{fig1}
\end{figure}
\end{example}

As we will see in Example~\ref{ex3} below, Scheme~\ref{sch1} may not always work. A more stable option is given by the following alternative scheme, based on the transport equation proposed in Theorem~\ref{thm:affinetrans} or Theorem~\ref{thm1} respectively.

For each $t\in[0,T]$ and each $\u$, such that the involved quantities are well defined, set
$v(t,\u):=\mathbb E\Big[\exp(\langle \u,\X_t^{sig}\rangle)\Big].$
By Theorem~\ref{thm:affinetrans} we expect that 
$$\partial_t v(t,\u)=\partial_\u^{R^{sig}(\u)}v(t,\u),\qquad v(0,\u)=\exp(\u_\emptyset)$$
where
$$\partial_\u^{R^{sig}(\u)}v(t,\u)=\lim_{M\to\infty}M\Big(v(t,\u+\frac {R^{sig}(\u)}M)-v(t,\u)\Big).$$
Discretizing both derivatives, we get 
$$\frac NT\Big(v(t+\frac T N,\u)-v(t,\u)\Big)
\approx \partial_t v(t,\u)
= \partial_u^{R^{sig}(u)}v(t,\u)
\approx M\Big(v(t,\u+\frac {R^{sig}(\u)}M)-v(t,\u)\Big),$$
which leads to
\begin{align*}
v(t,\u)
&\approx v(t-\frac T N,\u)+ \frac{MT}N\Big(v(t-\frac T N,\u+\frac {R^{sig}(\u)}M)-v(t-\frac T N,\u)\Big)\\
&=\lambda(v(t-\frac T N,A_M^{sig}(\u))+ (1-\lambda)v(t-\frac T N,\u),
\end{align*}
for $\lambda:=MT/N$ and $A_M^{sig}(\u):=\u+{R^{sig}(\u)}/M$.
Starting with $t=Tn/N$ and repeating the procedure $n$ times we thus obtain
\begin{align*}
v(Tn/N,\u)
&\approx\sum_{m=0}^n \binom n m(1-\lambda)^{n-m}\lambda^mv(0,(A_M^{sig})^{\circ m}(\u))\\
&=\sum_{m=0}^n \binom n m(1-\lambda)^{n-m}\lambda^m \exp((A_M^{sig})^{\circ m}(\u)_\emptyset),
\end{align*}
where $\circ$ denotes the composition.
Observe that the longest composition of the Riccati operator has length $n$. Due to the particular form of $R$, this implies that this approximation does not depend on terms of $\u$ of order higher than $2n$. The resulting scheme is the following.

\begin{scheme}\label{sch2}
\begin{enumerate}
\item Fix $N,M$ and fix the truncation level to $K=2N$. Set $\lambda=MT/N$.
    \item Define the truncated operators $R^{sig,K}$ and $R^{pow,K}$ as
    $$R^{sig,K}(\u)_k:=R^{sig}(\u)_k1_{\{k\leq K\}}\qquad\text{and}\qquad R^{pow,K}(\u)_k:=R^{pow}(\u)_k1_{\{k\leq K\}}.$$
    Set $A_M^{sig,K}(\u):=\u+\frac {R^{sig,K}(\u)}M$ and $A_M^{pow,K}(\u):=\u+\frac {R^{pow,K}(\u)}M$.
    \item Consider the truncated initial conditions $\u^{sig,\leq K}$ and $\u^{pow,\leq K}$.
    \item The expected representations are then given by
    \begin{align*}
        \E[\exp(\langle \u^{sig},\X_{\frac {Tn}N}^{sig}\rangle)]&\approx \sum_{m=0}^n \binom n m(1-\lambda)^{n-m}\lambda^m \exp((A_M^{sig,K})^{\circ m}(\u^{sig})_\emptyset)
        \qquad\text{and}\\
        \E[\exp(\langle \u^{pow},\X_{\frac {Tn}N}^{pow}\rangle)]&\approx \sum_{m=0}^n \binom n m(1-\lambda)^{n-m}\lambda^m \exp((A_M^{pow,K})^{\circ m}(\u^{pow})_\emptyset).
    \end{align*}
\end{enumerate} 
\end{scheme}

\begin{example}\label{ex3}
    A second challenging example is given by $\E[\exp(-\frac{X_t^4}{4!})]$ for a standard Brownian motion $X$. Also in this case, since $f(x):=\exp(-x^4/4!)$ lies in $\Dcal$, we can apply Proposition~\ref{prop1} to get that 
$\E[\exp(-\frac{X_t^4}{4!})]=\exp(\bm \psi(T)_\emptyset),$
where ${\bm \psi}$ is $\Wcal^{sig}$-valued solution  of the Riccati ODE
\begin{equation}\label{eqn16b}
\partial_s{\bm \psi}(s)={R^{sig}({\bm \psi}(s))},\qquad {\bm \psi}(0)=-e_1^{\otimes 4}.
\end{equation}
    In this case, Scheme~\ref{sch1} does not produce the desired output. Indeed, as one can see from Figure~\ref{fig2}, increasing the truncation level $K$ leads to an earlier explosion of the approximated solution.
\begin{figure}[h!]
\centering
\includegraphics[width=0.7\textwidth]{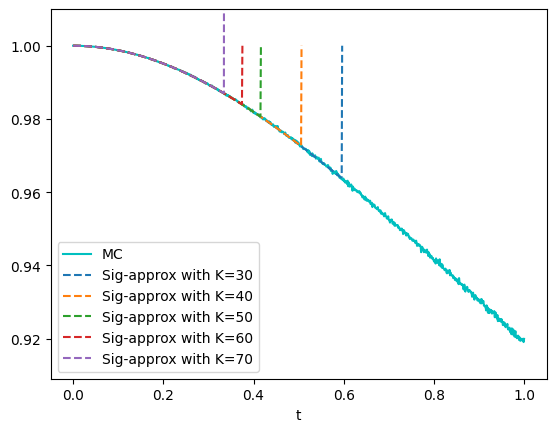}
\caption{$\E[\exp(-X_t^4/4!)]$ computed with Scheme~\ref{sch1} for different values of $K$. The approximations $\E[\exp(-X_t^4/4!)]$ are represented only until their explosion time.}\label{fig2}
\end{figure}
Moreover, since  $\E[\exp(X_t^4)]=\infty$, it follows that
    $\lim_{N\to\infty}\E[{\bf c}^N(T)_\emptyset]=
    \lim_{N\to\infty}\E[\sum_{k=0}^N\frac 1 {k!} (-X_T^{4})^k]=\infty,$
    where ${\bf c}^N$ is a solution  of the linear equation 
    $\partial_s{\bf c}^N(s)={L^{sig}({\bf c}^N(s))}$, for the initial condition ${\bf c}^N(0)=\exp(\shuffle(-e_1^{\otimes 4}))^{\leq N}.$
    This implies that a (naive) scheme implementing the moment formula  is also  problematic.
Scheme~\ref{sch2} based on the transport equation works significantly better than the Riccati approximation and we get the result illustrated in Figure~\ref{fig3}. Observe that the explosion time of the approximated solution increases with $M$. 
{For an alternative  scheme that works for this specific example we refer to \cite{jaber2024fourier},  Section 4.1.}
  \begin{figure}[h!]
\centering
\includegraphics[width=0.7\textwidth]{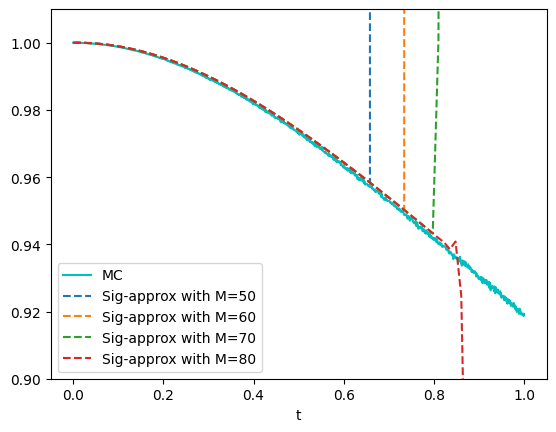}
\caption{$\E[\exp(-X_t^4/4!)]$ computed with Scheme~\ref{sch2} for $N=80$, $K=160$, and different values of $M$. The approximations are just represented until their explosion time.}\label{fig3}
\end{figure}  
\end{example}

\subsection{Examples on $[0,1]$}\label{ex1}
We 
now propose a numerical scheme for the moment formula in the spirit of Scheme~\ref{sch1}.
As before $L^{pow}$ denotes the operator corresponding to $\mathbb{X}^{pow}$.

\begin{scheme}\label{sch3}
\begin{enumerate}
\item Set a truncation level $K$.
    \item Define the truncated operator $L^{pow,K}$ as
    $L^{pow,K}(\u)_k:=L^{pow}(\u)_k1_{\{k\leq K\}}.$
    \item Consider the truncated initial condition $\u^{pow,\leq K}$.
    \item Compute the entries of the matrix $G$ satisfying
    $L^{pow,K}(e_1^{\otimes k})_j=G^{pow,K}_{kj}$. Observe that the solution ${\bf c}^{pow,K}$ of the linear ODE corresponding to $L^{pow,K}$ satisfies
    $${\bf c}^{pow,K}(t)=\exp(tG^{pow,K})\u^{pow,\leq K},$$
    where $\exp$ denotes the matrix exponential.
    \item The expected representations are  given by
$
       \E[\langle \u^{pow},\X_T^{pow}\rangle]\approx \sum_{n=1}^{\infty} {\bf c}^{pow,K}_n(T) x_0^n.
$
\end{enumerate} 
\end{scheme}

We exemplify this scheme by means of  the moment generating function of a Jacobi diffusion, see Figure \ref{fig4}.

  \begin{figure}[h!]
\centering
\includegraphics[width=0.7\textwidth]{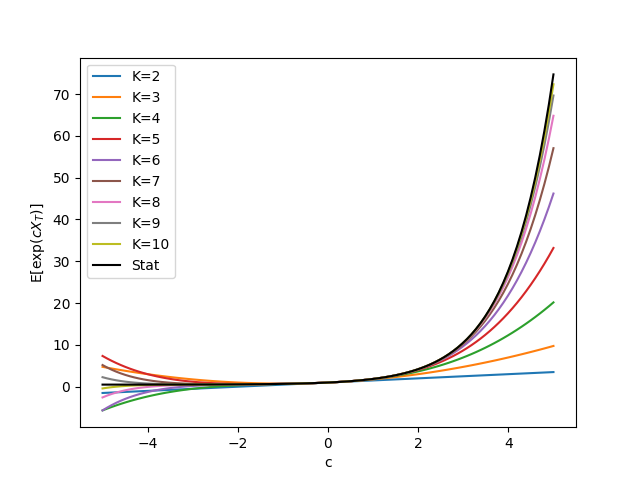}
\caption{$\E[\exp(cX_T)]$
for a Jacobi diffusion with $X_0=\frac 1 2$ and $T=1000$ computed with Scheme~\ref{sch3} for  different truncation levels $K$. The result is compared with the moment generating function of the stationary  measure $(\delta_0+\delta_1)/2$.}\label{fig4}
\end{figure}

\begin{appendix}
\section{Adaptation of Theorem 4.4.11 of \cite{EK:05}}
\begin{lemma}\label{lemA}
Consider a filtered probability space $(\Omega, \mathcal{F},(\mathcal{F}_t)_{t\in[0,T]}, \mathbb{P})$ and two measurable maps $X,Y:[0,T]\times [0,T]\times \Omega\to \R$. Assume that both
$$\E\Big[X(s,t)-X(0,t)-\int_0^s Y(r,t)dr\Big]=0\quad\text{and}\quad\E\Big[X(s,t)-X(s,0)-\int_0^t Y(s,r)dr\Big]=0$$
for each $s,t\in[0,T]$. If $\int_0^T\int_0^T \E[ |Y(s,t)|]dsdt<\infty$, then
 $\E[X(0,T)]=\E[X(T,0)].$
\end{lemma}
\begin{proof}
We follow the proof of Theorem 4.4.11 in \cite{EK:05}. Set $f(s,t):=\E[X(s,t)]$ and note that the martingale properties and Fubini yield
$$f(s,t)-f(0,t)-\int_0^s\E[Y(r,t)]dr=f(s,t)-f(s,0)-\int_0^t\E[Y(s,r)]dr=0,$$
for each $s,t\in[0,T]$. This shows that $f(s,t)$ is absolutely continuous in $s$ for each fixed $t$ and in $t$ for each fixed $s$. By Lemma~4.4.10 in the same book we can conclude that
$f(T,0)-f(0,T)=0$ and the result follows.
\end{proof}

\section{Proofs for Section \ref{sec:concrete}}

\subsection{Proof of Lemma~\ref{lem111}}\label{appB1}
\ref{itnew1}
Fix $f\in\Dcal$ and $t>0$ and note that $P_tf(x_0)=\frac 1 {\sqrt{2t\pi}}\int_{\R}  f(x)\exp(-\frac{(x-x_0)^2}{2t})dx$ reads
\begin{align*}
&\exp\Big(\frac{-x_0^2}{2t}\Big)\frac 1 {\sqrt{2t\pi}}
\int_{\R}  f(x)
\exp\Big(\frac{-x^2}{2t}\Big)
\sum_{k=0}^\infty\frac 1 {k! }\Big(\frac{xx_0}{t}\Big)^kdx\\
&\qquad=\exp\Big(\frac{-x_0^2}{2t}\Big)\frac 1 {\sqrt{2t\pi}}
\sum_{k=0}^\infty\frac 1 {k! }
\Big(\int_{\R}  f(x)
\exp\Big(\frac{-x^2}{2t}\Big)
\Big(\frac x t\Big)^kdx\Big)x_0^k,
\end{align*}
where in the last step we used Fubini justified by the fact that
$$\int_{\R}  \sum_{k=0}^\infty |f(x)|
\exp\Big(\frac{-x^2}{2t}\Big)
\frac 1 {k! }\Big(\frac{|xx_0|}{t}\Big)^kdx
\leq
\int_{\R}
\exp\Big(\frac{-x^2}{2t}+a(|x|+1)+\frac{|xx_0|}{t}\Big)dx<\infty,
$$
for some $a\in \R$.
Since the same bound shows that $P_tf(x_0)<\infty$ for each $x_0\in \R$, we can conclude that $P_tf$ is an entire map. 
To show that $P_tf\in \Bcal$, observe that
$$|P_tf(x_0)| 
\leq 
P_t\exp(a(|x_0+\fdot|+1))
\leq \exp(a|x_0|)C_t
$$
where $C_t$ is the finite constant given by $C_t:=P_t\exp(a(|\cdot|+1))$.
Finally, by the Leibniz rule we have that $(P_tf)'=P_tf'$ and $(P_tf)''=P_tf''$ implying that $(P_tf)',(P_tf)''\in \Bcal$ and thus $P_tf\in \Dcal$.

\ref{itnew2}
The continuity at $t>0$ is clear. For the continuity in 0,
fix $f\in \Bcal$ and note that it is always possible to find some $f_0\in C_0(\R)$ and some $f_c\in C(\R)$ such that $f=f_0+f_c$ and   $f_c(x)=0$ for each $x\in B_\e(x_0)$ for some $\e>0$. Since $\Ncal(x_0,t)$ converges weakly to the Dirac measure in $x_0$ for $t$ going to 0, we already know that $\lim_{t\to0}P_tf_0(x_0)=f_0(x_0)=f(x_0)$. Since by the dominated convergence theorem we also have
$$\lim_{t\to0}P_tf_c(x_0)=\lim_{t\to0}\frac 1 {\sqrt{2t\pi}}\int_{\R\setminus B_\e(x_0)}  f_c(x)\exp\Big(-\frac{(x-x_0)^2}{2t}\Big)dx
=0,$$
we can conclude that $\lim_{t\to0}P_tf(x_0)=P_0f(x_0)$.

\ref{itnew3} The claim follows as there is an $a\in \R$ such that $\int_0^t P_s|g|(x_0) ds$ is bounded by
\begin{align*}
 \int_0^t P_s\exp(a(|\fdot|+1))(x_0) ds
\leq e^a\int_0^t P_s\exp(a(\fdot))(x_0)+ P_s\exp(-a(\fdot))(x_0)ds
<\infty.
\end{align*}

\ref{itnew4}
Observe that for each $f\in \Dcal$ by the It\^o formula the process $M^f$ given by
$$M^f_t:=f(X_t)-f(x_0)-\int_0^t \frac 1 2 f''(X_s) ds=\int_0^t f'(X_s)dX_s,$$
is a local martingale. 
By definition of $\Bcal$ we have that $g^2\in \Bcal$ for each $g\in \Bcal$. This in particular implies that $(f')^2\in \Bcal$ and using \ref{itnew3} we can thus compute
$$\E[|\sup_{t\in [0,T]}M^f_t|^2]
\leq C\int_0^T P_s(|f'|^2)(x_0)ds<\infty,
$$
showing that $M^f$ is a true martingale. As $\Acal f=\frac 1 2 f''\in \Bcal$, by \ref{itnew3} we know that $\int_0^t P_s|\Acal f|(x_0)ds$ is finite.
Taking expectations of $M_t^f$ and applying Fubini we thus get that
$P_tf(x_0)-f(x_0)=\int_0^t P_s\Acal f(x_0)ds$
for each $f\in \Dcal$.
To conclude, observe that by  \ref{itnew2} (together with the fundamental theorem of calculus) yields 
$\partial_s P_{s}f(x_0)
= P_s\Acal f(x_0)$ which we have already shown to coincide with $\Acal P_sf(x_0)$.
\endproof

\subsection{Proof of Lemma \ref{lemaffine}} \label{appB2}
\ref{itc1}
Note first that the exponential moment condition on $X$ implies that the affine transform formula, i.e., 
$$
\E[\exp((\varepsilon+iu)X_t)]=\exp(\phi(t,\varepsilon+iu)+\psi(t,\varepsilon+iu)x_0),$$
 holds for all $u\in [-\tau,\tau]$ and $t\in [0,T]$, which follows from Theorem 3.7 in \cite{spreij2010affine}.
By Fubini's theorem we have that
$$P_th_\u(x)=\E_x\Big[\int_{-\tau}^\tau \exp((\varepsilon+iu)X_t)g(u) du\Big]
=\int_{-\tau}^\tau \exp(\phi(t,\varepsilon+iu)+\psi(t,\varepsilon+iu)x)g(u) du.$$
We now apply Morera's theorem and show that
$
\int_{\triangle} P_t h_{\u}(x) dx=0
$
for any triangle $\triangle$ in $\mathbb{C}$, which implies that $P_t h_{\u}$
is entire. First note that 
$ x \mapsto P_th_{\u}(x)$ can be extended to $\mathbb{C}$ since this is the case for 
$x\mapsto \exp(\phi(t,\varepsilon+iu)+\psi(t,\varepsilon+iu)x)g(u)$.
Again, by Fubini's theorem
we have
\[
 P_t h_{\u}(x) dx= \int_{-\tau}^\tau \int_{\triangle} \exp(\phi(t,\varepsilon+iu)+\psi(t,\varepsilon+iu)x) dx g(u) du=0,
\]
where the last equality follows from Goursat's theorem and the fact that $ x \mapsto \exp(\phi(t,\varepsilon+iu)+\psi(t,\varepsilon+iu)x) $ is entire.

\ref{itc2} For this property we only need that $h_{\u} \in \mathcal{D}$ which implies that $h_{\u}, h'_{\u}, h''_{\u}$ can be bounded by some polynomial of some degree $n$.  As diffusion and drift coefficient of affine diffusions are affine, this yields  that $\Acal h_{\u}\in \Ecal$ and
the assertion then follows by the polynomial property of the affine process $X$ (see e.g., Lemma 2.17 \cite{CKT:12}).

\ref{itc3} Here, we need again only that $h_{\u} \in \mathcal{D}$. 
By dominated convergence it holds that
\[
P_t \mathcal{A}h_{\u}(x_0)=\mathbb{E}[\lim_{N \to \infty} -N \vee (\mathcal{A}h_{\u} \wedge N)]=\lim_{N \to \infty} \mathbb{E}[-N \vee (\mathcal{A}h_{\u} \wedge N)]
\]
and by the Feller property of affine processes (see e.g.,  Theorem 2.7 in \cite{DFS:03})  $x_0 \mapsto \mathbb{E}[-N \vee (\mathcal{A}h_{\u} \wedge N)|X_0=x_0] $
is continuous. 
Now observe that
$$|P_t\Acal h_\u(x_0)-\E[-N \vee (\Acal h_\u(X_t)\wedge N]|
\leq 
\E[|\Acal h_\u(X_t)|1_{\{ |\Acal h_\u(X_t)|>N\}}],$$
which by Cauchy-Schwarz and Markov's inequality can be bounded by
\begin{align*}
    &\E[|\Acal h_\u(X_t)|^2 ]^{1/2}\mathbb P(\{|\Acal h_\u(X_t)|>N\} )^{1/2}
    \leq 
\E[|\Acal h_\u(X_t)|^2 ]^{1/2}
\frac{\E_x[|\Acal h_\u(X_t)|]^{1/2}}{N^{1/2}}.
\end{align*}
By the polynomial property of $X$, proceeding as in \ref{itc2} we get that the convergence is uniform on compacts and thus $x \mapsto P_t\Acal h_\u(x)$ is continuous. 

\ref{itc4} Similarly, again by the polynomial property of $X$ we get that for each $x\in S$, $P_t|\Acal h_\u|(x)$ is bounded by a continuous function in $t$ and thus $\int_0^TP_t|\Acal h_\u|(x) dt<\infty$ for each $x\in S$.  The same argument implies also that $\int_0^T\int_0^TP_tP_s|\Acal h_\u|(x)dsdt<\infty$. 
\endproof
\end{appendix}

\end{document}